\documentclass[a4paper, 11pt, reqno]{amsart}
\usepackage{amsmath, amsthm, amssymb, amstext}

\usepackage[left=2cm,right=2cm,top=2cm,bottom=2cm]{geometry}
\renewcommand{\Bbb}{\mathbb}
\usepackage{hyperref,xcolor}
\usepackage[hyperpageref]{backref}
\hypersetup{pdfborder={0 0 0},colorlinks}

\usepackage{enumitem}
\allowdisplaybreaks
\newtheorem{theorem}{Theorem}

\newtheorem{lemma}[theorem]{Lemma}
\newtheorem{proposition}[theorem]{Proposition}
\newtheorem{corollary}[theorem]{Corollary}
\newtheorem{definition}[theorem]{Definition}
\newtheorem{example}[theorem]{Example}

\numberwithin{theorem}{section}
\numberwithin{equation}{section}

\title[Sharpened Singular Adams-Type Inequalities]{A Note on Sharpened Singular Adams-Type Inequalities}
\author[D.K. Mahanta]{Deepak Kumar Mahanta}
\address[D.K. Mahanta]{Department of Mathematics, Indian Institute of Technology Jodhpur, Rajasthan 342030, India}
\email{mahanta.1@iitj.ac.in}

\author[T. Mukherjee]{Tuhina Mukherjee}
\address[T. Mukherjee]{Department of Mathematics, Indian Institute of Technology Jodhpur, Rajasthan 342030, India}
\email{tuhina@iitj.ac.in}

\author[A. Sarkar]{Abhishek Sarkar$^*$}
\address[A. Sarkar]{Department of Mathematics, Indian Institute of Technology Jodhpur, Rajasthan 342030, India}
\email{abhisheks@iitj.ac.in}
\thanks{$^*$corresponding author}

\subjclass{35A23; 35B33; 35J30; 26D10}
\keywords{Adams' inequality, $(p,\frac{n}{2})$-biharmonic operator, Singular exponential growth, Mountain pass theorem}

\begin{document}
\begin{abstract}
We establish a sharp Adams-type inequality in higher-order function spaces with singular weights on $\mathbb{R}^n.$ A sharp singular concentration–compactness principle, improving Lions’ result, is also proved. The study distinguishes between critical and subcritical sharp singular Adams-type inequalities and shows their equivalence. Furthermore, we analyze the asymptotic behavior of the associated bounds and relate the suprema of the critical and subcritical cases. A new compact embedding, crucial to our analysis, is also derived. Moreover, as an application of these results, by employing the mountain pass theorem, we study the existence of nontrivial solutions to a class of nonhomogeneous quasilinear elliptic equations involving the $(p,\frac{n}{2})$-biharmonic operator with singular exponential growth.
\end{abstract}

\maketitle
\section{Introduction and main results}\label{sec1}
To deal with nonlinear elliptic PDEs, it is worth mentioning that several authors work with Sobolev spaces of different orders. Moreover, to understand the basic features of the Sobolev spaces, one sees that for any $m\in\mathbb{N}$ and $p\geq 1$, the space $W^{m,p}(\mathbb{R}^n)$ can be categorized in three different ways, notably 
 \begin{itemize}
 \item[\textnormal{(a)}] the {\it Gagliardo-Nirenberg-Sobolev} ({\it GNS}, in short) case: $mp<n$, 
 \item[\textnormal{(b)}] the {\it Sobolev borderline} case: $mp=n$, and 
 \item[\textnormal{(c)}] the {\it Morrey's} case: $mp>n$. \end{itemize}
 The GNS inequality asserts that the embedding $W^{m,p}(\mathbb{R}^n)\hookrightarrow L^r (\mathbb{R}^n)$ is continuous for all $r\in [p,\frac{np}{n-mp}]$ and $mp<n$. In this case, we can study the variational problems in Sobolev spaces $W^{m,p}(\mathbb{R}^n)$ with subcritical and critical polynomial growth, that is, the nonlinear term cannot exceed the polynomial of degree $\frac{np}{n-mp}$. In spite of this, the Sobolev borderline case is very special because $\frac{np}{n-mp}\to+\infty$ as $mp\to n$, which suggests that $W^{m,p}(\mathbb{R}^n)\subset L^\infty (\mathbb{R}^n)$. Unfortunately, this only holds when $m=p=n=1$ and one knows that it may fail for $n\geq 2$. In this instance, we note that every polynomial growth can be entertained and thus, one may ask for a function with another kind of maximal growth. The Moser-Trudinger and Adams' inequalities play key roles in handling such difficulties. In this direction, for any bounded smooth domain $\Omega\subset \mathbb{R}^n~(n\geq 2)$ with $p=n$ and $m=1$, Trudinger \cite{Trudinger-1967} established that the embedding $W^{1,n}_0(\Omega)\hookrightarrow L_{\Phi_\alpha}(\Omega)$ is valid, where $L_{\Phi_\alpha}(\Omega)$ denotes the Orlicz space associated with the Young function $\Phi_\alpha(t)=\operatorname{\exp}(\alpha|t|^{\frac{n}{n-1}})-1$ for some constant $\alpha>0$.
 The sharpness of the exponent $\alpha$ was first studied by Moser \cite{Moser-1970/71} and proved a version of the Moser-Trudinger inequality as follows 
 \begin{align}\label{eq1.3}
    \sup_{u\in W^{1,n}_0(\Omega),~\|\nabla u\|_{L^n(\Omega)}~\leq 1}~\int_{\Omega} \operatorname{\exp}(\alpha|u|^{\frac{n}{n-1}})\,\mathrm{d}x\leq C(n)|\Omega|,~\forall~~\alpha\leq \alpha_n, 
 \end{align}
 where $\alpha_n=n\omega_{n-1}^{\frac{1}{n-1}}$ with $\omega_{n-1}$ is the $(n-1)$-dimensional surface measure of the unit sphere $\mathbb{S}^{n-1}$ in $\mathbb{R}^n$. In addition, the above supremum is no longer finite if $\alpha> \alpha_n$. Later, Adimurthi-Sandeep \cite{Adimurthi-Sandeep-2007} generalized the result obtained by Moser in \cite{Moser-1970/71} and established a sharp singular Moser-Trudinger inequality with the same Dirichlet norm used in \eqref{eq1.3}. Both nonsingular and singular Moser-Trudinger inequalities in the Euclidean space have been studied in \cite{Cao-1992, Ruf-2005} in $\mathbb{R}^2$, then for higher dimensions in \cite{Adachi-Tanaka-2000, Adimurthi-Yang-2010, de Souza-2011, do O-1997, do O-de Souza-2015, Li-Ruf-2008}. Such inequalities have several applications in conformal geometry and geometric analysis of PDEs; for a detailed study, we refer to \cite{Chang-Yang-2003, Beckner-1993, Lam-Lu-2012@} and the references therein.

 Due to the loss of compactness in the embedding $W^{1,n}_0(\Omega)\hookrightarrow L_{\Phi_\alpha}(\Omega)$, the main difficulty that arises in studying variational problems is to prove the Palais-Smale compactness condition. To avoid this, Lions' \cite{Lions-1985} introduced a concentration-compactness principle, basically a generalization of the Moser-Trudinger inequality. It states that if $u_k\rightharpoonup u$ weakly in $W^{1,n}_0(\Omega)$ with $u\neq 0$, $\|\nabla u_k\|_{L^n(\Omega)}\leq 1$, $|\nabla u_k|^n \rightharpoonup \mu$ weakly in $\mathcal{M}(\overline{\Omega})$, where $\mathcal{M}(\overline{\Omega})$ denotes the space of the Radon measure in $\mathbb{R}^n$, then 
 \begin{align}\label{eq1.4}
   \sup_{k\in\mathbb{N}} \int_{\Omega} \operatorname{\exp}(\alpha_n \wp|u_k|^{\frac{n}{n-1}})\,\mathrm{d}x\leq C(\wp,\Omega), ~\forall~~ \wp\in [1,\eta),   
 \end{align}
 where $\eta=\big(1-\|\nabla u^\ast \|^n_{L^n(\Omega)}\big)^{-\frac{1}{n-1}}$ and $u^\ast$ is the spherical symmetric decreasing rearrangement of $u$. Later, in \cite{Cerny-Cianchi-Hencl-2013} the authors sharpened the constant $\wp$ in \eqref{eq1.4} and established an improved version of Lions' concentration-compactness principle in $W^{1,n}_0(\Omega)$. The corresponding results in $\mathbb{R}^n$ can also be found in
 \cite{do O-de Souza-de Medeiros-Severo-2014, Zhang-Chen-2018}.

 To the best of our knowledge, there has been significant progress in the Moser-Trudinger inequalities to analyzing the existence and multiplicity of solutions to quasilinear elliptic equations involving the $n$-Laplace operator as well as the $(p,n)$-Laplace operator in the whole space $\mathbb{R}^n$. In this context, one may see \cite{Adimurthi-Yang-2010, Carvalho-Figueiredo-Furtado-Medeiros-2021, Lam-Lu-2012, Mahanta-Mukherjee-Sarkar-2025, Mahanta-Mukherjee-Winkert-2025, Mahanta-Winkert-2026, Yang-2012@} and the references therein. 

Let us introduce some notation to familiarize ourselves with higher-order Sobolev spaces. For any $m\in\mathbb{N}$ and $u\in C^m$, the class of $m$-th order differentiable function, we denote $\nabla^m$ as the $m$-$\text{th}$ order gradient operator, which is given by 
$$   \nabla^m u:=\begin{cases}
        \Delta^{\frac{m}{2}}u\quad \quad \quad\text{if}~m~\text{is even},\\
        \nabla \Delta^{\frac{m-1}{2}}u~\quad\text{if}~m~\text{is odd}.  \end{cases} $$
Moreover, for $m<n$, the Sobolev space with homogeneous Navier boundary condition on $\Omega\subset\mathbb{R}^n$ is denoted by $W^{m,\frac{n}{m}}_{\mathcal{N}}(\Omega)$ such that
$$
 W^{m,\frac{n}{m}}_{\mathcal{N}}(\Omega):=\left\{u\in   W^{m,\frac{n}{m}}(\Omega)\colon\,\Delta^j u_{|\partial\Omega}=0~\text{in the sense of trace,}~0\leq j\leq \left[\frac{m-1}{2}\right]\right\},  
$$
where $[\cdot]$ stands for the greatest integer function (see Tarsi \cite{Tarsi-2012}). In addition, we can notice that $W^{m,\frac{n}{m}}_0(\Omega)\subsetneq W^{m,\frac{n}{m}}_{\mathcal{N}}(\Omega)$ and the space $W^{2,\frac{n}{2}}_{\mathcal{N}}(\Omega)$ can be explicitly defined by
\begin{align*}
   W^{2,\frac{n}{2}}_{\mathcal{N}}(\Omega):=W_0^{1,\frac{n}{2}}(\Omega)\cap W^{2,\frac{n}{2}}(\Omega),~\forall~n>2.
\end{align*}

In the literature, it was Adams \cite{Adams-1988}, who first studied the Moser-Trudinger inequality that appears in \eqref{eq1.3} into higher-order Sobolev spaces. He established that for $m\in\mathbb{N}$ with $m<n$ and $\Omega $ is any bounded domain in $\mathbb{R}^n$, then
 \begin{align}\label{eq1.5}
    \sup_{u\in W_0^{m,\frac{n}{m}}(\Omega)_,~\|\nabla^m u\|_{L^\frac{n}{m}(\Omega)}~\leq 1}~\int_{\Omega} \operatorname{\exp}(\beta|u|^{\frac{n}{n-m}})\,\mathrm{d}x\leq C(n,m)|\Omega|,~\forall~~\beta\leq \beta(n,m), 
 \end{align}
 where
 $$ \beta(n,m):=\frac{n}{\omega_{n-1}}\begin{cases}
        \bigg[\frac{\pi^{\frac{n}{2}}2^m\Gamma\big(\frac{m+1}{2}\big)}{\Gamma\big(\frac{n-m+1}{2}\big)}\bigg]^{\frac{n}{n-m}}\quad\text{if}~m~\text{is odd},\\
        \bigg[\frac{\pi^{\frac{n}{2}}2^m\Gamma\big(\frac{m}{2}\big)}{\Gamma\big(\frac{n-m}{2}\big)}\bigg]^{\frac{n}{n-m}}\quad \quad \text{if}~m~\text{is even}.
         \end{cases} $$
Despite this, the constant $\beta(n,m)$ is optimal in the sense that the supremum \eqref{eq1.5} will become infinite if $\beta>\beta(n,m)$. Moreover, (thanks to \cite{Tarsi-2012}) Adams' inequality is also valid whenever $W^{m,\frac{n}{m}}(\Omega)$ is replaced by the larger space $ W^{m,\frac{n}{m}}_{\mathcal{N}}(\Omega)$. Subsequently, we remark here that the above Adams' inequality \eqref{eq1.5} has been extended in many directions. For example, in the bounded domain, the following inequality with singular weight has been proved in \cite[Lam-Lu]{Lam-Lu-2012@@@}, which states that
 \begin{align}\label{eq1.6}
    \sup_{u\in W_0^{m,\frac{n}{m}}(\Omega),~\|\nabla^m u\|_{L^\frac{n}{m}(\Omega)}~\leq 1}~\int_{\Omega} \frac{\operatorname{\exp}(\beta|u|^{\frac{n}{n-m}})}{|x|^\alpha}\,\mathrm{d}x<+\infty
 \end{align}
 for all $0\leq \beta\leq \beta_{\alpha,n,m}=(1-\frac{\alpha}{n})\beta(n,m)$ and $\alpha\in[0,n)$. In addition, $\beta_{\alpha,n,m}$ is sharp, that is, the supremum in \eqref{eq1.6} is infinite if $\beta>\beta_{\alpha,n,m}$. It was also noted that if $m$ is an even integer, then this inequality is also true for the Sobolev space $ W^{m,\frac{n}{m}}_{\mathcal{N}}(\Omega)$. Next, for any $u\in  W^{m,\frac{n}{m}}(\mathbb{R}^n)$, we define 
 $$
   \| u\|_{m,\frac{n}{m}}:=\begin{cases}
       \|(I-\Delta)^{\frac{m}{2}}u\|_{\frac{n}{m}}\quad \quad \quad \quad \quad \quad \quad \quad\quad \quad \quad \quad \quad \quad \qquad\text{if}~m~\text{is even},\\
       \Big( \|(I-\Delta)^{\frac{m-1}{2}}u\|^{\frac{n}{m}}_{\frac{n}{m}}+\|\nabla(I-\Delta)^{\frac{m-1}{2}}u\|^{\frac{n}{m}}_{\frac{n}{m}}\Big)^{\frac{m}{n}}\quad \quad\qquad\text{if}~m~\text{is odd},
   \end{cases}  
$$
where $\|\cdot\|_p$ denotes $L^p(\mathbb{R}^n)$ norm for $p\in[1,+\infty)$. Then the related inequalities corresponding to \eqref{eq1.5} when $\Omega=\mathbb{R}^n$ have been first proposed by \cite[Ruf-Sani]{Ruf-Sani-2013} for any positive even integer $m$. After that, it was again generalized in \cite[Lam-Lu]{Lam-Lu-2012@@} for any arbitrary integer $m\geq 1$. Their combined results can be stated as follows
  \begin{align}\label{eq1.7}
    \sup_{u\in W^{m,\frac{n}{m}}(\mathbb{R}^n),~\| u\|_{m,\frac{n}{m}}~\leq 1}~\int_{\mathbb{R}^n} \Phi(\beta(n,m)|u|^{\frac{n}{n-m}})\,\mathrm{d}x<+\infty,
 \end{align}
 where 
 $$ \Phi(t):=\operatorname{exp}(t)-\sum_{j=0}^{j_{\frac{n}{m}}-2}\frac{t^j}{j!}\quad\text{with}\quad j_{\frac{n}{m}}:=\min\Big\{j\in\mathbb{N}:~j\geq \frac{n}{m}\Big\}.$$
In addition, the constant $\beta(n,m)$ in the above supremum is sharp; that is, the supremum in \eqref{eq1.7} is infinite if we replace $\beta(n,m)$ by any $\beta>\beta(n,m)$. Further improvement of Adams' inequalities \eqref{eq1.7} has been studied in \cite{Zhang-Li-Chen-2020} and has shown that \eqref{eq1.7} still holds under the constraint
$$ \left\{u\in W^{m,\frac{n}{m}}(\mathbb{R}^n):~\|\nabla^m u\|^{\frac{n}{m}}_{\frac{n}{m}}+\tau \|u\|^{\frac{n}{m}}_{\frac{n}{m}}\leq 1~\text{for all}~\tau>0\right\}.$$
However, they did not sharpen it, and we believe it can be sharpened by using a similar strategy that used in \cite[Theorem 1.1]{Nguyen-2019}.

It was Yang \cite{Yang-2012} who first extended the results in \cite{Ruf-Sani-2013} and proved a version of the subcritical singular Adams' inequality in dimension four by introducing the Sobolev full norm. Basically, he established that 
 \begin{align}\label{eq1.8}
    \sup_{u\in W^{2,2}(\mathbb{R}^4),~\int_{\mathbb{R}^4}(|\Delta u|^2+\tau|\nabla u|^2+\sigma|u|^2)~\mathrm{d}x~\leq 1}~\int_{\mathbb{R}^4} \frac{\operatorname{\exp}(\alpha u^2)-1}{|x|^\beta}\,\mathrm{d}x<+\infty,~ \forall~\alpha\in\left(0,32\pi^2\left(1-\frac{\beta}{4}\right)\right)
 \end{align}
 for all constants $\tau,\sigma>0$ and $\beta\in[0,4)$. Furthermore, when $\alpha>32\pi^2\big(1-\frac{\beta}{4}\big)$, the supremum \eqref{eq1.8} is no longer finite. Note that it has remained an open question for the readers whether the above inequality \eqref{eq1.8} still holds for the critical case, that is, $\alpha=32\pi^2\big(1-\frac{\beta}{4}\big)$.  Then again, the singular Adams' inequality \eqref{eq1.6} was fully extended to $\mathbb{R}^n$ if the Dirichlet norm $\|\nabla^m \cdot\|_{L^\frac{n}{m}(\Omega)}$ is replaced by $\|(\tau I-\Delta)^{\frac{m}{2}}\cdot\|_{\frac{n}{m}}$ for any $\tau>0$ (see \cite{Lam-Lu-2013}). More precisely, they proved that 
 \begin{align}\label{eq1.9}
   \sup_{u\in W^{m,\frac{n}{m}}(\mathbb{R}^n),~\|(\tau I-\Delta)^{\frac{m}{2}}u\|_{\frac{n}{m}}~\leq 1}~\int_{\mathbb{R}^n} \frac{\Phi(\beta|u|^{\frac{n}{n-m}})}{|x|^\alpha}\,\mathrm{d}x<+\infty, ~\forall~\beta\leq \beta_{\alpha,n,m},
 \end{align}
for all $m\in\mathbb{N}$ with $m<n$ and $\alpha\in[0,n)$. Consequently, if $\beta>\beta_{\alpha,n,m}$, then the supremum \eqref{eq1.9} is no longer finite. Note that in the second-order Sobolev space $W^{2,\frac{n}{2}}(\mathbb{R}^n)$ with $n\geq 4$, the sharp singular Adams' inequality was investigated in \cite{Lam-Lu-2013} under the following constraint
\begin{align*}
 \left\{u\in W^{2,\frac{n}{2}}(\mathbb{R}^n):~\|\Delta u\|^{\frac{n}{2}}_{\frac{n}{2}}+\tau \|u\|^{\frac{n}{2}}_{\frac{n}{2}}\leq 1~\text{for all}~\tau>0\right\}.   
\end{align*}
Later, this result was again generalized in \cite[Theorem 1.1]{Zhang-Li-Chen-2020} into the arbitrary Sobolev space $ W^{m,\frac{n}{m}}(\mathbb{R}^n)$. Explicitly, they proved a strengthened version of the singular Adams' inequality as follows: let $\tau>0$, then 
 \begin{align}\label{eq1.10}
   \sup_{u\in W^{m,\frac{n}{m}}(\mathbb{R}^n),~\|\nabla^m u\|^{\frac{n}{m}}_{\frac{n}{m}}+\tau \|u\|^{\frac{n}{m}}_{\frac{n}{m}}~\leq 1}~\int_{\mathbb{R}^n} \frac{\Phi(\beta|u|^{\frac{n}{n-m}})}{|x|^\alpha}\,\mathrm{d}x<+\infty, ~\forall~\beta\leq  \beta_{\alpha,n,m},
 \end{align}
for all $m\in\mathbb{N}$ with $m<n$ and $\alpha\in[0,n)$. It has been observed that the inequality \eqref{eq1.10} was not sharpened yet. Moreover, the existence and nonexistence of extremal functions for sharp Adams' inequalities can be found in some recent articles (see \cite{Lu-Yang-2009, Chen-Lu-Zhu-2022, Chen-Lu-Zhu-2020, Zhang-Chen-2023}). 

In recent years, Adams' type inequalities have been widely studied by many authors across diverse domains such as Hyperbolic spaces, Lorentz spaces, CR spheres, compact Riemannian manifolds, and so on. In this context, we cite some delightful works captured in \cite{Karmakar-Sandeep-2016, Fontana-Morpurgo-2011, Fontana-1993} and the references therein. For applications of Adams' inequalities in the analysis of elliptic PDEs, we also refer to some notable articles \cite{do O-Macedo-2015, Chen-Li-Lu-Zhang-2018, Bao-Lam-Lu-2016}. 

It is worth mentioning that the concentration-compactness principle was first studied in higher-order Sobolev spaces in \cite[Theorem 1]{do O-Macedo-2014}. Later, on exploiting the Hilbert structure of $W^{2,2}(\mathbb{R}^4)$, the authors \cite{Chen-Li-Lu-Zhang-2018} established a sharp singular concentration-compactness principle and, as an application, they studied a class of biharmonic equations in dimension four. Then, it was fully extended to the entire space $\mathbb{R}^n$ in \cite{Nguyen-2019} and this result can be read as follows: let $u_k\rightharpoonup u$ weakly in $W^{m,\frac{n}{m}}(\mathbb{R}^n)$ with $u\neq 0$, $\|u_k\|_{m,\frac{n}{m}}\leq 1$ and 
$$1<\wp<Q_{n,m}(u):=\Big(1-\|u\|^{\frac{n}{m}}_{m,\frac{n}{m}}\Big)^{-\frac{m}{n-m}}\quad (~Q_{n,m}(u)=+\infty \quad\text{when}\quad\|u\|_{m,\frac{n}{m}}=1), $$ then it holds that
$$\sup_{k\in\mathbb{N}} \int_{\mathbb{R}^n} \Phi(\wp\beta(n,m)|u_k|^{\frac{n}{n-m}})\,\mathrm{d}x<+\infty. $$
In addition, the author also proved that the above result holds whenever $\|\cdot\|_{m,\frac{n}{m}}$ is replaced by the Ruf norm on $W^{m,\frac{n}{m}}(\mathbb{R}^n)$, which is defined by
$$ \|u\|_{\text{Ruf}}:=\bigg(\|u\|^{\frac{n}{m}}_{\frac{n}{m}}+\|\nabla^m u\|^{\frac{n}{m}}_{\frac{n}{m}}\bigg)^{\frac{m}{n}},~\forall~u\in W^{m,\frac{n}{m}}(\mathbb{R}^n). $$

Moreover, the asymptotic behavior of the supremum associated with subcritical sharp Moser-Trudinger inequalities in $\mathbb{R}^n$, under various functional settings, has been investigated in \cite{Lam-Lu-Zhang-2019@, Lam-Lu-Zhang-2019, Tang-2020}. These works also established the equivalence between the critical and subcritical sharp Moser-Trudinger inequalities. Related higher-order counterparts in $\mathbb{R}^n$ were obtained in \cite{Lam-Lu-Zhang-2017}.

Let $n\geq4$ and $1<p<\frac{n}{2}$ be hold. Next, we define our working function space $E$ as follows
$$ E:=\left\{u\in L^1_{\textit{loc}}(\mathbb{R}^n)\colon~\int_{\mathbb{R}^n}|\Delta u|^p\,\mathrm{d}x<+\infty,~\int_{\mathbb{R}^n}|\Delta u|^{\frac{n}{2}}\,\mathrm{d}x<+\infty\right\}, $$
 which is the completion of $C_0^\infty(\mathbb{R}^n)$ with respect to the following norm 
$$ \|u\|:=\Big(\|\Delta u\|^{\frac{n}{2}}_{\frac{n}{2}}+\|\Delta u\|^{\frac{n}{2}}_p\Big)^{\frac{2}{n}},~\forall~u\in E, $$
where $\|\cdot\|_s$ standards $L^s(\mathbb{R}^n)$ norm for $s\in[1,+\infty]$. In addition, one can observe that if we define another norm $$\|u\|_E:=\|\Delta u\|_{\frac{n}{2}}+\|\Delta u\|_p,~\forall~u\in E, $$ then $\|\cdot\|$ and $\|\cdot\|_E$ are equivalent norms for the space $E$. Indeed, we have the following relation
\begin{align}\label{eq1.99}
\|u\|\leq \|u\|_E\leq 2^{\frac{n-2}{n}}\|u\|,~\forall~u\in E.    
\end{align}
By using standard arguments, we remark that $(E,\|\cdot\|_E)$ is a uniform convex, reflexive, and separable Banach space. Moreover, since reflexivity is preserved under equivalent norms, we infer that $(E,\|\cdot\|)$ is also a reflexive Banach space. Next, we define the space $D^{2,p}(\mathbb{R}^n)$ as the completion of $C_0^\infty(\mathbb{R}^n)$ equipped with the norm 
$$ \|u\|_{D^{2,p}}:=\|\Delta u\|_p,~\forall~u\in D^{2,p}(\mathbb{R}^n). $$
Due to the Sobolev embedding, one can see that the chain of embeddings $E\hookrightarrow D^{2,p}(\mathbb{R}^n)\hookrightarrow L^{p^\ast}(\mathbb{R}^n)$ holds. Now, for any $\gamma\in [0,n)$, we define the space $L^r(\mathbb{R}^n,|x|^{-\gamma} \mathrm{d}x)$, consisting of real-valued measurable function $u$ satisfying $|u|^r |x|^{-\gamma}\in L^1(\mathbb{R}^n)$, equipped with the norm
$$ \|u\|_{r,\gamma}:=\bigg(\int_{\mathbb{R}^n} |u|^r |x|^{-\gamma}\,\mathrm{d}x\bigg)^{\frac{1}{r}}.$$
Throughout this article, we have the following notations: 
\begin{itemize}
\item $(E^*,\|\cdot\|_{*})$, $\langle{\cdot,\cdot}\rangle$ denotes the continuous dual of $(E,\|\cdot\|)$, 
\item $\langle \cdot, \cdot \rangle$ denotes the duality order pair between $E^\ast$ and $E$,
\item $o_k(1)$ denotes that the real sequence converges to zero as $k\to\infty$, 
\item $\rightharpoonup $ and $\rightarrow$ denote the weak convergence and the strong convergence,
\item $B_r(y)$ denotes an open ball centered at $y\in \mathbb{R}^n$ with radius $r>0$ and $B_r=B_r(0)$,
\item $f\lesssim ~( \gtrsim)~ g$, which means $f\leq ( \geq)~C g$ for some suitable constant $C>0$. We shall also write $f\sim g$ to denote that $f\lesssim  g$ and $f \gtrsim g$,
\item $f\underset{\approx}{<}g$ means $f$ is close enough to $g$, $f\thickapprox g$ stands for $f$ is approximately equal to function $g$ and $O(f)$ means some constant multiple of $f$,
\item $|A|$ represents $n$-dimensional Lebesgue measure for $A\subset\mathbb{R}^n$; and its complement by $A^c$. 
\end{itemize}
Briefly speaking, as our interest is to study a class of $(p,\frac{n}{2})$-biharmonic equations in the space $E$, it should be pointed out that, unfortunately, we cannot use the singular Adams' inequality that appeared in \cite[Theorem 1.5]{Lam-Lu-2013}. Thus, motivated by the above-mentioned works and primarily inspired by the works in \cite{Yang-2012, Lam-Lu-2013, Nguyen-2019, Chen-Li-Lu-Zhang-2018, Carvalho-Figueiredo-Furtado-Medeiros-2021, Aouaoui-2025, Lam-Lu-Zhang-2017}, we establish a new version of Adams' inequality in space $E$, its concentration-compactness principle, and as a byproduct of these inequalities, we consider a class of fourth-order elliptic PDEs involving $(p,\frac{n}{2})$-biharmonic operator in the entire space $\mathbb{R}^n$. For this, we first define the Young function $\Phi_{\alpha,j_0}:\mathbb{R}\to [0,+\infty)$  by
\begin{align}\label{eq1.55}
 \Phi_{\alpha,j_0}(s):=\operatorname{exp}\big(\alpha|s|^{\frac{n}{n-2}}\big)-\sum_{j=0}^{j_0-1} \frac{\alpha^j|s|^{\frac{nj}{n-2}}}{j!},~\forall~s\in\mathbb{R}\quad\text{and}\quad\alpha\in[0,+\infty)  \end{align}
with \begin{align*}j_0:=\left\lceil \frac{p^\ast (n-2)}{n} \right\rceil=\min\left\{j\in\mathbb{N}\colon\,j\geq \frac{p^\ast (n-2)}{n}\right\},    
\end{align*}
where $\left\lceil \cdot \right\rceil$ stands for the ceiling function. Furthermore, we emphasized that our work examines the more generalized singular Adams' type inequality, which includes the nonsingular situation, that is, $\gamma\in[0,n)$. It can be interpreted as follows.
\begin{theorem}[\textbf{Sharp Singular Adams' Type Inequality}] \label{thm1.1}
Let $n\geq4$, $1<p<\frac{n}{2}$ and $0\leq \gamma<n$ be hold. Then for all $0\leq\alpha\leq \beta_{\gamma,n}:=\big(1-\frac{\gamma}{n}\big)\beta(n,2)$ and $u\in E$, there holds
\begin{align}\label{eq1.12}
    \sup_{u\in E,\,\|u\|\leq 1}\int_{\mathbb{R}^n}\frac{\Phi_{\alpha,j_0}(u)}{|x|^\gamma}\,\mathrm{d}x<+\infty,
\end{align}
where $\Phi_{\alpha,j_0}(\cdot)$ is defined in \eqref{eq1.55} and $\beta(n,2)=n\big[(n-2)\omega^{\frac{2}{n}}_{n-1}\big]^{\frac{n}{n-2}}$ with $\omega_{n-1}$ stands for the $(n-1)$-dimensional surface measure of the unit sphere $\mathbb{S}^{n-1}$ in $\mathbb{R}^n$. Further, the constant $\beta_{\gamma,n}$ is sharp, that is, if $\alpha>\beta_{\gamma,n}$, then the supremum \eqref{eq1.12} is infinite.
\end{theorem}
Inspired by the seminal works of do \'{O}-Macedo \cite{do O-Macedo-2014} and Nguyen \cite{Nguyen-2019}, we derive the following sharp singular concentration-compactness principle, which improves upon the preceding singular Adams' inequality.
\begin{theorem}[\textbf{Sharp Singular Concentration-Compactness Principle}] \label{thm1.2}
Let $n\geq4$, $1<p<\frac{n}{2}$ and $0\leq \gamma<n$ be hold. If $\{u_k\}_k\subset E$ is a sequence such that $\|u_k\|=1$ and $u_k\rightharpoonup u\neq 0$ in $E$ as $k\to\infty$, then  for all $\ell$ satisfying $$ 0<\ell<L_n(u):= \big(1-\|u\|^{\frac{n}{2}}\big)^{-\frac{2}{n-2}} \quad\quad (~L_n(u)=+\infty \quad\text{whenever}\quad\|u\|=1),$$
we have 
\begin{align}\label{eq1.13}
\sup_{k\in\mathbb{N}}\int_{\mathbb{R}^n}\frac{\Phi_{\ell\beta_{\gamma,n},j_0}(u_k)}{|x|^\gamma}\,\mathrm{d}x<+\infty.
\end{align}
Finally, the constant $L_n(u)$ is sharp in the sense that the supremum \eqref{eq1.13} is infinite for $\ell> L_n(u)$.
\end{theorem}
The following natural question is still open at this time:
  \[ \text{ Is the above supremum defined as in \eqref{eq1.13} still true for}~ \ell=L_n(u)~?\]      

The next result says about the lower and upper bounds of the subcritical sharp singular Adams' type inequality in the space $E$ asymptotically. In fact, we illustrate the subsequent theorem.
\begin{theorem}[\textbf{Subcritical Sharp Singular Adams' Type Inequality}]\label{thm5.1} Let $n\geq4$, $1<p<\frac{n}{2}$, $0\leq \gamma<n$ and $0\leq \ell<\beta(n,2)$ be hold. Then there holds
$$\text{ATSC}(\ell,\gamma):=\sup_{u\in E,~\|\Delta u\|_{\frac{n}{2}}\leq 1}\cfrac{1}{\|\Delta u\|_p^{p^\ast(1-\frac{\gamma}{n})}}\int_{\mathbb{R}^n}\frac{\Phi_{\ell(1-\frac{\gamma}{n}),j_0}(u)}{|x|^\gamma}\,\mathrm{d}x<+\infty. $$
The constant $\beta(n,2)$ is sharp in the sense that $\text{ATSC}(\beta(n,2),\gamma)=+\infty$. Also, there exists two positive constants $c(\gamma,p,n)$ and $C(\gamma,p,n)$ such that whenever $l\underset{\approx}{<}\beta(n,2)$, there holds
$$ \cfrac{c(\gamma,p,n)}{\Bigg(1-\bigg(\frac{\ell}{\beta(n,2)}\bigg)^{\frac{n-2}{2}}\Bigg)^{\frac{2p^\ast}{n}\big(1-\frac{\gamma}{n}\big)}}\leq \text{ATSC}(\ell,\gamma)\leq \cfrac{C(\gamma,p,n)}{\Bigg(1-\bigg(\frac{\ell}{\beta(n,2)}\bigg)^{\frac{n-2}{2}}\Bigg)^{\frac{2p^\ast}{n}\big(1-\frac{\gamma}{n}\big)}}.$$    
\end{theorem}
In the following theorem, we first describe a necessary and sufficient condition for the boundedness of the critical sharp singular Adams' inequality in the space $E$. Finally, we shall demonstrate the equivalence of critical and subcritical sharp singular Adams’ inequalities.
\begin{theorem}[\textbf{Critical Sharp Singular Adams' Type Inequality}]\label{thm5.22} Let $n\geq4$, $1<p<\frac{n}{2}$, $0\leq \gamma<n$ and $a,b>0$ be hold. Then there holds
$$\text{ATC}_{a,b}(\gamma):=\sup_{u\in E,~\|\Delta u\|^a_{\frac{n}{2}}+\|\Delta u\|^b_p\leq 1}\int_{\mathbb{R}^n}\frac{\Phi_{\beta_{\gamma,n},j_0}(u)}{|x|^\gamma}\,\mathrm{d}x<+\infty\iff b\leq\frac{n}{2}. $$
The constant $\beta(n,2)$ is sharp in the sense that the above supremum is infinite when $\beta(n,2)$ is replaced by a larger constant. In addition, the following identity holds
\begin{align}\label{eq1.15}
  \text{ATC}_{a,b}(\gamma)=\sup_{\ell\in(0,\beta(n,2))}\Bigg(\cfrac{1-\big(\frac{\ell}{\beta(n,2)}\big)^{(\frac{n-2}{n})a}}{\big(\frac{\ell}{\beta(n,2)}\big)^{(\frac{n-2}{n})b}}\Bigg)^{\frac{p^\ast}{b}\big(1-\frac{\gamma}{n}\big)}\text{ATSC}(\ell,\gamma).  
\end{align}
Moreover, there hold $\text{ATC}_{\frac{n}{2},\frac{n}{2}}(\gamma)<+\infty$ and
$$\text{ATC}_{\frac{n}{2},\frac{n}{2}}(\gamma)=\sup_{\ell\in(0,\beta(n,2))}\Bigg(\cfrac{1-\big(\frac{\ell}{\beta(n,2)}\big)^{\frac{n-2}{2}}}{\big(\frac{\ell}{\beta(n,2)}\big)^{\frac{n-2}{2}}} \Bigg)^{\frac{2p^\ast}{n}\big(1-\frac{\gamma}{n}\big)}\text{ATSC}(\ell,\gamma). $$
\end{theorem}
Prior to our investigation, we remark that the compact embedding that is explained below is crucial in our analysis. The following is the interpretation of its statement.
\begin{theorem}\label{thm1.4}
Let $n\geq 4$, $1<p<\frac{n}{2}$ and $\gamma\in(0,n)$ be hold. Then the embedding $E\hookrightarrow L^\varrho(\mathbb{R}^n,|x|^{-\gamma}\mathrm{d}x)$ is compact for all $\varrho\geq p^\ast$, where $p^\ast:=\frac{np}{n-2p}$, the so-called critical Sobolev exponent.
\end{theorem}

Moreover, by the applications of Theorems \ref{thm1.1}, Theorem \ref{thm1.2} and Theorem \ref{thm1.4}, we investigate the existence of nontrivial solutions to a class of $(p,\frac{n}{2})$-biharmonic equations with singular exponential growth in $\mathbb{R}^n$. Specifically, we study the following equation  
\begin{equation}\label{main problem}
   \Delta^2_p u+\Delta^2_{\frac{n}{2}} u=\frac{g(x,u)}{|x|^\gamma}\quad\text{in}\quad \mathbb{R}^n,\tag{$\mathcal{P}$}
\end{equation}
with $n\geq 4$, $1<p<\frac{n}{2}$, $\gamma\in(0,n)$ and $\Delta^2_t~ \cdot:=\Delta(|\Delta \cdot|^{t-2}\Delta ~\cdot)$ is a fourth-order operator, which is known as the standard $t$-biharmonic operator for all $t > 1$. The nonlinearity $g:\mathbb{R}^n\times \mathbb{R}\to \mathbb{R}$ is a Carath\'eodory function and has critical exponential growth at infinity, that is, it behaves like $\exp{(\alpha|s|^{\frac{n}{n-2}})}$ as $|s|\to~+\infty$ for some $\alpha>0$, which means that there exists a positive constant $\alpha_0$ such that 
\begin{align*}
	\lim_{|s|\to +\infty} |g(x,s)| \operatorname{exp}(-\alpha|s|^{\frac{n}{n-2}})=
	\begin{cases}
		0       & \text{if } \alpha>\alpha_0, \\
		+\infty & \text{if }\alpha<\alpha_0, 
	\end{cases}
\end{align*}
uniformly with respect to $x\in \mathbb{R}^n$. This notion of criticality is basically driven by the well-known  Moser-Trudinger type inequality. Throughout this article, without further mention, we assume that $g:\mathbb{R}^n\times \mathbb{R}\to \mathbb{R}$ satisfies the following properties.
 \begin{enumerate}[label={($\bf g{\arabic*}$)}]
\setcounter{enumi}{0}
\item\label{g1} The map $g:\mathbb{R}^n\times \mathbb{R}\to \mathbb{R}$ is a Carath\'eodory function, the map $\cdot\mapsto g(x,\cdot)$ is an odd function and $g(x,0)=0$ for all $x\in \mathbb{R}^n$. In addition, we also have $g(x,s)>0$ for all $(x,s)\in \mathbb{R}^n\times (0,+\infty) $ and $g(x,s)<0$ for all $(x,s)\in \mathbb{R}^n\times (-\infty,0)$.
    \item\label{g2} There exists constant $\tau>\max\{p^\ast,\frac{n}{2}\}$ with $p^\ast:=\frac{np}{n-2p}$, so-called critical Sobolev exponent such that $$\lim_{s\to 0} \frac{|g(x,s)|}{\quad|s|^{\tau-1}}=0\quad\text{uniformly with respect to}~~ x\in \mathbb{R}^n.$$
    \item\label{g3} There exists constant $\alpha_0>0$ such that 
    \begin{align*}
	\lim_{|s|\to +\infty} \frac{|g(x,s)|}{\operatorname{exp}(\alpha|s|^{\frac{n}{n-2}})}=
	\begin{cases}
		0       & \text{if } \alpha>\alpha_0, \\
		+\infty & \text{if }\alpha<\alpha_0, 
	\end{cases}\quad\text{uniformly with respect to}~~ x\in \mathbb{R}^n.
\end{align*}
      \item\label{g4} There exists constant $\mu>\frac{n}{2}$ such that
      $$0<\mu G(x,s):=\mu\int_{0}^{s} g(x,t)\,\mathrm{d}t\leq g(x,s)s\quad\text{ for all}\quad(x,s)\in \mathbb{R}^n\times (\mathbb{R}\setminus\{0\}). $$
       \item\label{g5} There exists positive constants $s_0$ and $M_0$ such that for all $x\in \mathbb{R}^n$, we have 
       $$0< G(x,s)\leq M_0|g(x,s)|\quad\text{for all}\quad |s|\geq s_0. $$
       \item\label{g6} There exists constant $\vartheta>\max\{p^\ast,\frac{n}{2}\}$ and $\lambda>0$ such that 
       $$ G(x,s)\geq \lambda |s|^\vartheta\quad\text{ for all}\quad(x,s)\in\mathbb{R}^n\times \mathbb{R}.$$
       \end{enumerate}

In addition, one can notice that, due to \ref{g2} and \ref{g3}, for any fixed $q>\max\{p^\ast,\frac{n}{2}\}$ and $\alpha>\alpha_0$, there exists $\zeta>0$ and a constant $D_\zeta>0$ depending upon $\zeta,~\alpha$ and $q$ such that 
\begin{align}\label{eq1.1}
    |g(x,s)|\leq \zeta |s|^{\tau-1}+D_\zeta|s|^{q-1}\Phi_{\alpha,j_0}(s)\quad\text{for all}\quad (x,s)\in\mathbb{R}^n\times\mathbb{R}
\end{align}
and by using \ref{g4}, we have
\begin{align}\label{eq1.2}
    |G(x,s)|\leq \zeta |s|^{\tau}+D_\zeta|s|^{q}\Phi_{\alpha,j_0} (s)\quad\text{for all}\quad (x,s)\in\mathbb{R}^n\times\mathbb{R},
\end{align}
where $\Phi_{\alpha,j_0}(\cdot)$ is the Young function defined in \eqref{eq1.55}. The following function provides a concrete example that satisfies all the previously stated hypotheses.
\begin{example} The function $G\colon \mathbb{R}^n\times\mathbb{R}\to \mathbb{R}$ given by $G(x,s):=\lambda|s|^\vartheta \operatorname{exp}\big(\alpha_0|s|^{\frac{n}{n-2}}\big)$ for all $(x,s)\in\mathbb{R}^n\times\mathbb{R},$ where $n\geq 4$, $\lambda>0$, $0<\alpha_0 <\alpha$, $\vartheta>\tau,\mu>\max\{p^\ast,\frac{n}{2}\}$, $g(x,s)=\partial_s G(x,s)$ for all $(x,s)\in \mathbb{R}^n\times (\mathbb{R}\setminus\{0\})$, and $g(x,0)=0$ for all $x\in \mathbb{R}^n$, then $g$ satisfies all the hypotheses \ref{g1}--\ref{g6}. 
\end{example}
Next, we would like to emphasize the important aspects and novelties of this article below:
\begin{itemize}
    \item [(a)] In higher-order Sobolev spaces, the P\'olya–Szeg\"{o} types inequalities don't hold, which causes us to face several barriers in proving the singular Adams' type inequality in $E$ and the associated concentration-compactness principle, as can be shown in Theorem \ref{thm1.1} and Theorem \ref{thm1.2}.
    \item [(b)] Due to the critical exponential growth of the nonlinearity, there is a healthy competition between the nonhomogeneous $(p,\frac{n}{2})$-biharmonic operator and the nonlinear term because the embedding $E\hookrightarrow L_{\Phi_{\alpha,j_0}}(\mathbb{R}^n)$ is not compact for the non-singular case, that is, $\gamma=0$, for instance, one can see to Theorem \ref{thm1.1}, where $\Phi_{\alpha,j_0}(\cdot)$ is the Young function defined in \eqref{eq1.55}. Therefore, proving the compactness of Palais-Smale sequences for the variational energy associated with our main problem becomes a challenge in this situation. Note that a similar type of difficulty also arises for the singular case, that is, $\gamma\neq 0$. To avoid this difficulty, we shall discuss a concentration-compactness principle, that is, Theorem \ref{thm1.2}.
    \item [(c)] The study of the existence of solutions to our main problem relies on the compact embedding $ E \hookrightarrow L^\varrho(\mathbb{R}^n,|x|^{-\gamma} \mathrm{d}x)$ for all $\varrho\geq p^\ast$ and $\gamma\in(0,n)$, thanks to Theorem \ref{thm1.4}. It follows that our method is different from $\gamma=0$. However, we highlighted here that by employing the same method developed in this article and Lemma \ref{lem2.5}, one can study the non-singular case, that is, $\gamma=0$.
    \item [(d)] Owing to the fact that $E$ is not a Hilbert space, for any bounded  Palais-Smale sequence $\{u_k\}_k \subset E$ with $u_k\rightharpoonup u$ in $E$ as $k\to\infty$, we cannot directly get $$|\Delta u_k|^{t-2}\Delta u_k \rightharpoonup |\Delta u|^{t-2}\Delta u\quad\text{in}\quad L^{\frac{t}{t-1}}(\mathbb{R}^n)\quad\text{as}\quad k\to\infty$$ for $t\in\{p,\frac{n}{2}\}$. It indicates that more delicate analysis is required to prove $$\Delta u_k\to \Delta u \quad\text{a.e. in}\quad\mathbb{R}^n\quad\text{as}\quad k\to\infty.$$
     \item [(e)] It is well-known that for any $u\in E$, there is no guarantee that  $|u|,~u^{\pm}\in E$, where $u^{\pm}=\max\{\pm u,0\}$ and thus, we shall unable to get positive solutions to our main problem.
     \item[(f)] Sophisticated methods, such as variational and topological tools, are used in the proofs.
\end{itemize}

The following defines the weak formulation of our main problem. 
\begin{definition}
    We say that $u\in E$ is a weak solution to the problem \eqref{main problem}, if there holds
    $$\int_{\mathbb{R}^n} |\Delta u|^{p-2}\Delta u ~\Delta v~\mathrm{d}x+\int_{\mathbb{R}^n} |\Delta u|^{\frac{n}{2}-2}\Delta u ~\Delta v~\mathrm{d}x=\int_{\mathbb{R}^n} \frac{g(x,u)v}{|x|^\gamma}\,\mathrm{d}x,~\forall~v\in E.$$
\end{definition}
Now, we can state the main result of this article as follows.
\begin{theorem} \label{thm1.5}
Suppose that $1<p<\frac{n}{2}$, $\gamma\in(0,n)$ and $n\geq 4$. Let the hypotheses \ref{g1}--\ref{g6} be satisfied. Moreover, let there exists $\lambda_0$ be sufficiently large enough such that \ref{g6} holds for $\lambda\geq \lambda_0$, then the problem \eqref{main problem} has a nontrivial weak solution.
\end{theorem}

The structure of this article is organized in the following manner: Section \ref{sec2} is dedicated to some preliminary results, which are needed in our proofs.  In Section \ref{sec3}, we prove Theorem \ref{thm1.1} via a rearrangement-free approach, while in Section \ref{sec4}, we establish Theorem \ref{thm1.2} by means of Schwarz symmetrization. Sections \ref{sec5} and \ref{sec6} are devoted to the proofs of Theorems \ref{thm5.1} and \ref{thm5.22}, respectively. In Section \ref{sec7}, we establish a compact embedding result; for instance, see Theorem \ref{thm1.4}. Finally, in Section \ref{sec8}, we apply the preceding results together with the mountain pass theorem to prove Theorem \ref{thm1.5}.
\section{Preliminary results}\label{sec2}
In this section, we shall discuss the basic properties of Schwarz symmetrization and some useful lemmas. In this regard, we refer to \cite{Baernstein-2019, Bennett-Sharpley-1988}. First, we introduce the symmetrization of a Lebesgue measurable set $S$ in $\mathbb{R}^n$. Let $S^\ast$ be the open ball centered at the origin of $\mathbb{R}^n$ such that $|S^\ast|=|S|$. Further, we denote $|S|=0$ implies $S^\ast$ is the empty set, $|S|=+\infty$ implies $S^\ast=\mathbb{R}^n$ and $0<|S|<+\infty$ implies $S^\ast=B_R(0)$ such that $|S^\ast|=|S|=\sigma_n R^n$, where $\sigma_n$ is the volume of the unit ball in $\mathbb{R}^n$ and $R=R(S)$ is the volume radius of $S$. It is well-known that $S^\ast$ is called the symmetric decreasing rearrangement of $S$. Let $f\colon S\to\mathbb{R}$ be a measurable function vanishing at infinity, which means that
\begin{align*}
 |\{x\in S\colon~|f(x)|>t\}|=\int_{\{x\in S\colon~|f(x)|>t\}} \mathrm{d}x<+\infty,~\forall~t>0.   
\end{align*}
Its distribution function $\mu_f\colon[0,+\infty)\to [0,+\infty]$ is given by $\mu_f(t)=|\{x\in S\colon~|f(x)|>t\}|$ for all $t\geq 0$. Note that $\mu_f$ is always a nonnegative, decreasing and right continuous function on $[0,+\infty)$. Moreover, the unidimensional rearrangement of $f$ is a function $f^\sharp\colon[0,|S|]\to [0,+\infty]$, which is defined by $$f^\sharp(0):=\text{ess sup} f\quad\text{and}\quad f^\sharp(s):=\inf\{t\geq 0\colon~\mu_f(t)<s\},~\forall~s\in(0,|S|].$$ 
In fact, it has been observed by Bennett-Sharpley \cite{Bennett-Sharpley-1988} that the following identity also holds
$$f^\sharp(s)=m_{\mu_f}(s)=\sup\{t\geq 0:~\mu_f(t)>s\},~\forall~s\geq 0,$$
where $m_{\mu_f}$ is the distribution function of $\mu_f$. Besides this, $f^\sharp$ is a nonnegative, decreasing and right continuous function on $[0,+\infty)$. The maximal function of $f^\sharp$ is denoted by the symbol $f^{\sharp\sharp}$, and given by
\begin{align}\label{eq2.22}
  f^{\sharp\sharp}(s):=\frac{1}{s}\int_{0}^{s} f^{\sharp}(t)\,\mathrm{d}t,~\forall~~ s>0.  
\end{align}
In addition, $f^{\sharp\sharp}$ is nonnegative, decreasing and continuous on $(0,+\infty)$ and $f^\sharp\leq f^{\sharp\sharp}$. Now, we denote the function $f^\ast\colon S^\ast\to [0,+\infty]$ as the Schwarz symmetrization of $f$, alternatively known as the spherically symmetric decreasing rearrangement of $f$, and defined by 
$$ f^\ast(x):= f^\sharp(\sigma_n |x|^n),~\forall ~x\in S^\ast.$$
In particular, $f^\ast$ is a radially symmetric, nonnegative and nonincreasing function. Moreover, if $\Psi\colon\mathbb{R}\to \mathbb{R}$ is nondecreasing, then we have
\begin{equation}\label{eq2.1}
 (\Psi(f))^\ast=\Psi(f^\ast).   
\end{equation}
It has been observed that $f^\ast$ and $|f|$ are equimeasurable, that is, $\{x\colon~ f^\ast>t\}=\{x\colon~ |f|>t\}^\ast$ and $(\chi_A)^\ast=\chi_{A^\ast}$, where $\chi_A$ is the indicator function associate with measurable set $A\subset\mathbb{R}^n$. Moreover, if $f\colon\mathbb{R}^n\to \mathbb{R}$ is a measurable function vanishing at infinity, we can deduce from the Layer cake representation that
$$ f^\ast(x):=\int_{0}^{\infty} \chi_{\{x\colon~|f|>t\}^\ast }(x)\,\mathrm{d}t,~\forall~x\in\mathbb{R}^n.$$
In addition, we also have $\|f\|_p=\|f^\ast\|_p$ for all $p\in[1,+\infty)$ and the following Hardy-Littlewood inequality holds.
\begin{lemma} {\cite[Corollary 2.16]{Baernstein-2019}}\label{lem2.1}
 Suppose that $f$ and $g$ are two nonnegative measurable functions on $\mathbb{R}^n$ with $\mu_f(t)<+\infty$ and $\mu_g(t)<+\infty$ for all $t>0$. Then there holds
 \begin{equation}\label{eq2.2}
  \int_{\mathbb{R}^n} f(x)g(x)\,\mathrm{d}x\leq \int_{\mathbb{R}^n} f^\ast(x) g^\ast(x)\,\mathrm{d}x    
 \end{equation}
 with the understanding that when the left side is infinite, so is the right side. Moreover, if it is true that $\int_{\mathbb{R}^n} f(x)g(x)\,\mathrm{d}x<+\infty$, then the equality in \eqref{eq2.2} holds if and only if $|\{(x,y)\in\mathbb{R}^{2n}\colon~f(x)<(f(y)~\text{and}~g(x)>g(y)\}|=0.$
\end{lemma}
We highlight that the following lemmas are essential for studying the subsequent sections. 
\begin{lemma}{\cite[Lemma 2]{do O-Macedo-2014}}\label{lem2.2}
  Let $S\subset \mathbb{R}^n$ be an open set and $f_k,f\in L^\wp(S)$ such that $f_k\rightharpoonup f$ in $L^\wp(S)$ as $k\to\infty$ for any $\wp>1$. Then up to a subsequence $f_k^\sharp:=h_k \to h$  as $k\to\infty$ a.e. in $[0,|S|]$ for some $h\in L^\wp([0,|S|])$ with $\|h\|_{L^\wp([0,|S|])}\geq \|f^\sharp\|_{L^\wp([0,|S|])}$.
\end{lemma}
The following result is a crucial tool for proving Theorem \ref{thm1.2}. 
\begin{lemma}{\cite[Proposition 2.2]{Nguyen-2019}}\label{lem2.3}
   Let $u\in C^\infty_0(\mathbb{R}^n)$ and $-\Delta u=f\in C^\infty_0(\mathbb{R}^n)$, then there exists a constant $D(n)$ depends upon only $n$ such that
   \begin{align*}
       u^\sharp(s_1)-u^\sharp(s_2)\leq \frac{1}{n^2\sigma_n^\frac{2}{n}}\int_{s_1}^{s_2}\frac{f^{\sharp\sharp}(t)}{t^{1-\frac{2}{n}}}\,\mathrm{d}t+D(n)\|f\|_{\frac{n}{2}},~\forall~~ 0<s_1<s_2<+\infty.
   \end{align*}
\end{lemma}
\begin{lemma}{\cite[Lemma A.VI]{Berestycki-Lions-1983}}\label{lem2.4}
    If $u\in L^\wp(\mathbb{R}^n)$ for all $\wp\geq 1$ is a radial nonincreasing function, then there holds
    $$|u(x)|\leq |x|^{-\frac{n}{\wp}}\Big(\frac{n}{\omega_{n-1}}\Big)^{\frac{1}{\wp}}\|u\|_\wp,~\forall~x\in\mathbb{R}^n\setminus\{0\}. $$
\end{lemma}
\begin{lemma}{\cite[Proposition 1]{Aouaoui-2025}}\label{lem2.5}
 Let $1<p<\frac{n}{2}$ with $n\geq 4$, then the embedding $E\hookrightarrow L^\theta (\mathbb{R}^n)$ is continuous for all  $\theta\in [p^\ast,+\infty)$. In addition, the embedding $E\hookrightarrow L^\theta _{\textit{loc}}(\mathbb{R}^n)$ is compact for all $\theta\in [1,+\infty)$.
\end{lemma}
The following corollary follows from \cite[Lemma 2.1 and Lemma 2.2]{Yang-2012@}. 
\begin{corollary}\label{cor2.6}
    The function $s\mapsto\operatorname{\exp}(s)-\sum_{j=0}^{j
_0-1}\frac{s^j}{j!}$ is continuous, increasing and convex in $[0,+\infty)$, where $j_0$ is defined in \eqref{eq1.55}. Moreover, for any $\alpha>0$, $\wp>1$ and $n\geq 4$, we have
 $$ \bigg(\operatorname{exp}(\alpha|s|^{\frac{n}{n-2}})-\sum_{j=0}^{j_0-1} \frac{\alpha^j}{j!}|s|^{\frac{nj}{n-2}}\bigg)^\wp\leq  \operatorname{exp}(\alpha \wp|s|^{\frac{n}{n-2}})-\sum_{j=0}^{j_0-1} \frac{\alpha^j\wp^j}{j!}|s|^{\frac{nj}{n-2}},~\forall~s\in\mathbb{R}. $$
\end{corollary}      
\begin{lemma}{\cite[Theorem 1.2 and Lemma 3.5]{Fiorenza-Formica-Roskovec-Soudsky-2021}}\label{lem2.7}
If $ q\in[1,+\infty]$, $r\in[1,+\infty)$, $j,~k\in\mathbb{N}$, $j<k$ such that $\frac{1}{p}=\frac{j}{kr}+\frac{k-j}{kq}$, then there exists a constant $C>0$ independent of $u$ such that
\begin{align*}
   \|\nabla^ju\|_p\leq C  \|\nabla^k u\|^{\frac{j}{k}}_r \| u\|^{1-\frac{j}{k}}_q,~\forall~u\in L^q(\mathbb{R}^n)\cap W^{k,r}(\mathbb{R}^n).
\end{align*}
In particular, for $1<p<\frac{n}{2}$ with $n\geq 4$ and $t\in\{p,\frac{n}{2}\}$, there holds
 \begin{align*}
     \|\nabla u\|_t\leq C\|\Delta u\|^{\frac{1}{2}}_t\| u\|^{\frac{1}{2}}_t,~\forall~u\in W^{2,t}(\mathbb{R}^n).
 \end{align*}
\end{lemma}
Next, for the convenience of the reader, we include a proof of the following result.
\begin{lemma}\label{lem2.8}
 Suppose $\{u_k\}_k\subset E$ is a sequence satisfying $u_k\to u$ in $E$ as $k\to\infty$. Then there exists a subsequence $\{u_{k_j}\}_j$  in $E$ and a nonegative function $g\in E$ such that
 \begin{align*}
     u_{k_j}(x)\to u(x)\quad\text{a.e. in}\quad \mathbb{R}^n \quad\text{as}\quad j\to\infty
 \end{align*}
 and
 \begin{align*}
    | u_{k_j}(x)|\leq g(x)\quad\text{a.e. in}\quad \mathbb{R}^n \quad\text{for each}\quad j\in \mathbb{N}.
 \end{align*}
\end{lemma}
\begin{proof}
  It follows from $u_k\to u$ in $E$ as $k\to\infty$ and Lemma \ref{lem2.5} that for any subsequence $\{u_{k_j}\}_j$ in $E$, we have  $u_{k_j}\to u$ in $L^\theta(\mathbb{R}^n)$ as $k\to\infty$ for all $\theta\in[ p^\ast,+\infty)$. Hence, one can easily deduce that $u_{k_j}(x)\to u(x)$ a.e. in $\mathbb{R}^n$ as $j\to\infty$. Now, we denote $f_k:=-\Delta u_k$ and $f:=-\Delta u$. This infers that $f_k\to f$ in $L^t(\mathbb{R}^n)$ as $k\to\infty$ for $t\in\{p,\frac{n}{2}\}$. Choose a subsequence $\{u_{k_j}\}_j$ in $E$ such that
  \begin{align*}
      \|f_{k_j}-f\|_p+\|f_{k_j}-f\|_{\frac{n}{2}}\leq 2^{-j}\quad\text{for all}\quad j\in\mathbb{N}.
  \end{align*}
  This shows that
  \begin{align*}
      \sum_{j=1}^{\infty}\left(\|f_{k_j}-f\|_p+\|f_{k_j}-f\|_{\frac{n}{2}}\right)=\sum_{j=1}^{\infty} 2^{-j}<+\infty.
  \end{align*}
  Define
  \begin{align*}
      F(x):=\sum_{j=1}^{\infty}\left|f_{k_j}(x)-f\right(x)|\quad\text{for all}\quad x\in \mathbb{R}^n.
  \end{align*}
In virtue of Minkowski's inequality, we deduce that $F\in L^t(\mathbb{R}^n)$ for $t\in\{p,\frac{n}{2}\}$. Let $\Phi\colon \mathbb{R}^n\setminus\{0\}\to[0,+\infty)$ be defined by
  \begin{align*}
      \Phi(x):=\frac{1}{(n-2)\omega_{n-1}}\,|x|^{2-n}\quad\text{for all}\quad x\in \mathbb{R}^n\setminus\{0\}\quad\text{and}\quad n\geq4,
  \end{align*}
which is the fundamental solution of $-\Delta$, namely, there holds $-\Delta\Phi=\delta_0$ in $\mathcal{D}^\prime(\mathbb{R}^n)$,
where $\mathcal{D}^\prime(\mathbb{R}^n)$ denotes the space of distributions on $\mathbb{R}^n$. By convolving both the sides with $u$, we get $-\Delta\Phi\ast u=\delta_0 \ast u$. Recalling that $\delta_0 \ast u=u$, one has
\begin{align*}
    u=-\Delta\Phi\ast u=\Phi\ast -\Delta u.
\end{align*}
In particular, we have 
\begin{align*}
    u=\Phi\ast f\qquad\text{and }\qquad u_{k_j}=\Phi\ast f_{k_j}.
\end{align*}
We again define
$$g:=\Phi\ast(|f|+F). $$
By using $\Phi\geq0$, we obtain
\begin{align*}
    |u_{k_j}|=|\Phi\ast f_{k_j}|\leq \Phi\ast |f_{k_j}|\leq \Phi\ast (|f|+|f_{k_j}-f|)\leq \Phi\ast (|f|+F)=g.
\end{align*}
This yields that 
 \begin{align*}
    | u_{k_j}(x)|\leq g(x)\quad\text{a.e. in}\quad \mathbb{R}^n \quad\text{for each}\quad j\in \mathbb{N}.
 \end{align*}
 Moreover, there also holds
 \begin{align*}
     -\Delta g=-\Delta(\Phi\ast(|f|+F))=-\Delta\Phi\ast (|f|+F)=|f|+F\in L^t(\mathbb{R}^n)\quad\text{for}\quad t\in\left\{p,\frac{n}{2}\right\}.
 \end{align*}
 This implies that $g\in E$ and thus, the proof is completed at this point.
\end{proof}
\section{Sharp singular Adams' type inequality: Proof of Theorem \ref{thm1.1}}\label{sec3}
In this section, our aim is to prove Theorem \ref{thm1.1}. To establish Theorem \ref{thm1.1}, we shall use a rearrangement-free approach introduced by Lam-Lu \cite{Lam-Lu-2013}. Furthermore, as a consequence of Theorem \ref{thm1.1} and more specifically inspired by Nguyen \cite{Nguyen-2019}, we shall also demonstrate the $L^1$-integrability of the Young function \eqref{eq1.55} with singular weights on $\mathbb{R}^n$ in the desired function space. 
\begin{proof}[\bf{Proof of Theorem \ref{thm1.1}}.]
Invoking the fact that $ \Phi_{\alpha,j_0}(s)= \Phi_{\alpha,j_0}(|s|)$ and exploiting the denseness of $C^\infty_0(\mathbb{R}^n)$ in $E$, without loss of generality, we can take that $u\in C^\infty_0(\mathbb{R}^n)\setminus\{0\}$ with $u\geq 0$ and $\|u\|\leq 1$. Define the following
$$H(u):=2^{-\frac{4}{n(n-2)}}\|\Delta u\|_p\quad\text{and}\quad~S:=\{x\in \mathbb{R}^n\colon~u(x)>H(u)\}.$$
It follows at once that
\begin{equation}\label{eq3.1}
    H(u)<1\quad\text{and}\quad |S|\leq 2^{\frac{4p^\ast}{n(n-2)}}C,
\end{equation}
where $C>0$ is a suitable constant depending only on $p$ and $n$,  thanks to the chain of embeddings $E\hookrightarrow D^{2,p}(\mathbb{R}^n)\hookrightarrow L^{p^\ast}(\mathbb{R}^n)$. This shows that $S$ is a bounded domain. In addition, we can notice that
$$\int_{\mathbb{R}^n}\frac{\Phi_{\alpha,j_0}(u)}{|x|^\gamma}\,\mathrm{d}x=\int_{S}\frac{\Phi_{\alpha,j_0}(u)}{|x|^\gamma}\,\mathrm{d}x+\int_{\mathbb{R}^n\setminus S}\frac{\Phi_{\alpha,j_0}(u)}{|x|^\gamma}\,\mathrm{d}x:=J_1+J_2. $$
Initially, we shall evaluate $J_2$. By using \eqref{eq3.1}, we have $\mathbb{R}^n\setminus S\subset \{x\in \mathbb{R}^n:~u(x)<1\}$. Due to Lemma \ref{lem2.5} and the fact that $\|u\|\leq 1$, we obtain
\begin{equation}\label{eq3.2}
\begin{aligned}
J_2 & \leq \int_{\{x\in \mathbb{R}^n\colon~ u(x)<1\}}\frac{1}{|x|^\gamma}\sum_{j=j_0}^\infty \frac{\alpha^j}{j!}|u|^{\frac{nj}{n-2}}\,\mathrm{d}x \leq \int_{\{x\in \mathbb{R}^n\colon~ u(x)<1\}}\frac{1}{|x|^\gamma}\sum_{j=j_0}^\infty \frac{\alpha^j}{j!}|u|^{p^\ast}\,\mathrm{d}x \\ 
 & \leq \operatorname{\exp}(\alpha)\Bigg[ \int_{\{x\in \mathbb{R}^n\colon~ u(x)<1,~|x|<1\}}\frac{|u|^{p^\ast}}{|x|^\gamma}\,\mathrm{d}x+  \int_{\{x\in \mathbb{R}^n\colon~u(x)<1,~|x|\geq1\}}\frac{|u|^{p^\ast}}{|x|^\gamma}\,\mathrm{d}x\Bigg]\\ 
& \leq \operatorname{\exp}(\alpha)\Bigg[ \int_{\{ x\in \mathbb{R}^n\colon~|x|<1\}}\frac{1}{|x|^\gamma}\,\mathrm{d}x+ \int_{\{ x\in \mathbb{R}^n\colon~|x|\geq1\}}|u|^{p^\ast}\,\mathrm{d}x\Bigg]\\
&\leq C(\alpha,\gamma,p,n),
\end{aligned}
\end{equation}
where $C(\alpha,\gamma,p,n)>0$ is a constant depending only on  $\alpha$, $\gamma$, $p$ and $n$. To evaluate $J_1$, we first denote $w(x):=u(x)-H(u)$ in $S$. It is easy to see that $w\in W^{2,\frac{n}{2}}_\mathcal{N}(S)$. As a result, we have the following estimates 
\begin{align*}
    |u|^{\frac{n}{n-2}} &\leq  (|w|+H(u))^{\frac{n}{n-2}} \leq |w|^{\frac{n}{n-2}}+\frac{n}{n-2} 2^{\frac{n}{n-2}-1}\bigg[|w|^{\frac{n}{n-2}-1}H(u)+|H(u)|^{\frac{n}{n-2}}\bigg]\\
    & \leq |w|^{\frac{n}{n-2}}+\frac{n}{n-2} 2^{\frac{n}{n-2}-1}\Bigg[\frac{|w|^{\frac{n}{n-2}}|H(u)|^{\frac{n}{2}}}{\frac{n}{2}}+\frac{n-2}{n}+|H(u)|^{\frac{n}{n-2}}\Bigg] \\
    & \leq \bigg[1+\frac{2^{\frac{n}{n-2}}}{n-2}|H(u)|^{\frac{n}{2}}\bigg]|w|^{\frac{n}{n-2}}+D_n\quad\text{in}~~ S,
\end{align*}
we thank to the well-known Young's inequality and the following elementary inequality of calculus
$$(a+b)^\wp\leq a^\wp+\wp 2^{\wp-1}(a^{\wp-1}b+b^\wp),~\forall~\wp\geq 1\quad\text{and}\quad a,b\geq 0.$$
Set $$\xi(x):=\bigg[1+\frac{2^{\frac{n}{n-2}}}{n-2}|H(u)|^{\frac{n}{2}}\bigg]^{\frac{n-2}{n}}w(x)\quad\text{in}~~ S.$$
Observe that $\xi\in W^{2,\frac{n}{2}}_\mathcal{N}(S)$ and thus, one has  $|u|^{\frac{n}{n-2}}\leq |\xi|^{\frac{n}{n-2}}+D_n$ in $S$. In addition, by direct calculation, we get
\begin{equation*}
    \|\Delta \xi\|^{\frac{n}{2}}_{L^{\frac{n}{2}}(S)}\leq \bigg[1+\frac{2^{\frac{n}{n-2}}}{n-2}|H(u)|^{\frac{n}{2}}\bigg]^{\frac{n-2}{2}}\big(1-\|\Delta u\|^{\frac{n}{2}}_p\big).
\end{equation*}
The above inequality together with the fact that $(1-x)^\wp\leq 1-\wp x$ for all $x\in[0,1]$ and $\wp\in(0,1]$ imply that 
\begin{align*}
 \|\Delta \xi\|^{\frac{n}{n-2}}_{L^{\frac{n}{2}}(S)}&\leq \bigg[1+\frac{2^{\frac{n}{n-2}}}{n-2}|H(u)|^{\frac{n}{2}}\bigg]\Big(1-{\frac{2}{n-2}}\|\Delta u\|^{\frac{n}{2}}_p\Big)\\ 
 &\leq \Big(1+{\frac{2}{n-2}}\|\Delta u\|^{\frac{n}{2}}_p\Big) \Big(1-{\frac{2}{n-2}}\|\Delta u\|^{\frac{n}{2}}_p\Big)\\ 
 &= 1-\frac{4}{(n-2)^2}\|\Delta u\|^n_p\leq 1.   
\end{align*}
It follows that $\xi\in W^{2,\frac{n}{2}}_\mathcal{N}(S)$ and $\|\Delta \xi\|_{L^{\frac{n}{2}}(S)}\leq 1$. Consequently, by employing the singular Adams' inequality established in \cite[Theorem 1.2]{Lam-Lu-2012@@@} with homogeneous Navier boundary condition, we obtain 
\begin{equation}\label{eq3.3}
\begin{aligned}
 J_1 & \leq \int_S\frac{\operatorname{\exp}(\alpha |u|^{\frac{n}{n-2}})}{|x|^\gamma}\,\mathrm{d}x\leq \int_S\frac{\operatorname{\exp}(\beta_{\gamma,n} (|\xi|^{\frac{n}{n-2}}+D_n))}{|x|^\gamma}\,\mathrm{d}x \\ 
 &\leq \operatorname{\exp}(\beta_{\gamma,n}D_n)C(\gamma,n)|S|^{1-\frac{\gamma}{n}}\leq C(\gamma,p,n),
\end{aligned}
\end{equation}
where $C(\gamma,p,n)>0$ is a constant depending only on  $\gamma$, $p$ and $n$. In virtue of \eqref{eq3.2} and \eqref{eq3.3}, we infer that \eqref{eq1.12} holds. Next, to check the validity of the sharpness of $\beta_{\gamma,n}$, inspired by Lu-Yang \cite{Lu-Yang-2009}, we can define the following sequence
\begin{align}\label{eq3.4}
    \xi_k(x):=\begin{cases}
        \Big(\frac{\ln k}{\beta(n,2)}\Big)^{1-\frac{2}{n}}+(1-k^{\frac{2}{n}}|x|^2)\frac{n \beta(n,2)^{\frac{2}{n}-1}}{2(\ln k)^{\frac{2}{n}}}\qquad\text{if}~|x|\in [0,k^{-\frac{1}{n}}],\\ 
          n\beta(n,2)^{\frac{2}{n}-1} (\ln k)^{-\frac{2}{n}}\ln\big(\frac{1}{|x|}\big)\quad\quad\quad\quad\quad\quad\text{if}~|x|\in [k^{-\frac{1}{n}},1], \\   \eta_k(x)\quad\quad\quad\quad\quad\quad\quad\quad\quad\quad\quad\quad\quad\quad\quad\quad\text{if}~|x|\in [1,+\infty),
    \end{cases}
\end{align}
where $\{\eta_k\}_k$ is a sequence of radially smooth functions such that
$$ \text{supp}(\eta_k)\subset \{x\in \mathbb{R}^n\colon~1<|x|<2 \}\quad\text{and}\quad \frac{\partial \eta_k}{\partial \nu}\Big|_{\partial B_1}=n\beta(n,2)^{\frac{2}{n}-1} (\ln k)^{-\frac{2}{n}} .$$
Moreover, $\eta_k$ and $\Delta \eta_k$ are all of $O((\ln k)^{-\frac{2}{n}})$. By direct calculations, one can notice that 
\begin{align*}
    \Delta \xi_k(x):=\begin{cases}
       -n^2\beta(n,2)^{\frac{2}{n}-1}k^{\frac{2}{n}} (\ln k)^{-\frac{2}{n}}\quad\quad\quad\quad\quad\quad\text{if}~|x|\in (0,k^{-\frac{1}{n}}),\\ 
        - n(n-2)\beta(n,2)^{\frac{2}{n}-1} (\ln k)^{-\frac{2}{n}}|x|^{-2}\quad\quad\text{if}~|x|\in (k^{-\frac{1}{n}},1), \\  \Delta\eta_k(x)\quad\quad\quad\quad\quad\quad\quad\quad\quad\quad\quad\quad\quad\quad\text{if}~|x|\in (1,+\infty).
    \end{cases}
\end{align*}
Hence, we have 
\begin{align*}
    \|\Delta \xi_k\|_p^p &=\int_{\{x\in\mathbb{R}^n\colon~1<|x|<2\}} |\Delta\eta_k|^p\,\mathrm{d}x+\omega_{n-1}\int_{0}^{k^{-\frac{1}{n}}}\big( n^2\beta(n,2)^{\frac{2}{n}-1}k^{\frac{2}{n}} (\ln k)^{-\frac{2}{n}}\big)^p r^{n-1}\,\mathrm{d}r\\
    & \quad+\omega_{n-1}\int_{k^{-\frac{1}{n}}}^{1}\big( n(n-2)\beta(n,2)^{\frac{2}{n}-1} (\ln k)^{-\frac{2}{n}}r^{-2}\big)^p r^{n-1}\,\mathrm{d}r\\
    &=O((\ln k)^{-\frac{2p}{n}})+\frac{\omega_{n-1}}{nk}\bigg( n^2\beta(n,2)^{\frac{2}{n}-1}\bigg(\frac{k}{\ln k}\bigg)^{\frac{2}{n}}\bigg)^p\\ &\quad +\frac{\omega_{n-1}}{n-2p}\big( n(n-2)\beta(n,2)^{\frac{2}{n}-1} (\ln k)^{-\frac{2}{n}}\big)^p\bigg(1-\frac{1}{k^{\frac{n-2p}{n}}}\bigg)\\
    &\leq O((\ln k)^{-\frac{2p}{n}})+\frac{\omega_{n-1}}{nk}\bigg( n^2\beta(n,2)^{\frac{2}{n}-1}\bigg(\frac{k}{\ln k}\bigg)^{\frac{2}{n}}\bigg)^p\\ &\quad +\frac{\omega_{n-1}}{n-2p}\big( n(n-2)\beta(n,2)^{\frac{2}{n}-1} (\ln k)^{-\frac{2}{n}}\big)^p\\
    &= O((\ln k)^{-\frac{2p}{n}})+O((\ln k)^{-\frac{2p}{n}})~\frac{1}{k^{\frac{n-2p}{n}}}
\end{align*}
and 
\begin{align*}
    \|\Delta \xi_k\|_{\frac{n}{2}}^{\frac{n}{2}} &=\int_{\{x\in\mathbb{R}^n\colon~1<|x|<2\}} |\Delta\eta_k|^{\frac{n}{2}}\,\mathrm{d}x+\omega_{n-1}\int_{0}^{k^{-\frac{1}{n}}}\big( n^2\beta(n,2)^{\frac{2}{n}-1}k^{\frac{2}{n}} (\ln k)^{-\frac{2}{n}}\big)^{\frac{n}{2}} r^{n-1}\,\mathrm{d}r\\
    & \quad+\omega_{n-1}\int_{k^{-\frac{1}{n}}}^{1}\big( n(n-2)\beta(n,2)^{\frac{2}{n}-1} (\ln k)^{-\frac{2}{n}}r^{-2}\big)^{\frac{n}{2}} r^{n-1}\,\mathrm{d}r\\
    &=O((\ln k)^{-1})+\frac{\omega_{n-1}}{nk}\bigg( n^2\beta(n,2)^{\frac{2}{n}-1}\bigg(\frac{k}{\ln k}\bigg)^{\frac{2}{n}}\bigg)^{\frac{n}{2}}\\&\quad +\omega_{n-1}\big( n(n-2)\beta(n,2)^{\frac{2}{n}-1} (\ln k)^{-\frac{2}{n}}\big)^{\frac{n}{2}}\frac{\ln k}{n}\\
    &=O((\ln k)^{-1})+\frac{\omega_{n-1}}{n}\big( n(n-2)\beta(n,2)^{\frac{2}{n}-1}\big)^{\frac{n}{2}}=O((\ln k)^{-1})+1.
\end{align*}
This yields at once that 
$$\|\xi_k\|^{\frac{n}{2}}_{\frac{n}{2}}=\|\Delta \xi_k\|_{\frac{n}{2}}^{\frac{n}{2}}+\|\Delta \xi_k\|_p^{\frac{n}{2}}\to 1\quad\text{as}~k\to\infty .$$
It follows that $\|\xi_k\|\to 1$ as $k\to\infty$. Let us assume $\alpha>\beta_{\gamma,n}$, then there exists $\delta>0$ such that $\alpha=(1+\delta)\beta_{\gamma,n}$. Moreover, one has
\begin{align*}
    \sup_{u\in E,~\|u\|\leq 1}\int_{\mathbb{R}^n}\frac{\Phi_{\alpha,j_0}(u)}{|x|^\gamma}\,& \mathrm{d}x \geq \int_{\{x\in\mathbb{R}^n\colon~|x|< k^{-\frac{1}{n}} \}}\frac{\Phi_{\alpha,j_0}(\xi_k/\|\xi_k\|)}{|x|^\gamma}\,\mathrm{d}x \\
    & =\omega_{n-1} \int_{0}^{k^{-\frac{1}{n}}}\frac{\Phi_{\alpha,j_0}\bigg(\frac{1}{\|\xi_k\|}\bigg(\Big(\frac{\ln k}{\beta(n,2)}\Big)^{1-\frac{2}{n}}+(1-k^{\frac{2}{n}}r^2)\frac{n \beta(n,2)^{\frac{2}{n}-1}}{2(\ln k)^{\frac{2}{n}}}\bigg)\bigg)}{r^{\gamma-n+1}}\,\mathrm{d}r\\
    &\geq \frac{\omega_{n-1}}{(n-\gamma)k^{\frac{n-\gamma}{n}}}~\Phi_{\alpha,j_0}\bigg(\frac{1}{\|\xi_k\|}\bigg(\frac{\ln k}{\beta(n,2)}\bigg)^{1-\frac{2}{n}}\bigg) \\
    &\sim \frac{\omega_{n-1}}{(n-\gamma)k^{\frac{n-\gamma}{n}}}~\operatorname{\exp}\Bigg(\alpha \Bigg|\frac{1}{\|\xi_k\|}\bigg(\frac{\ln k}{\beta(n,2)}\bigg)^{1-\frac{2}{n}}\Bigg|^{\frac{n}{n-2}}\Bigg)\\
    & =\frac{\omega_{n-1}}{(n-\gamma)}~\operatorname{\exp}\Bigg[\bigg(1-\frac{\gamma}{n}\bigg)\ln k\bigg(\frac{1+\delta}{\|\xi_k\|^{\frac{n}{n-2}}}-1\bigg)\Bigg]\to+\infty\quad\text{as}\quad k\to\infty.\end{align*} This completes the proof of Theorem \ref{thm1.1}.
\end{proof}
To end this section, we prove the following result, which is a byproduct of the previously discussed theorem.
\begin{corollary}\label{cor3.1}
Let $1<p<\frac{n}{2}$, $0\leq \gamma<n$ and $n\geq 4$ be hold. Then for all $\alpha\geq 0$ and $u\in E$, there holds
\begin{align*}
    \frac{\Phi_{\alpha,j_0}(u)}{|x|^\gamma}\in L^1(\mathbb{R}^n).
\end{align*}    
\end{corollary}
\begin{proof}
By employing the denseness of $C^\infty_0(\mathbb{R}^n)$ in $E$, without loss of generality, for a given $\varepsilon>0$, one can assume $\psi\in C^\infty_0(\mathbb{R}^n)$ such that $\|u-\psi\|<\varepsilon$. Moreover, we define two sets $S_1$ and $S_2$ as follows $$S_1:=\{x\in \mathbb{R}^n\colon~|u(x)-\psi(x)|\leq 1\}\quad\text{and}\quad S_2:=\{x\in \mathbb{R}^n\colon~|u(x)-\psi(x)|> 1\}.$$ Notice that $|u(x)|\leq 1+\sup_{x\in \mathbb{R}^n}|\psi(x)|
:=c$ for all $x\in S_1$, where $c>0$ is a suitable constant. By employing Lemma \ref{lem2.5} and using the fact that $\frac{nj_0}{n-2}\geq p^\ast$, we have
\begin{equation}\label{eq3.5}
\begin{aligned}
       \int_{S_1}\frac{\Phi_{\alpha,j_0}(u)}{|x|^\gamma}\,\mathrm{d}x&=\int_{S_1}\sum_{j=j_0}^{\infty}\frac{\big(\alpha c^{\frac{n}{n-2}}\big)^j}{j!}\bigg|\frac{u}{c}\bigg|^{\frac{nj}{n-2}}\,\frac{\mathrm{d}x}{|x|^\gamma}\leq \sum_{j=j_0}^{\infty}\frac{\big(\alpha c^{\frac{n}{n-2}}\big)^j}{j!}\int_{S_1}\bigg|\frac{u}{c}\bigg|^{\frac{nj_0}{n-2}}\,\frac{\mathrm{d}x}{|x|^\gamma}\\ &
       \lesssim  \int_{S_1}\frac{|u|^{\frac{nj_0}{n-2}}}{|x|^\gamma}\,\mathrm{d}x\lesssim \int_{\{x\in \mathbb{R}^n:~|x|<1\}} \frac{1}{|x|^\gamma} \,\mathrm{d}x+\int_{\{x\in \mathbb{R}^n:~|x|\geq 1\}} |u|^{\frac{nj_0}{n-2}} \,\mathrm{d}x<+\infty.
   \end{aligned}
   \end{equation}
   Recall that for a given $\eta>0$ and $\wp>1$, there holds
   \begin{align}\label{eq3.6}
     (a+b)^\wp\leq (1+\eta)a^\wp+C_\eta b^\wp,~\forall~a,b>0,  
\end{align}
where $C_\eta=\big(1-(1+\eta)^{-\frac{1}{\wp-1}}\big)^{1-\wp}$. Moreover, by using \eqref{eq3.6} and Young's inequality, one has
\begin{align*}
   \int_{S_2}\frac{\Phi_{\alpha,j_0}(u)}{|x|^\gamma}\,\mathrm{d}x & \leq \int_{S_2}\frac{\operatorname{\exp}(\alpha|u|^{\frac{n}{n-2}})}{|x|^\gamma}\,\mathrm{d}x\leq \int_{S_2 \cap~ \text{supp}(\psi)}\frac{\operatorname{\exp}(\alpha(1+\eta)|u-\psi|^{\frac{n}{n-2}}+C_\eta\alpha|\psi|^{\frac{n}{n-2}})}{|x|^\gamma}\,\mathrm{d}x\\
   & \lesssim \int_{S_2 \cap~ \text{supp}(\psi)}\frac{\operatorname{\exp}(2\alpha(1+\eta)|u-\psi|^{\frac{n}{n-2}})}{|x|^\gamma}\,\mathrm{d}x+\int_{\text{supp}(\psi)}\frac{\operatorname{\exp}(2C_\eta\alpha|\psi|^{\frac{n}{n-2}})}{|x|^\gamma}\,\mathrm{d}x.
\end{align*}
Let $\eta,\varepsilon>0$ be sufficiently small enough such that $2\alpha(1+\eta)\varepsilon^{\frac{n}{n-2}}< \beta_{\gamma,n}$ holds. Hence, we obtain by using the definition of $S_2$ and Theorem \ref{thm1.1} that
\begin{align*}
\int_{S_2 \cap~ \text{supp}(\psi)}\frac{\operatorname{\exp}(2\alpha(1+\eta)|u-\psi|^{\frac{n}{n-2}})}{|x|^\gamma}\,\mathrm{d}x &\lesssim \int_{\mathbb{R}^n}\frac{\Phi_{2\alpha(1+\eta),j_0}(u-\psi)}{|x|^\gamma}\,\mathrm{d}x\\
&\lesssim\int_{\mathbb{R}^n}\frac{\Phi_{\beta_{\gamma,n},j_0}(u-\psi/\|u-\psi\|)}{|x|^\gamma}\,\mathrm{d}x <+\infty.  
\end{align*}
By using $\psi$ has compact support, we obtain
$$ \int_{\text{supp}(\psi)}\frac{\operatorname{\exp}(2C_\eta\alpha|\psi|^{\frac{n}{n-2}})}{|x|^\gamma}\,\mathrm{d}x<+\infty. $$
This immediately implies that
\begin{align}\label{eq3.7}
      \int_{S_2}\frac{\Phi_{\alpha,j_0}(u)}{|x|^\gamma}\,\mathrm{d}x<+\infty.  
\end{align}
Based on \eqref{eq3.5} and \eqref{eq3.7}, we infer that the assertion of the corollary is well-established.   
\end{proof}
\section{Sharp singular concentration-compactness principle: Proof of Theorem \ref{thm1.2}}\label{sec4}
This section focuses on proving the concentration-compactness principle of the singular Adams' type inequality stated in the previous section (see Theorem \ref{thm1.1}), which relies on conventional Schwarz symmetrization techniques. More information for a detailed analysis in this direction can be found in \cite{do O-Macedo-2014, Nguyen-2019}. The proof of Theorem \ref{thm1.2} is based on the following crucial lemma. More precisely, we have 
\begin{lemma}\label{lem3.2}
  Let the hypotheses of Theorem \ref{thm1.2} be satisfied. Then there exists a sequence $\{\overset{\sim}{u}_k\}_k\subset C^\infty_0(\mathbb{R}^n)$ such that $\|\overset{\sim}{u}_k\|=1$ and $\overset{\sim}{u}_k\rightharpoonup u$ in $E$ as $k\to\infty$. In addition, for some suitable conatant $C$ independent of $k$, there holds
  \begin{align*}
   \sup_{k\in\mathbb{N}}\int_{\mathbb{R}^n}\frac{\Phi_{\ell \beta_{\gamma,n},j_0}(u_k)}{|x|^\gamma}\,\mathrm{d}x\leq C\sup_{k\in\mathbb{N}}\int_{\mathbb{R}^n}\frac{\Phi_{\ell^\prime \beta_{\gamma,n},j_0}(\overset{\sim}{u}_k)}{|x|^\gamma}\, \mathrm{d}x,~\forall~\ell^\prime\in(\ell,L_n(u)).\end{align*}
\end{lemma}
\begin{proof}
  By the standard density argument, we can assume a sequence $\{\psi_k\}_k\subset C^\infty_0(\mathbb{R}^n)$ such that $\|u_k-\psi_k\|<\frac{1}{k}$ for all $k\in\mathbb{N}$. Suppose that $\overset{\sim}{u}_k:=\frac{\psi_k}{\|\psi_k\|}$ for all $k\in\mathbb{N}$. It follows that the sequence $\{\overset{\sim}{u}_k\}_k\subset C^\infty_0(\mathbb{R}^n)$ and satisfying 
  \begin{align}\label{eq3.999}
   \|\overset{\sim}{u}_k\|=1, \quad\|u_k-\overset{\sim}{u}_k\|\leq \frac{2}{k},~\forall~k\in\mathbb{N}\quad\text{and}\quad\overset{\sim}{u}_k\rightharpoonup u\quad\text{in}\quad E\quad\text{as}\quad k\to\infty.   
  \end{align}
 Divide the whole space $\mathbb{R}^n$ as follows
  \begin{align*}
      S_1:=\{x\in \mathbb{R}^n\colon~|u_k(x)|\leq 2\}\quad\text{and}\quad S_2:=\{x\in \mathbb{R}^n\colon~|u_k(x)|> 2\}.
  \end{align*}
Notice that
$$\int_{\mathbb{R}^n}\frac{\Phi_{\ell \beta_{\gamma,n},j_0}(u_k)}{|x|^\gamma}\, \mathrm{d}x=\int_{S_1}\frac{\Phi_{\ell \beta_{\gamma,n},j_0}(u_k)}{|x|^\gamma}\, \mathrm{d}x+\int_{S_2}\frac{\Phi_{\ell \beta_{\gamma,n},j_0}(u_k)}{|x|^\gamma}\,\mathrm{d}x:=J_1+J_2. $$
Due to Lemma \ref{lem2.5}, we obtain for all $k\in\mathbb{N}$ that
\begin{equation}\label{eq3.8}
\begin{aligned}
J_1 & = \int_{S_1}\frac{1}{|x|^\gamma}\sum_{j=j_0}^\infty \frac{(2^{\frac{n}{n-2}}\ell\beta_{\gamma,n})^j}{j!}\Big|\frac{u_k}{2}\Big|^{\frac{nj}{n-2}}\, \mathrm{d}x \leq \int_{S_1}\frac{1}{|x|^\gamma}\sum_{j=j_0}^\infty \frac{(2^{\frac{n}{n-2}}\ell\beta_{\gamma,n})^j}{j!}\Big|\frac{u_k}{2}\Big|^{p^\ast}\,\mathrm{d}x \\ 
 & \leq \operatorname{\exp}(2^{\frac{n}{n-2}}\ell\beta_{\gamma,n})\Bigg[ \int_{\{x\in S_1\colon~|x|<1\}}\frac{|u_k|^{p^\ast}}{2^{p^\ast}|x|^\gamma}\,\mathrm{d}x+ \int_{\{x\in S_1\colon~|x|\geq1\}}\frac{|u_k|^{p^\ast}}{2^{p^\ast}|x|^\gamma}\, \mathrm{d}x\Bigg] \\ 
& \leq \operatorname{\exp}(2^{\frac{n}{n-2}}\ell\beta_{\gamma,n})\Bigg[ \int_{\{x\in S_1\colon~|x|<1\}}\frac{1}{|x|^\gamma}\,\mathrm{d}x+  \int_{\{x\in S_1\colon~|x|\geq1\}}|u_k|^{p^\ast}\, \mathrm{d}x\Bigg] \\ & \leq C(\gamma,\ell,p,n),
\end{aligned}
\end{equation}
where $C(\gamma,\ell,p,n)>0$ is a constant depending only on $\gamma,\ell,p$ and $n$ but independent on $k$. Further, we can also write
$$
 J_2=\int_{\Gamma_1}\frac{\Phi_{\ell \beta_{\gamma,n},j_0}(u_k)}{|x|^\gamma}\,\mathrm{d}x+\int_{\Gamma_2}\frac{\Phi_{\ell \beta_{\gamma,n},j_0}(u_k)}{|x|^\gamma}\,\mathrm{d}x:=J_{2,\Gamma_1}+J_{2,\Gamma_2},   
$$
where 
$$
  \Gamma_1:=\{x\in S_2\colon~|u_k(x)-\overset{\sim}{u}_k(x)|\leq 1\}\quad\text{and}\quad\Gamma_2:=\{x\in S_2\colon~|u_k(x)-\overset{\sim}{u}_k(x)|> 1\}.   
$$
By direct calculation, one has $|\overset{\sim}{u}_k(x)|>1$ for all $x\in \Gamma_1$ and $k\in\mathbb{N}$. In virtue of \eqref{eq3.6} with $\eta=\frac{\ell^\prime-\ell}{\ell}$, we get
\begin{equation}\label{eq3.9}
\begin{aligned}
    J_{2,\Gamma_1}& \leq \int_{\Gamma_1}\frac{\operatorname{\exp}(\ell\beta_{\gamma,n}|u_k|^{\frac{n}{n-2}})}{|x|^\gamma}\, \mathrm{d}x\leq \int_{\Gamma_1}\frac{\operatorname{\exp}(\ell^\prime\beta_{\gamma,n}|\overset{\sim}{u}_k|^{\frac{n}{n-2}}+C_\eta \ell\beta_{\gamma,n})}{|x|^\gamma}\, \mathrm{d}x \\  &\lesssim \int_{\mathbb{R}^n} \frac{\Phi_{\ell^\prime \beta_{\gamma,n},j_0}(\overset{\sim}{u}_k)}{|x|^\gamma}\, \mathrm{d}x\lesssim \sup_{k\in\mathbb{N}}\int_{\mathbb{R}^n}\frac{\Phi_{\ell^\prime \beta_{\gamma,n},j_0}(\overset{\sim}{u}_k)}{|x|^\gamma}\, \mathrm{d}x.
\end{aligned}
\end{equation}
We again split the integral $J_{2,\Gamma_2}$ as follows
$$
 J_{2,\Gamma_2}=\int_{\Gamma_{21}}\frac{\Phi_{\ell \beta_{\gamma,n},j_0}(u_k)}{|x|^\gamma}\, \mathrm{d}x+\int_{\Gamma_{22}}\frac{\Phi_{\ell \beta_{\gamma,n},j_0}(u_k)}{|x|^\gamma}\, \mathrm{d}x:=J_{2,\Gamma_{21}}+J_{2,\Gamma_{22}},   
$$
where 
$$\Gamma_{21}:=\Gamma_{2}\cap \{x\in \mathbb{R}^n\colon~|\overset{\sim}{u}_k(x)|< 1\}\quad\text{and}\quad\Gamma_{22}:=\Gamma_{2}\cap \{x\in \mathbb{R}^n\colon~|\overset{\sim}{u}_k(x)|\geq 1\}.   
$$
By utilizing the inequality \eqref{eq3.6} with $\eta=1$, we obtain from the definition of $\Gamma_2$ that
\begin{equation}\label{eq3.10}
\begin{aligned}
    J_{2,\Gamma_{21}}&  \leq \int_{\Gamma_{21}}\frac{\operatorname{\exp}(\ell\beta_{\gamma,n}|u_k|^{\frac{n}{n-2}})}{|x|^\gamma}\,\mathrm{d}x\leq \int_{\Gamma_{21}}\frac{\operatorname{\exp}(2\ell\beta_{\gamma,n}|u_k-\overset{\sim}{u}_k|^{\frac{n}{n-2}}+C_1 \ell\beta_{\gamma,n})}{|x|^\gamma}\, \mathrm{d}x \\  &\lesssim \int_{\mathbb{R}^n} \frac{\Phi_{2\ell \beta_{\gamma,n},j_0}(u_k-\overset{\sim}{u}_k)}{|x|^\gamma}\, \mathrm{d}x.
\end{aligned}
\end{equation}
Set $w_k:=\frac{u_k-\overset{\sim}{u}_k}{\|u_k-\overset{\sim}{u}_k\|}$ and take $k_0\in\mathbb{N}$ sufficiently large enough such that $2\ell\beta_{\gamma,n}\big(\frac{2}{k}\big)^{\frac{n}{n-2}}< \beta_{\gamma,n}$ for all $k\geq k_0$. Then by using \eqref{eq3.999}, Corollary \ref{cor2.6} and Theorem \ref{thm1.1}, we obtain for all $k\geq k_0$ that
$$ \int_{\mathbb{R}^n} \frac{\Phi_{2\ell \beta_{\gamma,n},j_0}(u_k-\overset{\sim}{u}_k)}{|x|^\gamma}\, \mathrm{d}x\leq \int_{\mathbb{R}^n} \frac{\Phi_{ \beta_{\gamma,n},j_0}(w_k)}{|x|^\gamma}\, \mathrm{d}x\leq \overset{\sim}{C},$$
for some suitable constant $\overset{\sim}{C}>0$, which is independent on the choice of $k$. In addition, by using Corollary \ref{cor3.1}, one has
$$ C=\max\bigg\{\overset{\sim}{C},\int_{\mathbb{R}^n} \frac{\Phi_{2\ell \beta_{\gamma,n},j_0}(u_1-\overset{\sim}{u}_1)}{|x|^\gamma}\, \mathrm{d}x,\dots,\int_{\mathbb{R}^n} \frac{\Phi_{2\ell \beta_{\gamma,n},j_0}(u_{k_0}-\overset{\sim}{u}_{k_0})}{|x|^\gamma}\, \mathrm{d}x\bigg\}<+\infty. $$
In conclusion, we deduce that
\begin{align}\label{eq3.11}
\int_{\mathbb{R}^n} \frac{\Phi_{2\ell \beta_{\gamma,n},j_0}(u_k-\overset{\sim}{u}_k)}{|x|^\gamma}\, \mathrm{d}x\leq C,~\forall~k\in\mathbb{N}.   
\end{align}
From \eqref{eq3.10} and \eqref{eq3.11}, one can notice that
\begin{align}\label{eq3.12}
    J_{2,\Gamma_{21}}= \int_{\Gamma_{21}}\frac{\Phi_{\ell \beta_{\gamma,n},j_0}(u_k)}{|x|^\gamma}\, \mathrm{d}x\leq C,~\forall~k\in\mathbb{N}, 
\end{align}
for some suitable constant $C>0$ independent on the choice of  $k$. Let $t,t^\prime>1$ be such $\frac{1}{t}+\frac{1}{t^\prime}=1$, where $t=\frac{2\ell^\prime}{\ell+\ell^\prime}$. Choose $\eta=\frac{\ell^\prime-\ell}{2\ell}$, then by using \eqref{eq3.6} and Young's inequality, we have
\begin{equation}\label{eq3.13}
\begin{aligned}
    J_{2,\Gamma_{22}}&  \leq \int_{\Gamma_{22}}\frac{\operatorname{\exp}(\ell\beta_{\gamma,n}|u_k|^{\frac{n}{n-2}})}{|x|^\gamma}\, \mathrm{d}x\leq \int_{\Gamma_{22}}\frac{\operatorname{\exp}((1+\eta)\ell\beta_{\gamma,n}|\overset{\sim}{u}_k|^{\frac{n}{n-2}}+C_\eta \ell\beta_{\gamma,n}|u_k-\overset{\sim}{u}_k|^{\frac{n}{n-2}})}{|x|^\gamma}\, \mathrm{d}x \\ 
    & \leq \frac{1}{t}\int_{\Gamma_{22}}\frac{\operatorname{\exp}(\ell^\prime\beta_{\gamma,n}|\overset{\sim}{u}_k|^{\frac{n}{n-2}})}{|x|^\gamma}\, \mathrm{d}x+ \frac{1}{t^\prime}\int_{\Gamma_{22}}\frac{\operatorname{\exp}(t^\prime C_\eta\ell\beta_{\gamma,n}|u_k-\overset{\sim}{u}_k|^{\frac{n}{n-2}})}{|x|^\gamma}\, \mathrm{d}x \\
    & \leq C_1 \int_{\mathbb{R}^n}\frac{\Phi_{\ell^\prime \beta_{\gamma,n},j_0}(\overset{\sim}{u}_k)}{|x|^\gamma}\, \mathrm{d}x+C_2\int_{\mathbb{R}^n}\frac{\Phi_{t^\prime C_\eta\ell\beta_{\gamma,n},j_0}(u_k-\overset{\sim}{u}_k)}{|x|^\gamma}\, \mathrm{d}x,
\end{aligned}
\end{equation}
where $C_1>0$ and $C_2>0$ are some suitable constants independent on the choice of $k$. Let $k_0\in\mathbb{N}$ be sufficiently large enough such that $t^\prime C_\eta\ell\beta_{\gamma,n}\big(\frac{2}{k}\big)^{\frac{n}{n-2}}< \beta_{\gamma,n}$ for all $k\geq k_0$. Then arguing as before, it is not difficult to prove that
\begin{align}\label{eq3.14}
    \int_{\mathbb{R}^n}\frac{\Phi_{t^\prime C_\eta\ell\beta_{\gamma,n},j_0}(u_k-\overset{\sim}{u}_k)}{|x|^\gamma}\, \mathrm{d}x\leq C_3,~\forall~k\in\mathbb{N}, 
\end{align}
for some suitable constant $C_3>0$ independent of $k$. Thus, by using \eqref{eq3.13} and \eqref{eq3.14}, we infer that 
\begin{align}\label{eq3.15}
 \sup_{k\in\mathbb{N}} J_{2,\Gamma_{22}} \leq C_1\sup_{k\in\mathbb{N}} \int_{\mathbb{R}^n}\frac{\Phi_{\ell^\prime \beta_{\gamma,n},j_0}(\overset{\sim}{u}_k)}{|x|^\gamma}\,\mathrm{d}x+C_4,
\end{align}
where $C_4>0$ is some suitable constant independent of the choice of $k$. Hence,
by using \eqref{eq3.8}, \eqref{eq3.9}, \eqref{eq3.12} and \eqref{eq3.15}, we conclude the proof.  
\end{proof}

Now, we are in a position to prove the concentration-compactness principle as follows.
\begin{proof}[\bf{Proof of Theorem \ref{thm1.2}}.]
To prove Theorem \ref{thm1.2}, first one can see that $\Phi_{\alpha,j_0}(|s|)=\Phi_{\alpha,j_0}(s)$ and thus, without loss of generality, we assume that $\{u_k\}_k\subset E$ and its weak limit $u$ are nonnegative functions. Moreover, it follows from the weak lower semicontinuity of the norm in $E$ that
$$\|u\|\leq\liminf_{k\to\infty}\|u_k\|=1. $$
Thus, the proof of Theorem \ref{thm1.2} is based on two cases as follows:

\textbf{Case-1:} In this case, we consider the situation $\|u\|=1$. Recall that $E$ is a reflexive Banach space, $u_k\rightharpoonup u$ in $E$ and $\|u_k\|\to \|u\|$ in $\mathbb{R}$ as $k\to\infty$. Hence, one can deduce from \cite[Corollary A.2]{Autuori-Pucci-2013} that $u_k\to u$ in $E$ as $k\to\infty$. It follows from Lemma \ref{lem2.8} that up to a subsequence of $\{u_k\}_k\subset E$ still denoted by same symbol, there exists some $w\in E$ such that $|u_k(x)|\leq w(x)$ for a.e. $\mathbb{R}^n$ for all $k\in\mathbb{N}$. Thus, we obtain from Corollary \ref{cor3.1} and the Lebesgue dominated convergence theorem that
\begin{align}\label{eq3.898}
\lim_{k\to\infty} \int_{\mathbb{R}^n}\frac{\Phi_{\ell\beta_{\gamma,n},j_0}(u_k)}{|x|^\gamma}\,\mathrm{d}x=\int_{\mathbb{R}^n}\frac{\Phi_{\ell\beta_{\gamma,n},j_0}(u)}{|x|^\gamma}\,\mathrm{d}x<+\infty.   
\end{align}
Now, by using Corollary \ref{cor3.1} and \eqref{eq3.898}, we conclude that \eqref{eq1.13} holds.

\textbf{Case-2:} In this case, we consider the situation $0<\|u\|<1$. The proof of the first part of Theorem \ref{thm1.2} is based on the method of contradiction. Due to Lemma \ref{lem3.2}, it suffices to prove Theorem \ref{thm1.2} for compactly supported smooth functions. Indeed, let $\{u_k\}_k\subset C^\infty_0(\mathbb{R}^n)$ be such that for any $\ell\in(0,L_n(u))$, we have
\begin{align}\label{eq3.17}
\|u_k\|=1,\quad u_k\rightharpoonup u\neq 0\quad\text{in}\quad E\quad\text{as}\quad k\to\infty\quad\text{and}\quad\lim_{k\to\infty}\int_{\mathbb{R}^n}\frac{\Phi_{\ell\beta_{\gamma,n},j_0}(u_k)}{|x|^\gamma}\,\ \mathrm{d}x=+\infty .
\end{align}
Let $u_k^\ast$ be the Schwarz symmetrization of $u_k$. Notice that $(1/|x|^\gamma)^\ast=1/|x|^\gamma$ and therefore, by applying Hardy-Littlewood inequality, for instance, see Lemma \ref{lem2.1} and  \eqref{eq2.1}, we have
\begin{align}\label{eq3.18}
  \int_{\mathbb{R}^n}\frac{\Phi_{\ell\beta_{\gamma,n},j_0}(u_k)}{|x|^\gamma}\,\ \mathrm{d}x\leq \int_{\mathbb{R}^n}\frac{\Phi_{\ell\beta_{\gamma,n},j_0}(u^\ast_k)}{|x|^\gamma}\, \mathrm{d}x.  
\end{align}
From \eqref{eq3.17} and \eqref{eq3.18}, one has
\begin{align}\label{eq3.19}
 \int_{\mathbb{R}^n}\frac{\Phi_{\ell\beta_{\gamma,n},j_0}(u^\ast_k)}{|x|^\gamma}\, \mathrm{d}x\nearrow +\infty \quad\text{as}\quad k\to\infty.    
\end{align}
Let $R>0$ be fixed and join us in writing
$$\int_{\mathbb{R}^n}\frac{\Phi_{\ell\beta_{\gamma,n},j_0}(u^\ast_k)}{|x|^\gamma}\, \mathrm{d}x=\int_{B_R}\frac{\Phi_{\ell\beta_{\gamma,n},j_0}(u^\ast_k)}{|x|^\gamma}\, \mathrm{d}x+\int_{B_R^c}\frac{\Phi_{\ell\beta_{\gamma,n},j_0}(u^\ast_k)}{|x|^\gamma}\, \mathrm{d}x,~\forall~ k\in\mathbb{N}. $$
It follows directly that
\begin{align*}
 \int_{B_R^c}\frac{\Phi_{\ell\beta_{\gamma,n},j_0}(u^\ast_k)}{|x|^\gamma}\, \mathrm{d}x\leq \frac{1}{R^\gamma} \int_{B_R^c}\Phi_{\ell\beta_{\gamma,n},j_0}(u^\ast_k)\,\ \mathrm{d}x.  
\end{align*}
Due to the monotone convergence theorem together with Lemma \ref{lem2.5} and \eqref{eq3.17} yields
\begin{align*}
  \int_{B_R^c}\Phi_{\ell\beta_{\gamma,n},j_0}(u^\ast_k)\, \mathrm{d}x &= \int_{B_R^c}\sum_{j=j_0}^{\infty} \frac{(\ell\beta_{\gamma,n})^j}{j!}|u^\ast_k|^{\frac{nj}{n-2}}\, \mathrm{d}x=\sum_{j=j_0}^{\infty}\int_{B_R^c} \frac{(\ell\beta_{\gamma,n})^j}{j!}|u^\ast_k|^{\frac{nj}{n-2}}\, \mathrm{d}x\\ 
  &\leq \frac{(\ell\beta_{\gamma,n})^{j_0}}{j_0!}\|u_k\|^{\frac{nj_0}{n-2}}_{\frac{nj_0}{n-2}}+\sum_{j=j_0+1}^{\infty}\frac{(\ell\beta_{\gamma,n})^j}{j!}\int_{B_R^c}|u^\ast_k|^{\frac{nj}{n-2}}\, \mathrm{d}x\\
  &\leq \frac{(\ell\beta_{\gamma,n})^{j_0}} {j_0!}c_1+\sum_{j=j_0+1}^{\infty}\frac{(\ell\beta_{\gamma,n})^j}{j!}\int_{B_R^c}|u^\ast_k|^{\frac{nj}{n-2}}\, \mathrm{d}x,
\end{align*}
where $c_1>0$ is a suitable constant independent of the choice of $k$. Now, by using Lemma \ref{lem2.4}, Lemma \ref{lem2.5} and \eqref{eq3.17}, one sees that
$$ |u_k^\ast(x)|\leq \frac{1}{ |x|^{\frac{n}{p^\ast}}}\bigg(\frac{n}{\omega_{n-1}}\bigg)^{\frac{1}{p^\ast}}\|u_k\|_{p^\ast}\leq \bigg(\frac{nc_2^{p^\ast}}{\omega_{n-1}}\bigg)^{\frac{1}{p^\ast}}\frac{1}{ |x|^{\frac{n}{p^\ast}}},~\forall~~x\neq 0,$$
where $c_2>0$ is again a suitable constant independent on the choice of $k$. The above inequality with the fact that $\frac{n^2j}{(n-2)p^\ast}>n$ for all $j\geq j_0+1$ implies that
\begin{align*}
 \int_{B_R^c}|u^\ast_k|^{\frac{nj}{n-2}}\, \mathrm{d}x &\leq \omega_{n-1}\bigg(\frac{nc_2^{p^\ast}}{\omega_{n-1}}\bigg)^{\frac{nj}{(n-2)p^\ast}}\int_{R}^{\infty} r^{n-\frac{n}{p^\ast}\frac{nj}{n-2}-1}\, \mathrm{d}r\\
 & \leq \frac{\omega_{n-1}(n-2)p^\ast R^n}{n\big(n(j_0+1)-(n-2)p^\ast\big)}\bigg(\frac{nc_2^{p^\ast}}{\omega_{n-1}R^n}\bigg)^{\frac{nj}{(n-2)p^\ast}},~\forall~ j\geq j_0+1.
\end{align*}
By combining all the above information, one has 
\begin{align*}
    \int_{B_R^c}\Phi_{\ell\beta_{\gamma,n},j_0}(u^\ast_k)\, \mathrm{d}x &\leq \frac{(\ell\beta_{\gamma,n})^{j_0}} {j_0!}c_1+\frac{\omega_{n-1}(n-2)p^\ast R^n}{n\big(n(j_0+1)-(n-2)p^\ast\big)}\operatorname{\exp}\Bigg[\ell\beta_{\gamma,n}\bigg(\frac{nc_2^{p^\ast}}{\omega_{n-1}R^n}\bigg)^{\frac{n}{(n-2)p^\ast}}\Bigg]\\
    &:=C(\ell,\gamma,R,p,n)<+\infty,
\end{align*}
where $C(\ell,\gamma,R,p,n)>0$ is a constant depending only on $\ell,\gamma,R,p$ and $n$. In conclusion, we infer that
\begin{equation}\label{eq3.20} \sup_{k\in\mathbb{N}}\int_{B_R^c}\frac{\Phi_{\ell\beta_{\gamma,n},j_0}(u^\ast_k)}{|x|^\gamma}\, \mathrm{d}x\leq C(\ell,\gamma,R,p,n)<+\infty. 
\end{equation}
In addition, a direct consequence of \eqref{eq3.19} and \eqref{eq3.20} implies that
\begin{equation}\label{eq3.21}
\lim_{k\to\infty} \int_{B_R}\frac{\operatorname{\exp}(\ell\beta_{\gamma,n}|u_k^\ast(x)|^{\frac{n}{n-2}})}{|x|^\gamma}\, \mathrm{d}x = +\infty.   
\end{equation}
Further, assume that $g_k,g\in C^\infty_0(\mathbb{R}^n)$ for all $k\in\mathbb{N}$ and consider the following problems 
$$-\Delta u_k=g_k\quad\text{in}\quad \mathbb{R}^n\quad\text{and}\quad -\Delta u=g\quad\text{in}\quad \mathbb{R}^n.$$
It is obvious that $g_k,g\in L^t(\mathbb{R}^n)$ for all $k\in\mathbb{N}$ and for any $t\in\{p,\frac{n}{2}\}$. Define $g_k^\ast(x)=g_k^\sharp(\sigma_n |x|^n)$ for all $x\in\mathbb{R}^n$ and $k\in\mathbb{N}$. Hence, we obtain from \eqref{eq3.17} that $\|g_k\|_t\leq 1$ for all $k\in\mathbb{N}$ and $g_k\rightharpoonup g$ in $L^t(\mathbb{R}^n)$ as $k\to\infty$ for any $t\in\{p,\frac{n}{2}\}$. Now, by using 
Lemma \ref{lem2.2}, there exists a subsequence of $\{g_k\}_k$ still not relabeled and $h\in L^t([0,+\infty))$ such that for any $t\in\{p,\frac{n}{2}\}$, we have
\begin{align}\label{eq3.22}
    g_k^\sharp\to h~\text{a.e. in}~[0,+\infty)~~\text{as}~~ k\to\infty~~\text{and}~~\|h\|_{L^t([0,+\infty))}\geq \|g^\sharp\|_{L^t([0,+\infty))}=\|g\|_t.\end{align}
It follows from Lemma \ref{lem2.3}, \eqref{eq2.22}, and integration by parts formula that, for all $0<s_1<s_2<+\infty$, for all $k\in\mathbb{N}$, one has
\begin{align*}
  \quad\quad u_k^\sharp(s_1)-u_k^\sharp(s_2) &\leq \frac{1}{n^2\sigma_n^\frac{2}{n}}\int_{s_1}^{s_2}\frac{g_k^{\sharp\sharp}(t)}{t^{1-\frac{2}{n}}}\, \mathrm{d}t+D(n)\|g_k\|_{\frac{n}{2}} \\ & \leq \frac{1}{n^2\sigma_n^\frac{2}{n}}\int_{s_1}^{s_2}\bigg(\int_{0}^{t}g_k^{\sharp}(\eta)\,\ \mathrm{d}\eta\bigg)t^{\frac{2}{n}-2}\, \mathrm{d}t+D(n)\\
  &=\frac{1}{\beta(n,2)^{1-\frac{2}{n}}}\Bigg[\frac{1}{s_1^{1-\frac{2}{n}}}\int_{0}^{s_1}g_k^{\sharp}(t)\,\mathrm{d}t-\frac{1}{s_2^{1-\frac{2}{n}}}\int_{0}^{s_2}g_k^{\sharp}(t)\, \mathrm{d}t+\int_{s_1}^{s_2}\frac{g_k^{\sharp}(t)}{t^{1-\frac{2}{n}}}\,\mathrm{d}t\Bigg] +D(n).  
\end{align*}
Due to the H\"older's inequality and \eqref{eq3.22}, we obtain for all $s\in(0,+\infty)$ that
\begin{align*}
  \frac{1}{s^{1-\frac{2}{n}}}\int_{0}^{s}g_k^{\sharp}(t)\, \mathrm{d}t\leq \|g_k^{\sharp}\|^{\frac{n}{2}}_{L^{\frac{n}{2}}([0,+\infty))}=\|g_k\|^{\frac{n}{2}}_{\frac{n}{2}}\leq 1,~\forall~ k\in\mathbb{N}.
\end{align*}
Moreover, we deduce from the above inequalities that there exists a suitable constant $\overset{\sim}{D}(n)>0$ which depends on $n$ but is independent of the choice of $k$ and there holds
\begin{align}\label{eq3.23}
  u_k^\sharp(s_1)-u_k^\sharp(s_2)\leq  \frac{1}{\beta(n,2)^{1-\frac{2}{n}}} \int_{s_1}^{s_2}\frac{g_k^{\sharp}(t)}{t^{1-\frac{2}{n}}}\,\ \mathrm{d}t+\overset{\sim}{D}(n),~\forall~ k\in\mathbb{N}\quad\text{and}\quad 0<s_1<s_2<+\infty.
\end{align}
Define the function $w_k$ on $(0,|B_R|]$ by
\begin{align*}
    w_k(s):=\frac{1}{\beta(n,2)^{1-\frac{2}{n}}} \int_{s}^{|B_R|}\frac{g_k^{\sharp}(t)}{t^{1-\frac{2}{n}}}\,\ \mathrm{d}t,~\forall~s\in (0,|B_R|].
\end{align*}
Notice that the map $\cdot\mapsto w_k(\cdot)$ is decreasing on $(0,|B_R|]$ and $w_k(|B_R|)=0$. It follows from \eqref{eq3.17}, \eqref{eq3.23}, Lemma \ref{lem2.4} and Lemma \ref{lem2.5} that
\begin{align*}
    u_k^\sharp(s)&\leq  w_k(s)+u_k^\sharp(|B_R|)+\overset{\sim}{D}(n)\leq  w_k(s)+\overset{\sim}{D}(p,n,R),~\forall~ k\in\mathbb{N}\quad\text{and}\quad s\in(0,|B_R|), 
\end{align*}
where $\overset{\sim}{D}(p,n,R)>0$ is a constant depending only on $p$, $n$ and $R$. Choose $\eta>0$ sufficiently small enough such that $\bar{\ell}=(1+\eta)\ell<L_n(u)$, then from the above inequality together with $\eqref{eq3.6}$ yields
\begin{align*}
  | u_k^\sharp(s)|^{\frac{n}{n-2}}\leq(1+\eta)| w_k(s)|^{\frac{n}{n-2}}+C_\eta \overset{\sim}{D}(p,n,R)^{\frac{n}{n-2}},~\forall~ k\in\mathbb{N}\quad\text{and}\quad s\in(0,|B_R|).
\end{align*}
This together with the properties of the Schwarz symmetrization ensures that
\begin{align*}
    \int_{B_R}\frac{\operatorname{\exp}(\ell\beta_{\gamma,n}|u_k^\ast(x)|^{\frac{n}{n-2}})}{|x|^\gamma}\,& \mathrm{d}x =\int_{0}^{|B_R|}\Big(\frac{\sigma_n}{s}\Big)^{\frac{\gamma}{n}}\operatorname{\exp}(\ell\beta_{\gamma,n}|u_k^\sharp(s)|^{\frac{n}{n-2}})\, \mathrm{d}s\\
    &\leq \operatorname{exp}(\ell\beta_{\gamma,n}C_\eta \overset{\sim}{D}(p,n,R)^{\frac{n}{n-2}})\int_{0}^{|B_R|}\Big(\frac{\sigma_n}{s}\Big)^{\frac{\gamma}{n}}\operatorname{\exp}(\bar{\ell}\beta_{\gamma,n}|w_k(s)|^{\frac{n}{n-2}})\,\mathrm{d}s.
\end{align*}
It follows from the above inequality and \eqref{eq3.21} that
\begin{align}\label{eq3.94}
    \lim_{k\to\infty}\int_{0}^{|B_R|}\Big(\frac{\sigma_n}{s}\Big)^{\frac{\gamma}{n}}\operatorname{\exp}(\bar{\ell}\beta_{\gamma,n}|w_k(s)|^{\frac{n}{n-2}})\,\mathrm{d}s=+\infty.
\end{align}
In virtue of the Hölder's inequality and \eqref{eq3.22}, for all $s\in(0,|B_R|)$, we have
\begin{equation}\label{eq3.95}
\begin{aligned}
    w_k(s) & \leq \bigg(\int_{0}^{|B_R|}|g_k^\sharp(t)|^{\frac{n}{2}}\, \mathrm{d}t\bigg)^{\frac{2}{n}} \bigg(\int_{s}^{|B_R|}\frac{1}{\beta(n,2)t}\, \mathrm{d}t\bigg)^{1-\frac{2}{n}} \\
    &\leq \|g^\ast_k\|_{\frac{n}{2}}\bigg(\frac{1}{\beta(n,2)}\ln\bigg(\frac{|B_R|}{s}\bigg)\bigg)^{1-\frac{2}{n}} \leq\bigg(\frac{1}{\beta(n,2)}\ln\bigg(\frac{|B_R|}{s}\bigg)\bigg)^{1-\frac{2}{n}},~\forall~ k\in\mathbb{N}.
\end{aligned}
\end{equation}
Next, we claim that for any $\ell^\prime\in(\bar{\ell},L_n(u)),~k_0\in\mathbb{N}$ and $s_0\in(0,|B_R|)$, there exist $k\in\mathbb{N}$ and $s\in(0,s_0)$ satisfying
\begin{align*}
   w_k(s)\geq \bigg(\frac{1}{\ell^\prime\beta(n,2)}\ln\bigg(\frac{|B_R|}{s}\bigg)\bigg)^{1-\frac{2}{n}},~\forall~k>k_0.  
\end{align*}
Indeed, if not, let there exists $k_0\in\mathbb{N}$ and $s_0\in(0,|B_R|)$ be such that for some $\ell^\prime\in(\bar{\ell},L_n(u))$, we have
\begin{align}\label{eq3.25}
  w_k(s)< \bigg(\frac{1}{\ell^\prime\beta(n,2)}\ln\bigg(\frac{|B_R|}{s}\bigg)\bigg)^{1-\frac{2}{n}},~\forall~ k\geq k_0\quad\text{and}\quad s\in(0,s_0).    
\end{align}
By using \eqref{eq3.95} and \eqref{eq3.25}, for sufficiently large $k$, we have
\begin{align*}
\int_{0}^{|B_R|}\Big(\frac{\sigma_n}{s}\Big)^{\frac{\gamma}{n}}&\operatorname{\exp}(\bar{\ell}\beta_{\gamma,n}|w_k(s)|^{\frac{n}{n-2}})\,\mathrm{d}s \\
&=\bigg(\int_{0}^{s_0}+\int_{s_0}^{|B_R|}\bigg)\Big(\frac{\sigma_n}{s}\Big)^{\frac{\gamma}{n}}\operatorname{\exp}(\bar{\ell}\beta_{\gamma,n}|w_k(s)|^{\frac{n}{n-2}})\,\mathrm{d}s \\
    & \leq \int_{0}^{s_0} \Big(\frac{\sigma_n}{s}\Big)^{\frac{\gamma}{n}}\bigg(\frac{|B_R|}{s}\bigg)^{(1-\frac{\gamma}{n})\frac{\bar{\ell}}{\ell^\prime}}\, \mathrm{d}s + \int_{s_0}^{|B_R|}\Big(\frac{\sigma_n}{s}\Big)^{\frac{\gamma}{n}}\bigg(\frac{|B_R|}{s}\bigg)^{(1-\frac{\gamma}{n})\bar{\ell}}\, \mathrm{d}s<+\infty.    
\end{align*}
In view of the above inequality and \eqref{eq3.94}, we arrive at a contradiction. This completes the proof of our claim. Thus, there exists a subsequence $\{s_k\}_k\subset(0,s_0)$ with $s_k\leq\frac{1}{k}$ for all $k\in\mathbb{N}$ such that we can assume
\begin{align}\label{eq3.26}
   w_k(s_k)\geq \bigg(\frac{1}{\ell^\prime\beta(n,2)}\ln\bigg(\frac{|B_R|}{s_k}\bigg)\bigg)^{1-\frac{2}{n}}.  
\end{align}
For $L>0$ and $v\geq 0$, we define 
$$ H^L(v):=\min\{v,L\} ~\quad\text{and}\quad H_L(v):=v-H^L(v).$$
In view of \eqref{eq3.26}, one can easily see that $w_k(s_k)\nearrow +\infty$ as $k\to\infty$ and also we have $w_k(|B_R|)=0$. This shows that we can find a subsequence of $\{s_k\}_k$ still not relabeled and the existence of $\xi_k\in (s_k,|B_R|)$ satisfying $w_k(\xi_k)=L$ for each $k\in \mathbb{N}$. In addition, since the map $\cdot\mapsto g_k^\sharp(\cdot)$ is always nonnegative and nonincreasing, therefore by using the
H\"older's inequality, one has
\begin{align*}
   \frac{1}{\beta(n,2)^{1-\frac{2}{n}}} \int_{s_k}^{\xi_k}\frac{g_k^{\sharp}(t)}{t^{1-\frac{2}{n}}}\, \mathrm{d}t\leq  \frac{n g_k^{\sharp}(s_k)}{2\beta(n,2)^{1-\frac{2}{n}}}\Big(|B_R|^{\frac{2}{n}}-s_k^{\frac{2}{n}}\Big). 
\end{align*}
The above inequality together with \eqref{eq3.26} implies that
\begin{align}\label{eq3.27}
 \bigg(\frac{1}{\ell^\prime\beta(n,2)}\ln\bigg(\frac{|B_R|}{s_k}\bigg)\bigg)^{1-\frac{2}{n}}-L\leq  w_k(s_k)- w_k(\xi_k)\leq  \frac{n g_k^{\sharp}(s_k)}{2\beta(n,2)^{1-\frac{2}{n}}}\Big(|B_R|^{\frac{2}{n}}-s_k^{\frac{2}{n}}\Big),~\forall~k\in\mathbb{N}. 
\end{align}
Now, from \eqref{eq3.27}, we infer that $g_k^\sharp(s_k)\nearrow +\infty$ as $k\to\infty$ and thus, there exists $k_0\in\mathbb{N}$ such that $g_k^\sharp(s_k)>L$ for all $k\geq k_0$. Consequently, up to a subsequence of $\{s_k\}_k$ still not relabeled, we can find  $\eta_k\in (s_k,|B_R|)$ satisfying $g_k^\sharp(\eta_k)=L$ for each $k\in \mathbb{N}$ and $g_k^\sharp(\eta)< L$ for all $\eta>\eta_k$. For all $k\geq k_0$, we set $\theta_k:=\min\{\eta_k,\xi_k\}$, then one sees from \eqref{eq3.26} that
\begin{equation}\label{eq3.28}
\begin{aligned}
 \bigg(\frac{1}{\ell^\prime\beta(n,2)}&\ln\bigg(\frac{|B_R|}{s_k}\bigg)\bigg)^{1-\frac{2}{n}}-L \\ &\leq  w_k(s_k)- w_k(\xi_k)=  \frac{1}{\beta(n,2)^{1-\frac{2}{n}}} \int_{s_k}^{\xi_k}\frac{g_k^{\sharp}(t)}{t^{1-\frac{2}{n}}}\, \mathrm{d}t \\
 &=\frac{1}{\beta(n,2)^{1-\frac{2}{n}}}\Bigg[\int_{s_k}^{\theta_k}\frac{g_k^{\sharp}(t)-L}{t^{1-\frac{2}{n}}}\, \mathrm{d}t+\int_{\theta_k}^{\xi_k}\frac{g_k^{\sharp}(t)-L}{t^{1-\frac{2}{n}}}\, \mathrm{d}t  +L \int_{s_k}^{\xi_k}\frac{1}{t^{1-\frac{2}{n}}}\, \mathrm{d}t\Bigg] \\
 & \leq \frac{1}{\beta(n,2)^{1-\frac{2}{n}}}\Bigg[\int_{s_k}^{\theta_k}\frac{g_k^{\sharp}(t)-L}{t^{1-\frac{2}{n}}}\, \mathrm{d}t+\frac{nL}{2} \Big(|B_R|^{\frac{2}{n}}-s_k^{\frac{2}{n}}\Big)\Bigg].
\end{aligned}
\end{equation}
Moreover, the application of the H\"older's inequality yields
\begin{equation}\label{eq3.29}
\begin{aligned}
   \frac{1}{\beta(n,2)^{1-\frac{2}{n}}}\int_{s_k}^{\theta_k}\frac{g_k^{\sharp}(t)-L}{t^{1-\frac{2}{n}}}\, \mathrm{d}t
   &\leq \bigg(\int_{s_k}^{\theta_k}|g_k^{\sharp}(t)-L|^{\frac{n}{2}} \, \mathrm{d}t\bigg)^{\frac{2}{n}} \bigg(\int_{s_k}^{\theta_k}\frac{1}{\beta(n,2)t}\, \mathrm{d}t\bigg)^{1-\frac{2}{n}}\\ 
   &\leq \bigg(\int_{0}^{\eta_k}|g_k^{\sharp}(t)-L|^{\frac{n}{2}} \, \mathrm{d}t\bigg)^{\frac{2}{n}} \bigg(\int_{s_k}^{|B_R|}\frac{1}{\beta(n,2)t}\, \mathrm{d}t\bigg)^{1-\frac{2}{n}} \\
   &\leq \bigg(\int_{0}^{|B_R|}|H_L(g_k^\sharp(t))|^{\frac{n}{2}}\, \mathrm{d}t\bigg)^{\frac{2}{n}} \bigg(\frac{1}{\beta(n,2)}\ln\bigg(\frac{|B_R|}{s_k}\bigg)\bigg)^{1-\frac{2}{n}}\\ 
   &\leq \|H_L(g^\sharp_k)\|_{L^{\frac{n}{2}}([0,+\infty))}\bigg(\frac{1}{\beta(n,2)}\ln\bigg(\frac{|B_R|}{s_k}\bigg)\bigg)^{1-\frac{2}{n}}.
\end{aligned}
\end{equation}
If possible, let us choose $\overset{\sim}{\ell}\in(\ell^\prime,L_n(u))$. It is obvious that $\ln\Big(\frac{|B_R|}{s_k}\Big)\to +\infty$ as $k\to\infty$. Hence, we obtain from \eqref{eq3.28} and \eqref{eq3.29} as $k\to\infty$ that
\begin{equation}\label{eq3.30}
\begin{aligned}
    \Bigg(\frac{1}{\overset{\sim}{\ell}}\Bigg)^{\frac{n-2}{2}} &\notag\leq \|H_L(g^\sharp_k)\|^{\frac{n}{2}}_{L^{\frac{n}{2}}([0,+\infty))}+o_k(1)\\
    &\leq \|H_L(g^\sharp_k)\|^{\frac{n}{2}}_{L^{\frac{n}{2}}([0,+\infty))}+\|H_L(g^\sharp_k)\|^{\frac{n}{2}}_{L^{p}([0,+\infty))}+o_k(1).
\end{aligned}
\end{equation}
It may be deduced from direct calculations that
\begin{equation}\label{eq3.31}
\begin{aligned}
    1 &=\|u_k\|^\frac{n}{2}=\|g_k\|^\frac{n}{2}_\frac{n}{2}+\|g_k\|^\frac{n}{2}_p \\ 
    &=\|g^\ast_k\|^\frac{n}{2}_\frac{n}{2}+\|g^\ast_k\|^\frac{n}{2}_p =\|g^\sharp_k\|^\frac{n}{2}_{L^\frac{n}{2}([0,+\infty))}+ \|g^\sharp_k\|^\frac{n}{2}_{L^p([0,+\infty))} \\ 
    & \geq \|H_L(g^\sharp_k)\|^\frac{n}{2}_{L^\frac{n}{2}([0,+\infty))}+ \|H_L(g^\sharp_k)\|^\frac{n}{2}_{L^p([0,+\infty))} \\ 
    &\quad + \|g^\sharp_k-H_L(g^\sharp_k)\|^\frac{n}{2}_{L^\frac{n}{2}([0,+\infty))}+ \|g^\sharp_k-H_L(g^\sharp_k)\|^\frac{n}{2}_{L^p([0,+\infty))}.
\end{aligned}
\end{equation}
In view of \eqref{eq3.30} and \eqref{eq3.31}, we get as $k\to\infty$ that
\begin{align}\label{eq3.32}
    1-\Bigg(\frac{1}{\overset{\sim}{\ell}}\Bigg)^{\frac{n-2}{2}}\geq \|g^\sharp_k-H_L(g^\sharp_k)\|^\frac{n}{2}_{L^\frac{n}{2}([0,+\infty))}+ \|g^\sharp_k-H_L(g^\sharp_k)\|^\frac{n}{2}_{L^p([0,+\infty))}+o_k(1).
\end{align}
By the definition of $H_L$ and \eqref{eq3.22}, for a.a. $s\in[0,+\infty)$, one has
$$ h(s)-H_L(h(s))\to h(s)\quad\text{as}\quad L\to\infty$$ and $$ g^\sharp_k(s)-H_L(g^\sharp_k(s))\to  h(s)-H_L(h(s)) \quad\text{as}\quad k\to\infty. $$
Letting $k\to\infty$ in \eqref{eq3.32} and applying Fatou's lemma, we have
\begin{align*}
     1-\Bigg(\frac{1}{\overset{\sim}{\ell}}\Bigg)^{\frac{n-2}{2}}\geq \|h-H_L(h)\|^\frac{n}{2}_{L^\frac{n}{2}([0,+\infty))}+ \|h-H_L(h)\|^\frac{n}{2}_{L^p([0,+\infty))}.
\end{align*}
Despite this, again by Fatou's lemma, we obtain from the above inequality as $L\to\infty$ that
\begin{align*}
     1-\Bigg(\frac{1}{\overset{\sim}{\ell}}\Bigg)^{\frac{n-2}{2}} &\geq \|h\|^\frac{n}{2}_{L^\frac{n}{2}([0,+\infty))}+ \|h\|^\frac{n}{2}_{L^p([0,+\infty))}\geq \|g^\sharp\|^\frac{n}{2}_{L^\frac{n}{2}([0,+\infty))}+ \|g^\sharp\|^\frac{n}{2}_{L^p([0,+\infty))} \\ \notag
&=\|g^\ast\|^\frac{n}{2}_\frac{n}{2}+\|g^\ast\|^\frac{n}{2}_p  =\|g\|^\frac{n}{2}_\frac{n}{2}+\|g\|^\frac{n}{2}_p 
=\|u\|^\frac{n}{2}=1-\bigg(\frac{1}{L_n(u)}\bigg)^{\frac{n-2}{2}},
\end{align*}
which contradicts $\overset{\sim}{\ell}\in(\ell^\prime,L_n(u))$ and thus, we complete the proof of the first part of Theorem \ref{thm1.2}. To complete the proof, we only have to prove the sharpness of $L_n(u)$. For this, we shall have to contruct a sequence $\{u_k\}_k\subset E$ and a function $u\in E$ satisfying
\begin{align}\label{eq3.33}
    \|u_k\|=1,\quad u_k\rightharpoonup u\quad\text{in}\quad E\quad\text{as}\quad k\to\infty,\quad\|u\|=\delta<1
\end{align}
but 
\begin{align*}
\int_{\mathbb{R}^n}\frac{\Phi_{\ell\beta_{\gamma,n},j_0}(u_k)}{|x|^\gamma}\,\mathrm{d}x\nearrow +\infty\quad\text{as}\quad k\to\infty,~\forall~ \ell> L_n(u).
\end{align*}
Let $\{\xi_k\}_k\subset E$ be a sequence of radial functions defined as in \eqref{eq3.4}. It follows from the standard arguments that
$$ \xi_k\rightharpoonup 0 \quad\text{in}\quad E,\quad \|\Delta \xi_k\|_{\frac{n}{2}}^{\frac{n}{2}}\to 1 \quad \text{and}\quad \|\Delta \xi_k\|^p_p\to 0 \quad\text{as}\quad k\to\infty.$$
Take a radial function $u\in C^\infty_0(\mathbb{R}^n)$ such that $\|u\|:=\delta<1$ and $\text{supp}(u)\subset\{x\in \mathbb{R}^n:~2<|x|<3\}$. Define $v_k:=u+(1-\delta^{\frac{n}{2}})^{\frac{2}{n}}\xi_k$ for all $k\in\mathbb{N}$. Notice that $\Delta u$ and $\Delta \xi_k$ have disjoint supports. It follows at once that
$$\|\Delta v_k\|_{\frac{n}{2}}^{\frac{n}{2}}=\|\Delta u\|_{\frac{n}{2}}^{\frac{n}{2}}+ (1-\delta^{\frac{n}{2}})+\zeta_k\quad\text{and}\quad \|\Delta v_k\|^p_p=\|\Delta u\|^p_p+\zeta_k,$$
where $\zeta_k\to 0$ as $k\to\infty$. Hence, it is easy to see that $\|v_k\|=1+\zeta_k$. Denote $u_k:=\frac{v_k}{1+\zeta_k}$. This implies the validity of \eqref{eq3.33}. Moreover, for any $\ell>L_n(u)$, there exists $\varepsilon>0$ such that one can assume $\ell=(1+\varepsilon)L_n(u)$ and thus, we have 
\begin{align*}
   \int_{\mathbb{R}^n}&\frac{\Phi_{\ell\beta_{\gamma,n},j_0}(u_k)}{|x|^\gamma} \,\mathrm{d}x\geq \int_{\{x\in\mathbb{R}^n\colon~|x|< k^{-\frac{1}{n}} \}}\frac{\Phi_{\ell \beta_{\gamma,n},j_0}\big((1+\zeta_k)^{-1}(1-\delta^{\frac{n}{2}})^{\frac{2}{n}}\xi_k\big)}{|x|^\gamma}\,\mathrm{d}x \\
    &=\omega_{n-1} \int_{0}^{k^{-\frac{1}{n}}}\frac{\Phi_{\ell \beta_{\gamma,n},j_0}\bigg((1+\zeta_k)^{-1}(1-\delta^{\frac{n}{2}})^{\frac{2}{n}}\bigg(\Big(\frac{\ln k}{\beta(n,2)}\Big)^{1-\frac{2}{n}}+(1-k^{\frac{2}{n}}r^2)\frac{n \beta(n,2)^{\frac{2}{n}-1}}{2(\ln k)^{\frac{2}{n}}}\bigg)\bigg)}{r^{\gamma-n+1}}\,\mathrm{d}r\\
    &\geq \frac{\omega_{n-1}}{(n-\gamma)k^{\frac{n-\gamma}{n}}}~\Phi_{\ell \beta_{\gamma,n},j_0}\Bigg((1+\zeta_k)^{-1}(1-\delta^{\frac{n}{2}})^{\frac{2}{n}}\bigg(\frac{\ln k}{\beta(n,2)}\bigg)^{1-\frac{2}{n}}\Bigg) \\
    &\sim \frac{\omega_{n-1}}{(n-\gamma)k^{\frac{n-\gamma}{n}}}~\operatorname{\exp}\Bigg(\ell \beta_{\gamma,n} \Bigg|(1+\zeta_k)^{-1}(1-\delta^{\frac{n}{2}})^{\frac{2}{n}}\bigg(\frac{\ln k}{\beta(n,2)}\bigg)^{1-\frac{2}{n}}\Bigg|^{\frac{n}{n-2}}\Bigg)\\
    & =\frac{\omega_{n-1}}{(n-\gamma)}~\operatorname{\exp}\Bigg[\bigg(1-\frac{\gamma}{n}\bigg)\ln k\Bigg(\frac{1+\varepsilon}{(1+\zeta_k)^{\frac{n}{n-2}}}-1\Bigg)\Bigg]\to+\infty\quad\text{as}\quad k\to\infty.
\end{align*}
This finishes the proof of Theorem \ref{thm1.2}.
\end{proof}
\section{Asymptotic behavior of subcritical sharp singular Adams' type Inequality: Proof of Theorem \ref{thm5.1}}\label{sec5}
Basically, the purpose of this section is to prove Theorem \ref{thm5.1}. More precisely, we shall demonstrate the asymptotic behavior of the supremum $\text{ATSC}(\ell, \gamma)$ for the subcritical sharp singular Adams' type inequality. To prove this, we establish the following crucial result.
\begin{lemma}\label{lem7.1}
   There holds
   $$\text{ATSC}(\ell,\gamma)=\sup_{u\in E,~\|\Delta u\|_{\frac{n}{2}}\leq 1,~ \|\Delta u\|_p=1}\int_{\mathbb{R}^n}\frac{\Phi_{\ell(1-\frac{\gamma}{n}),j_0}(u)}{|x|^\gamma}\,\mathrm{d}x. $$
\end{lemma}
\begin{proof}
  Let $u\in E$ be such that $\|\Delta u\|_{\frac{n}{2}}\leq 1$. Define the function $w$ as follows $$ w(x)=u(\lambda x)\quad\text{with}\quad \lambda=\|\Delta u\|_p^{\frac{p}{n-2p}}>0. $$
  By direct calculations, one has
  $$\Delta w(x)=\lambda^2 \Delta u(\lambda x),\quad\|\Delta w\|_{\frac{n}{2}}=\|\Delta u\|_{\frac{n}{2}}\quad\text{and}\quad\|\Delta w\|_p=\lambda^{-\frac{n-2p}{p}}\|\Delta u\|_p=1. $$
  It follows that
  \begin{align*}
    \int_{\mathbb{R}^n}\frac{\Phi_{\ell(1-\frac{\gamma}{n}),j_0}(w)}{|x|^\gamma}\,\mathrm{d}x&=\int_{\mathbb{R}^n}\frac{\Phi_{\ell(1-\frac{\gamma}{n}),j_0}(u(\lambda x))}{|x|^\gamma}\,\mathrm{d}x=\frac{1}{\lambda^{n-\gamma}}\int_{\mathbb{R}^n}\frac{\Phi_{\ell(1-\frac{\gamma}{n}),j_0}(u)}{|x|^\gamma}\,\mathrm{d}x\\
    &=\|\Delta u\|_p^{-p^\ast(1-\frac{\gamma}{n})}\int_{\mathbb{R}^n}\frac{\Phi_{\ell(1-\frac{\gamma}{n}),j_0}(u)}{|x|^\gamma}\,\mathrm{d}x.
  \end{align*}
Hence, we infer from the above equality that the proof of Lemma \ref{lem7.1} is completed.
\end{proof}
Due to Lemma \ref{lem7.1}, one can assume $\|\Delta u\|_p=1$ in the subcritical sharp singular Adams' type inequality. Now, we finish this section by  proving Theorem \ref{thm5.1} as follows.

\begin{proof}[\bf{Proof of Theorem \ref{thm5.1}}.]
 By exploiting the denseness of $C_0^\infty(\mathbb{R}^n)$ in $E$, $\Phi_{\alpha,j_0}(s)= \Phi_{\alpha,j_0}(|s|)$ and Lemma \ref{lem7.1}, without loss of generality, one can assume $u\in C_0^\infty(\mathbb{R}^n)\setminus\{0\}$, $u\geq 0$, $\|\Delta u\|_{\frac{n}{2}}\leq 1$ and $\|\Delta u\|_p=1$. Define the set $S$ as follows
 $$S:=\Bigg\{x\in \mathbb{R}^n\colon~u(x)>\Bigg(1-\bigg(\frac{\ell}{\beta(n,2)}\bigg)^{\frac{n-2}{2}}\Bigg)^{\frac{2}{n}}\Bigg\}. $$
Now, using again the denseness of $C_0^\infty(\mathbb{R}^n)$ in $D^{2,p}(\mathbb{R}^n)$ and the embedding $D^{2,p}(\mathbb{R}^n)\hookrightarrow L^{p^\ast}(\mathbb{R}^n)$, one sees that $S$ is a bounded domain. Indeed, we have
\begin{align*}
    |S|=\int_S 1 \,\mathrm{d}x &\leq \|u\|^{p^\ast}_{p^\ast}\Bigg(1-\bigg(\frac{\ell}{\beta(n,2)}\bigg)^{\frac{n-2}{2}}\Bigg)^{-\frac{2p^\ast}{n}}\leq C(p,n)\Bigg(1-\bigg(\frac{\ell}{\beta(n,2)}\bigg)^{\frac{n-2}{2}}\Bigg)^{-\frac{2p^\ast}{n}},
\end{align*}
where $C(p,n)>0$ is a constant depending only on $p$ and $n$. Next, join us in writing
$$ \int_{\mathbb{R}^n}\frac{\Phi_{\ell(1-\frac{\gamma}{n}),j_0}(u)}{|x|^\gamma}\,\mathrm{d}x= \int_{S}\frac{\Phi_{\ell(1-\frac{\gamma}{n}),j_0}(u)}{|x|^\gamma}\,\mathrm{d}x+ \int_{\mathbb{R}^n\setminus S}\frac{\Phi_{\ell(1-\frac{\gamma}{n}),j_0}(u)}{|x|^\gamma}\,\mathrm{d}x=:I_1+I_2. $$
To evaluate $I_2$, we first notice that $\mathbb{R}^n\setminus S\subset\{x\in \mathbb{R}^n\colon~u(x)\leq 1\}$. By using the embedding $D^{2,p}(\mathbb{R}^n)\hookrightarrow L^{p^\ast}(\mathbb{R}^n)$ and the fact that $\|\Delta u\|_p=1$, we get
\begin{equation}\label{eq7.1}
\begin{aligned}
I_2 & \leq \int_{\{x\in \mathbb{R}^n\colon~ u(x)\leq 1\}}\frac{1}{|x|^\gamma}\sum_{j=j_0}^\infty \frac{\big(\ell(1-\frac{\gamma}{n})\big)^j}{j!}|u|^{\frac{nj}{n-2}}\,\mathrm{d}x\leq \int_{\{x\in \mathbb{R}^n\colon~ u(x)\leq 1\}}\frac{1}{|x|^\gamma}\sum_{j=j_0}^\infty \frac{\big(\ell(1-\frac{\gamma}{n})\big)^j}{j!}|u|^{p^\ast}\,\mathrm{d}x \\ 
 & \leq \operatorname{\exp}\bigg(\ell\bigg(1-\frac{\gamma}{n}\bigg)\bigg)\Bigg[ \int_{\{x\in \mathbb{R}^n\colon~ u(x)\leq 1,~|x|<1\}}\frac{|u|^{p^\ast}}{|x|^\gamma}\,\mathrm{d}x+  \int_{\{x\in \mathbb{R}^n\colon~u(x)\leq1,~|x|\geq1\}}\frac{|u|^{p^\ast}}{|x|^\gamma}\,\mathrm{d}x\Bigg]\\ 
& \leq \operatorname{\exp}\bigg(\ell\bigg(1-\frac{\gamma}{n}\bigg)\bigg)\Bigg[ \int_{\{ x\in \mathbb{R}^n\colon~|x|<1\}}\frac{1}{|x|^\gamma}\,\mathrm{d}x+ \int_{\{ x\in \mathbb{R}^n\colon~|x|\geq1\}}|u|^{p^\ast}\,\mathrm{d}x\Bigg] \\
&\leq C(\gamma,p,n),
\end{aligned}
\end{equation}
where $C(\gamma,p,n)>0$ is a constant depending only on  $\gamma$, $p$ and $n$. Finally, to evaluate $I_1$, we define
$$w(x):=u(x)-\bigg(1-\Big(\frac{\ell}{\beta(n,2)}\Big)^{\frac{n-2}{2}}\bigg)^{\frac{2}{n}}\quad\text{in}\quad S. $$
It follows that $w\in W^{2,\frac{n}{2}}_\mathcal{N}(S)$ and $\|\Delta w\|_{\frac{n}{2}}\leq 1$. By employing \eqref{eq3.6} with $\eta=\frac{\beta(n,2)}{\ell}-1$, we obtain
\begin{align*}
 |u|^{\frac{n}{n-2}} &\leq\Bigg(|w|+\Bigg(1-\bigg(\frac{\ell}{\beta(n,2)}\bigg)^{\frac{n-2}{2}}\Bigg)^{\frac{2}{n}}\Bigg)^{\frac{n}{n-2}}\leq (1+\eta)|w|^{\frac{n}{n-2}}+C_\eta \Bigg(1-\bigg(\frac{\ell}{\beta(n,2)}\bigg)^{\frac{n-2}{2}}\Bigg)^{\frac{2}{n-2}}\\
 &=\frac{\beta(n,2)}{\ell}|w|^{\frac{n}{n-2}}+\Bigg(1-\bigg(\frac{\ell}{\beta(n,2)}\bigg)^{\frac{n-2}{2}}\Bigg)^{-\frac{2}{n-2}}\Bigg(1-\bigg(\frac{\ell}{\beta(n,2)}\bigg)^{\frac{n-2}{2}}\Bigg)^{\frac{2}{n-2}}\\
 &=\frac{\beta(n,2)}{\ell}|w|^{\frac{n}{n-2}}+1\quad\text{in}\quad S.
\end{align*}
By gathering all the above information, we obtain from the singular Adams' inequality \cite[Theorem 1.2]{Lam-Lu-2012@@@} with homogeneous Navier boundary conditions that
\begin{equation}\label{eq7.2}
\begin{aligned}
 I_1 &  \leq \int_S\frac{\operatorname{\exp}\big(\ell(1-\frac{\gamma}{n}) |u|^{\frac{n}{n-2}}\big)}{|x|^\gamma}\,\mathrm{d}x\leq \operatorname{\exp}\bigg(\ell\bigg(1-\frac{\gamma}{n}\bigg)\bigg)\int_S\frac{\operatorname{\exp}\big(\beta_{\gamma,n} |w|^{\frac{n}{n-2}}\big)}{|x|^\gamma}\,\mathrm{d}x \\ 
 &\leq \operatorname{\exp}\bigg(\ell\bigg(1-\frac{\gamma}{n}\bigg)\bigg)C(\gamma,n)|S|^{1-\frac{\gamma}{n}}\leq C(\gamma,p,n)\bigg(1-\bigg(\frac{\ell}{\beta(n,2)}\bigg)^{\frac{n-2}{2}}\bigg)^{-\frac{2p^\ast}{n}\big(1-\frac{\gamma}{n}\big)},
\end{aligned}
\end{equation}
where $C(\gamma,p,n)>0$ is a constant depending only on $\gamma$, $p$ and $n$. Hence, we obtain from \eqref{eq7.1} and \eqref{eq7.2} that
$$\text{ATSC}(\ell,\gamma)\leq C(\gamma,p,n)\bigg(1-\bigg(\frac{\ell}{\beta(n,2)}\bigg)^{\frac{n-2}{2}}\bigg)^{-\frac{2p^\ast}{n}\big(1-\frac{\gamma}{n}\big)}. $$
To prove the sharpness of $\beta(n,2)$, that is, $\text{ATSC}(\beta(n,2),\gamma)=+\infty$, inspired by Adams \cite{Adams-1988}, we consider $f\in C^\infty([0,1])$ satisfying
$$ f(0)=f^\prime(0)=f^{\prime\prime}(0)=0;\quad f(1)=f^\prime(1)=1\quad\text{and}\quad f^{\prime\prime}(1)=0.$$
Fix $0<\varepsilon<\frac{1}{2}$ and define the function
$$\Psi(s)=\begin{cases}
    \varepsilon f(\frac{s}{\varepsilon})\quad\quad\quad\quad\quad\quad\text{if}\quad 0<s\leq\varepsilon,\\
    s\quad\quad\quad\quad\quad\quad\quad\quad\text{if}\quad \varepsilon<s\leq1-\varepsilon,\\
    1-\varepsilon f(\frac{1-s}{\varepsilon})\quad\quad\quad\text{if}\quad 1-\varepsilon<s\leq 1,\\
    1 \quad\quad\quad\quad\quad\quad\quad\quad\text{if}\quad s>1,
\end{cases} $$
with $\text{supp}(\Psi)\subset(0,+\infty)$. Let $v\colon~(0,+\infty)\to\mathbb{R}$ be such that
\begin{align}\label{eq7.3}
 v(s)=\Psi\bigg(\frac{\log\frac{1}{s}}{\log\frac{1}{r}}\bigg),~\forall~s>0\quad\text{and}\quad r\in(0,1).   
\end{align}
It follows that $v(|x|)\in E$ and $v(|x|)=1$ for all $x\in B_r$. Moreover, one can easily obtain
$$ \Delta v(|x|)=(-1)\frac{\alpha(n,2)}{\omega_{n-1}}\Psi^\prime\bigg(\frac{\log\frac{1}{s}}{\log\frac{1}{r}}\bigg)\Big(\log\frac{1}{r}\Big)^{-1}|x|^{-2}+O\Big(\Big(\log\frac{1}{r}\Big)^{-2}\Big)|x|^{-2}, $$
where $\alpha(n,2)=\omega_{n-1}(n-2)$ and the big $O$ term involves terms that are constant multiple of $\Psi^{\prime\prime}\Big(\frac{\log\frac{1}{s}}{\log\frac{1}{r}}\Big)\big(\log\frac{1}{r}\big)^{-2}$. In addition, by direct computations, we also have
$$\int_{\mathbb{R}^n}|\Delta v(|x|)|^{\frac{n}{2}}\,\mathrm{d}x=\omega_{n-1}^{1-\frac{n}{2}}\alpha(n,2)^{\frac{n}{2}}\Big(\log\frac{1}{r}\Big)^{1-\frac{n}{2}}\Theta_{\frac{n}{2}}, $$
where
\begin{align*}
  \Theta_{\frac{n}{2}}&=\Big(\log\frac{1}{r}\Big)^{-1}\int_{0}^{\infty} \Bigg|(-1)\Psi^\prime\bigg(\frac{\log\frac{1}{s}}{\log\frac{1}{r}}\bigg)+O\Big(\Big(\log\frac{1}{r}\Big)^{-1}\Big)\Bigg|^{\frac{n}{2}}\,\frac{\mathrm{d}s}{s}\\
  &=\Big(\log\frac{1}{r}\Big)^{-1}\Bigg(\int_{0}^{r}+\int_{r}^{r^{1-\varepsilon}}+\int_{r^{1-\varepsilon}}^{r^\varepsilon}+\int_{r^\varepsilon}^{1}+\int_{1}^{\infty}\Bigg) \Bigg|(-1)\Psi^\prime\bigg(\frac{\log\frac{1}{s}}{\log\frac{1}{r}}\bigg)+O\Big(\Big(\log\frac{1}{r}\Big)^{-1}\Big)\Bigg|^{\frac{n}{2}}\,\frac{\mathrm{d}s}{s}\\
  &\leq 1+2\varepsilon\Big(\|f^\prime\|_\infty+O\Big(\Big(\log\frac{1}{r}\Big)^{-1}\Big)\Big)^{\frac{n}{2}},
\end{align*}
where we have used the following estimates
\begin{itemize}
    \item[\textnormal{(a)}] There holds $$\Big(\log\frac{1}{r}\Big)^{-1}\int_{r^{1-\varepsilon}}^{r^\varepsilon} \Bigg|(-1)\Psi^\prime\bigg(\frac{\log\frac{1}{s}}{\log\frac{1}{r}}\bigg)+O\Big(\Big(\log\frac{1}{r}\Big)^{-1}\Big)\Bigg|^{\frac{n}{2}}\,\frac{\mathrm{d}s}{s}\leq 1, $$
    \item[\textnormal{(b)}] For all $(a,b)\in\{(0,r),(1,+\infty)\}$, there holds
    $$\Big(\log\frac{1}{r}\Big)^{-1}\int_{a}^{b} \Bigg|(-1)\Psi^\prime\bigg(\frac{\log\frac{1}{s}}{\log\frac{1}{r}}\bigg)+O\Big(\Big(\log\frac{1}{r}\Big)^{-1}\Big)\Bigg|^{\frac{n}{2}}\,\frac{\mathrm{d}s}{s}=0, $$
    \item[\textnormal{(c)}] For all $(a,b)\in\{(r,r^{1-\varepsilon}),(r^{\varepsilon},1)\}$, there holds
    \begin{align*} \Big(\log\frac{1}{r}\Big)^{-1}\int_{a}^{b} \Bigg|(-1)\Psi^\prime\bigg(\frac{\log\frac{1}{s}}{\log\frac{1}{r}}\bigg)+O\Big(\Big(\log\frac{1}{r}\Big)^{-1}\Big)\Bigg|^{\frac{n}{2}}\,\frac{\mathrm{d}s}{s}\leq&~ \varepsilon\Big(\|f^\prime\|_\infty+O\Big(\Big(\log\frac{1}{r}\Big)^{-1}\Big)\Big)^{\frac{n}{2}}.\end{align*} 
\end{itemize}
Similarly, we can prove that
$$\int_{\mathbb{R}^n}|\Delta v(|x|)|^{p}\,\mathrm{d}x=\omega_{n-1}^{1-p}\alpha(n,2)^p\Big(\log\frac{1}{r}\Big)^{1-p}\Theta_p, $$
where
\begin{align*}
  \Theta_p\leq 1+2\varepsilon\Big(\|f^\prime\|_\infty+O\Big(\Big(\log\frac{1}{r}\Big)^{-1}\Big)\Big)^{p}.
\end{align*}
Define $\bar{v}\colon(0,+\infty)\to\mathbb{R}$ by
$$\bar{v}(s)=\Big(\log\frac{1}{r}\Big)^{1-\frac{2}{n}}v(s),~\forall~s>0\quad\text{and}\quad r\in(0,1). $$
It follows immediately that
$$ \bar{v}(|x|)=\Big(\log\frac{1}{r}\Big)^{1-\frac{2}{n}},~\forall~x\in B_r,\quad\|\Delta \bar{v}\|^{\frac{n}{2}}_{\frac{n}{2}}=\omega_{n-1}^{1-\frac{n}{2}}\alpha(n,2)^{\frac{n}{2}}\Theta_{\frac{n}{2}} $$
and $$\|\Delta \bar{v}\|^p_p=\Big(\log\frac{1}{r}\Big)^{1-\frac{2p}{n}}\omega_{n-1}^{1-p}\alpha(n,2)^p\Theta_p <\Big(\frac{1}{r}\Big)^{1-\frac{2p}{n}}\omega_{n-1}^{1-p}\alpha(n,2)^p\Theta_p,$$
where the last inequality is obtained by using the elementary inequality $\frac{x-1}{x}\leq\log x\leq x-1$ for all $x>0 $. In virtue of the above information, one has
\begin{align*}
    &\text{ATSC}((\beta(n,2),\gamma)\\ &\geq\cfrac{1}{\Big\|\Delta\Big(\frac{\bar{v}}{\|\Delta \bar{v}\|_{\frac{n}{2}}}\Big)\Big\|_p^{p^\ast(1-\frac{\gamma}{n})}}\int_{\{x\in \mathbb{R}^n\colon~|x|<r\}}\cfrac{\Phi_{\beta_{\gamma,n},j_0}\Big(\frac{\bar{v}}{\|\Delta \bar{v}\|_{\frac{n}{2}}}\Big)}{|x|^\gamma}\,\mathrm{d}x\\
    &>\frac{\omega_{n-1}}{n-\gamma} ~C(n,\varepsilon,r,\|f^\prime\|_\infty)~r^{(1+n)(1-\frac{\gamma}{n})}~\Phi_{\beta_{\gamma,n},j_0}\Bigg(\cfrac{\big(\log\frac{1}{r}\big)^{1-\frac{2}{n}}}{\big(\omega_{n-1}^{1-\frac{n}{2}}\alpha(n,2)^{\frac{n}{2}}\Theta_{\frac{n}{2}} \big)^{\frac{2}{n}}}\Bigg)\\
    &\gtrsim \frac{\omega_{n-1}}{n-\gamma} ~C(n,\varepsilon,r,\|f^\prime\|_\infty)~r^{1-\frac{\gamma}{n}}~\Phi_{\beta_{\gamma,n},j_0}\Bigg(\cfrac{\big(\log\frac{1}{r}\big)^{1-\frac{2}{n}}}{\big(\omega_{n-1}^{1-\frac{n}{2}}\alpha(n,2)^{\frac{n}{2}}\Theta_{\frac{n}{2}} \big)^{\frac{2}{n}}}\Bigg)\\
    &\sim  \frac{\omega_{n-1}}{n-\gamma} ~C(n,\varepsilon,r,\|f^\prime\|_\infty)~r^{1-\frac{\gamma}{n}}~\operatorname{\exp}\Bigg(\beta_{\gamma,n} \Bigg|\cfrac{\big(\log\frac{1}{r}\big)^{1-\frac{2}{n}}}{\big(\omega_{n-1}^{1-\frac{n}{2}}\alpha(n,2)^{\frac{n}{2}}\Theta_{\frac{n}{2}} \big)^{\frac{2}{n}}}\Bigg|^{\frac{n}{n-2}}\Bigg)\\
    & =\frac{\omega_{n-1}}{n-\gamma} ~C(n,\varepsilon,r,\|f^\prime\|_\infty)\operatorname{\exp}\Bigg[\bigg(1-\frac{\gamma}{n}\bigg)\ln \frac{1}{r}\Bigg(\frac{n}{\Big(1+2\varepsilon\Big(\|f^\prime\|_\infty+O\Big(\Big(\log\frac{1}{r}\Big)^{-1}\Big)\Big)^{\frac{n}{2}}\Big)^{\frac{2}{n-2}}}-1\Bigg)\Bigg]\\
    &\to +\infty\quad\text{as}\quad r\to 0^{+},~\text{if we take }~\varepsilon>0~\text{sufficiently small enough},
\end{align*}
where
$$ C(n,\varepsilon,r,\|f^\prime\|_\infty)=\Bigg(\cfrac{\big(\omega_{n-1}^{1-\frac{n}{2}}\alpha(n,2)^{\frac{n}{2}}\Theta_{\frac{n}{2}} \big)^{\frac{2}{n}}}{\big(\omega_{n-1}^{1-p}\alpha(n,2)^{p}\Theta_{p} \big)^{\frac{1}{p}}}\Bigg)^{p^\ast(1-\frac{\gamma}{n})}. $$
This shows the validity of the sharpness of $\beta(n,2)$ and thus, the claim is well-established. Next, we consider the same sequence $\{\xi_k\}_k\subset E$ that is used in the proof of Theorem \ref{thm1.1}. Notice that the following estimates hold
$$\|\Delta \xi_k\|^{\frac{n}{2}}_{\frac{n}{2}}=1+O((\ln k)^{-1})\quad\text{and} \quad\|\Delta \xi_k\|^p_p\leq  C_1 (\ln k)^{-\frac{2p}{n}}+C_2 (\ln k)^{-\frac{2p}{n}}\frac{1}{k^{\frac{n-2p}{n}}},$$
where $C_1$ and $C_2$ are two positive constants. Now, we define the sequence $\{w_k\}_k$ as follows
$$ w_k=\frac{\xi_k}{\|\Delta \xi_k\|_{\frac{n}{2}}},~\forall~k\in\mathbb{N}. $$
It follows at once that
$$\|\Delta w_k\|_{\frac{n}{2}}=1\quad\text{and}\quad\|\Delta w_k\|^p_p\leq  \|\Delta \xi_k\|^p_p\leq C_1 (\ln k)^{-\frac{2p}{n}}+C_2 (\ln k)^{-\frac{2p}{n}}\frac{1}{k^{\frac{n-2p}{n}}}. $$
Moreover, we have
\begin{align*}
    &\text{ATSC}(\ell,\gamma)\\
    &\geq \cfrac{1}{\|\Delta w_k\|_p^{p^\ast(1-\frac{\gamma}{n})}}\int_{\{x\in\mathbb{R}^n\colon~|x|< k^{-\frac{1}{n}} \}}\frac{\Phi_{\ell(1-\frac{\gamma}{n}),j_0}(w_k)}{|x|^\gamma}\,\mathrm{d}x\\
    &\geq \frac{\omega_{n-1}(\ln k)^{\frac{2p^\ast}{n}(1-\frac{\gamma}{n})}}{\big(C_1+C_2 k^{-(\frac{n-2p}{n})}\big)^{\frac{p^\ast}{p}(1-\frac{\gamma}{n})}}\int_{0}^{k^{-\frac{1}{n}}}\frac{\Phi_{\ell(1-\frac{\gamma}{n}),j_0}\bigg(\frac{1}{\|\Delta \xi_k\|_{\frac{n}{2}}}\bigg(\Big(\frac{\ln k}{\beta(n,2)}\Big)^{1-\frac{2}{n}}+(1-k^{\frac{2}{n}}r^2)\frac{n \beta(n,2)^{\frac{2}{n}-1}}{2(\ln k)^{\frac{2}{n}}}\bigg)\bigg)}{r^{\gamma-n+1}}\,\mathrm{d}r\\
    &\geq \frac{(\ln k)^{\frac{2p^\ast}{n}(1-\frac{\gamma}{n})}}{\big(C_1+C_2 k^{-(\frac{n-2p}{n})}\big)^{\frac{p^\ast}{p}(1-\frac{\gamma}{n})}}~\frac{\omega_{n-1}}{(n-\gamma)k^{\frac{n-\gamma}{n}}}~\Phi_{\ell(1-\frac{\gamma}{n}),j_0}\bigg(\frac{1}{\|\Delta \xi_k\|_{\frac{n}{2}}}\Big(\frac{\ln k}{\beta(n,2)}\Big)^{1-\frac{2}{n}}\bigg)\\
    &\sim \frac{(\ln k)^{\frac{2p^\ast}{n}(1-\frac{\gamma}{n})}}{\big(C_1+C_2 k^{-(\frac{n-2p}{n})}\big)^{\frac{p^\ast}{p}(1-\frac{\gamma}{n})}}~\frac{\omega_{n-1}}{(n-\gamma)}~\operatorname{\exp}\Bigg[\Big(1-\frac{\gamma}{n}\Big)\frac{\ell}{\beta(n,2)}\ln k\Bigg(\frac{1}{\|\Delta \xi_k\|_{\frac{n}{2}}^{\frac{n}{n-2}}}-\frac{\beta(n,2)}{\ell}\Bigg)\Bigg].
\end{align*}
For sufficiently large $k$, one has
$$\frac{1}{\big(C_1+C_2 k^{-(\frac{n-2p}{n})}\big)^{\frac{p^\ast}{p}(1-\frac{\gamma}{n})}}\thickapprox \frac{1}{C_1^{\frac{p^\ast}{p}(1-\frac{\gamma}{n})}}\quad\text{and}\quad \frac{1}{\|\Delta \xi_k\|_{\frac{n}{2}}^{\frac{n}{n-2}}}-\frac{\beta(n,2)}{\ell} \thickapprox 1-\frac{\beta(n,2)}{\ell}.$$
In conclusion, we have
$$\text{ATSC}(\ell,\gamma)\gtrsim (\ln k)^{\frac{2p^\ast}{n}(1-\frac{\gamma}{n})}~\operatorname{\exp}\Bigg[\Big(1-\frac{\gamma}{n}\Big)\ln k\bigg(\frac{\ell}{\beta(n,2)}-1\bigg)\Bigg].$$
Notice that when $\ell\underset{\approx}{<}\beta(n,2)$, we can able to choose $k$ sufficiently large enough such that
$$ \ln k\thickapprox\frac{p}{1-\frac{\ell}{\beta(n,2)}} \quad\text{and }\quad  \Big(1-\frac{\gamma}{n}\Big)\ln k\Big(\frac{\ell}{\beta(n,2)}-1\Big)\thickapprox -\Big(1-\frac{\gamma}{n}\Big)p.$$
Therefore, whenever $\ell\underset{\approx}{<}\beta(n,2)$, we have
$$\text{ATSC}(\ell,\gamma)\gtrsim \cfrac{\operatorname{\exp}\Big(-\Big(1-\frac{\gamma}{n}\Big)p\Big)}{\Big(1-\frac{\ell}{\beta(n,2)}\Big)^{\frac{2p^\ast}{n}\big(1-\frac{\gamma}{n}\big)}}\thickapprox \cfrac{\big(\frac{n-2}{2}\big)^{\frac{2p^\ast}{n}\big(1-\frac{\gamma}{n}\big)} ~\operatorname{\exp}\Big(-\Big(1-\frac{\gamma}{n}\Big)p\Big)}{\bigg(1-\big(\frac{\ell}{\beta(n,2)}\big)^{\frac{n-2}{2}}\bigg)^{\frac{2p^\ast}{n}\big(1-\frac{\gamma}{n}\big)}},$$
where we have used
$$\frac{1-\Big(\frac{\ell}{\beta(n,2)}\Big)^{\frac{n-2}{2}}}{1-\frac{\ell}{\beta(n,2)}}\thickapprox\frac{n-2}{2}\quad\text{for}\quad \ell\underset{\approx}{<}\beta(n,2). $$
This yields that
$$ \text{ATSC}(\ell,\gamma)\geq c(n,\gamma,p)\bigg(1-\bigg(\frac{\ell}{\beta(n,2)}\bigg)^{\frac{n-2}{2}}\bigg)^{-\frac{2p^\ast}{n}\big(1-\frac{\gamma}{n}\big)},$$
where $c(n,\gamma,p)>0$ is a constant depending only on $n$, $\gamma$ and $p$. This finishes the proof.
\end{proof}
\section{Equivalence of critical and subcritical sharp singular Adams' type inequalities: Proof of
 Theorem \ref{thm5.22}}\label{sec6}
The purpose of this section is to establish the relationship between the suprema for the critical and subcritical sharp singular Adams' type inequalities. Initially, we prove the important lemma shown below.
\begin{lemma}\label{lem7.2}
 The subcritical sharp singular Adams' type inequality in $E$ is a consequence of the critical sharp singular Adams' type inequality in $E$. More specifically, if $\text{ATC}_{a,b}(\gamma)<+\infty$, then $\text{ATSC}(\ell,\gamma)<+\infty$. Furthermore, there holds
 $$ \text{ATSC}(\ell,\gamma)\leq \Bigg(\cfrac{(\frac{\ell}{\beta(n,2)}\big)^{(\frac{n-2}{n})b}}{1-\big(\frac{\ell}{\beta(n,2)}\big)^{(\frac{n-2}{n})a}}\Bigg)^{\frac{p^\ast}{b}\big(1-\frac{\gamma}{n}\big)}\text{ATC}_{a,b}(\gamma).$$
 In particular, if $\text{ATC}_{\frac{n}{2},\frac{n}{2}}(\gamma)<+\infty$, then
  $$ \text{ATSC}(\ell,\gamma)\leq \Bigg(\cfrac{\big(\frac{\ell}{\beta(n,2)}\big)^{\frac{n-2}{2}}}{1-\big(\frac{\ell}{\beta(n,2)}\big)^{\frac{n-2}{2}}}\Bigg)^{\frac{2p^\ast}{n}\big(1-\frac{\gamma}{n}\big)}\text{ATC}_{\frac{n}{2},\frac{n}{2}}(\gamma).$$
\end{lemma}
\begin{proof}
Let $u\in E$ be such that $\|\Delta u\|_{\frac{n}{2}}\leq 1$ and $\|\Delta u\|_p=1$. Define a new function $w$ as follows
$$ w(x)=\bigg(\frac{\ell}{\beta(n,2)}\bigg)^{\frac{n-2}{n}}u(\lambda x)\quad\text{with}\quad \lambda=\Bigg(\cfrac{(\frac{\ell}{\beta(n,2)}\big)^{(\frac{n-2}{n})b}}{1-\big(\frac{\ell}{\beta(n,2)}\big)^{(\frac{n-2}{n})a}}\Bigg)^{\frac{p}{b(n-2p)}}. $$
By direct calculations, we have
$$\|\Delta w\|^a_{\frac{n}{2}}=\bigg(\frac{\ell}{\beta(n,2)}\bigg)^{\big(\frac{n-2}{n}\big)a}\|\Delta u\|^a_{\frac{n}{2}}\leq \bigg(\frac{\ell}{\beta(n,2)}\bigg)^{\big(\frac{n-2}{n}\big)a} $$
and
\begin{align*}
 \|\Delta w\|^b_p &=\bigg(\frac{\ell}{\beta(n,2)}\bigg)^{(\frac{n-2}{n})b}\lambda^{-b\big(\frac{n-2p}{p}\big)}\|\Delta u\|^b_p= 1- \bigg(\frac{\ell}{\beta(n,2)}\bigg)^{(\frac{n-2}{n})a}.
\end{align*}
It follows that $\|\Delta w\|^a_{\frac{n}{2}}+\|\Delta w\|^b_p\leq 1$. Now, we deduce from the definition of $\text{ATC}_{a,b}(\gamma)$ that 
\begin{align*}
 \int_{\mathbb{R}^n}\frac{\Phi_{\ell(1-\frac{\gamma}{n}),j_0}(u)}{|x|^\gamma}\,\mathrm{d}x &= \int_{\mathbb{R}^n}\frac{\Phi_{\ell(1-\frac{\gamma}{n}),j_0}(u(\lambda x))}{|\lambda x|^\gamma}\,\mathrm{d}(\lambda x)=\lambda^{n-\gamma} \int_{\mathbb{R}^n}\frac{\Phi_{\beta_{\gamma,n},j_0}(w)}{|x|^\gamma}\,\mathrm{d}x\\
 &\leq \Bigg(\cfrac{\big(\frac{\ell}{\beta(n,2)}\big)^{(\frac{n-2}{n})b}}{1-\big(\frac{\ell}{\beta(n,2)}\big)^{(\frac{n-2}{n})a}}\Bigg)^{\frac{p^\ast}{b}\big(1-\frac{\gamma}{n}\big)}\text{ATC}_{a,b}(\gamma).
\end{align*}
The above inequality ensures that the proof of Lemma \ref{lem7.2} is finished at this point.
\end{proof}
Now, we are in a suitable position to prove Theorem \ref{thm5.22} as follows.
\begin{proof}[\bf{Proof of Theorem \ref{thm5.22}}.]
Let $0<b\leq\frac{n}{2}$ and $u\in E\setminus\{0\}$ be such that $\|\Delta u\|^a_{\frac{n}{2}}+\|\Delta u\|^b_p\leq 1$. Now, we define
$$ \|\Delta u\|_{\frac{n}{2}}=\theta\in(0,1)\quad\text{and}\quad\|\Delta u\|^b_p\leq 1-\theta^a.$$
If $\cfrac{1}{2^{\frac{n-2p}{p}}}<\theta<1$, then we again define
$$w(x)=\frac{u(\lambda x)}{\theta}\quad\text{and}\quad\lambda=\bigg(\frac{(1-\theta^a)^{\frac{1}{b}}}{\theta}\bigg)^{\frac{p}{n-2p}} .$$
This yields that
$$\|\Delta w\|_{\frac{n}{2}}=1\quad\text{and}\quad \|\Delta w\|^p_p=\frac{\|\Delta u\|^p_p}{\lambda^{n-2p}\theta^p}\leq\frac{(1-\theta^a)^{\frac{p}{b}}}{\lambda^{n-2p}\theta^p}=1.$$
Now, by employing Theorem \ref{thm5.1}, one has
\begin{align*}
   \int_{\mathbb{R}^n}\frac{\Phi_{\beta_{\gamma,n},j_0}(u)}{|x|^\gamma}\,\mathrm{d}x&=\int_{\mathbb{R}^n}\frac{\Phi_{\beta_{\gamma,n},j_0}(u(\lambda x))}{|\lambda x|^\gamma}\,\mathrm{d}(\lambda x)\leq \lambda^{n-\gamma}\int_{\mathbb{R}^n}\frac{\Phi_{\theta^{\frac{n}{n-2}}\beta_{\gamma,n},j_0}(u)}{|x|^\gamma}\,\mathrm{d}x \\
   &\leq \lambda^{n-\gamma}\text{ATSC}(\theta^{\frac{n}{n-2}}\beta(n,2),\gamma)\leq \bigg(\frac{(1-\theta^a)^{\frac{1}{b}}}{\theta}\bigg)^{p^\ast(1-\frac{\gamma}{n})}\frac{C(\gamma,p,n)}{(1-\theta^{\frac{n}{2}})^{\frac{2p^\ast}{n}(1-\frac{\gamma}{n})}}\\
   &\leq \frac{(1-\theta^a)^{\frac{p^\ast}{b}(1-\frac{\gamma}{n})}}{(1-\theta^{\frac{n}{2}})^{\frac{2p^\ast}{n}(1-\frac{\gamma}{n})}}C(\gamma,p,n)\leq C(\gamma,p,n,a,b),   
\end{align*}
 thanks to the fact that $b\leq\frac{n}{2}$, where $C(\gamma,p,n,a,b)>0$ is a constant depending only on $ \gamma,p,n,a$ and $b$.
If $0<\theta\leq \cfrac{1}{2^{\frac{n-2p}{p}}} $, then we consider
$$ w(x)=2^{\frac{n-2p}{p}}u(2x). $$
It follows directly that
$$ \|\Delta w\|_{\frac{n}{2}}=2^{\frac{n-2p}{p}} \|\Delta u\|_{\frac{n}{2}}\leq 1\quad\text{and}\quad \|\Delta w\|_p=\|\Delta u\|_p\leq 1.$$
Due to Theorem \ref{thm5.1}, we have
\begin{align*}
   \int_{\mathbb{R}^n}\frac{\Phi_{\beta_{\gamma,n},j_0}(u)}{|x|^\gamma}\,\mathrm{d}x&=\int_{\mathbb{R}^n}\frac{\Phi_{\beta_{\gamma,n},j_0}(u(2 x))}{|2 x|^\gamma}\,\mathrm{d}(2 x)\leq 2^{n-\gamma}\int_{\mathbb{R}^n}\frac{\Phi_{ 2^{-\big(\frac{n^2}{(n-2)p^\ast}\big)}\beta_{\gamma,n},j_0}(w)}{|x|^\gamma}\,\mathrm{d}x \\
   &\leq 2^{n-\gamma}\text{ATSC}\Big(2^{-\big(\frac{n^2}{(n-2)p^\ast}\big)}\beta(n,2),\gamma\Big)\\
   &\leq 2^{n-\gamma}\frac{C(\gamma,p,n)}{\Big(1-2^{-\frac{n^2}{2p^\ast}}\Big)^{\frac{2p^\ast}{n}(1-\frac{\gamma}{n})}}\leq C(\gamma,p,n) ,   
\end{align*}
where $C(\gamma,p,n)>0$ is a constant depending only on $ \gamma,n$ and $p$. Hence, we deduce that $\text{ATC}_{a,b}(\gamma)<+\infty$ for $b\leq \frac{n}{2}$. Next, to prove the sharpness of $\beta(n,2)$, inspired by Theorem \ref{thm5.1}, we assume that
$$\bar{v}(s)=\Big(\log\frac{1}{r}\Big)^{1-\frac{2}{n}}v(s),~\forall~s>0\quad\text{and}\quad r\in(0,1), $$
where $v$ is defined as in \eqref{eq7.3}. Further, we define another function $\bar{w}$ as follows
$$\bar{w}(s)=\delta_r\frac{\bar{v}}{\|\Delta \bar{v}\|_{\frac{n}{2}}},~\forall~s>0\quad\text{and}\quad \delta_r\in(0,1), $$
with $\delta_r\to 1$ as $r\to 0^+$ and there holds
$$\delta^a_r+\delta^b_r\frac{\|\Delta \bar{v}\|^b_p}{\|\Delta \bar{v}\|^b_{\frac{n}{2}}}=1. $$
Notice that $\|\Delta \bar{w}\|^a_{\frac{n}{2}}+\|\Delta \bar{w}\|^b_p=1$. Moreover, by choosing $\ell>\beta(n,2)$, we obtain
\begin{align*}
    \int_{\mathbb{R}^n}&\frac{\Phi_{\ell(1-\frac{\gamma}{n}),j_0}(\bar{w})}{|x|^\gamma}\,\mathrm{d}x\geq \int_{\{x\in \mathbb{R}^n\colon~|x|<r\}}\cfrac{\Phi_{\ell(1-\frac{\gamma}{n}),j_0}\Big(\delta_r\frac{\bar{v}}{\|\Delta \bar{v}\|_{\frac{n}{2}}}\Big)}{|x|^\gamma}\,\mathrm{d}x\\
    &=\frac{\omega_{n-1}}{n-\gamma}~r^{n-\gamma}~\Phi_{\ell(1-\frac{\gamma}{n}),j_0}\bigg(\frac{\delta_r(\log\frac{1}{r})^{1-\frac{2}{n}}}{\big(\omega_{n-1}^{1-\frac{n}{2}}\alpha(n,2)^{\frac{n}{2}}\Theta_{\frac{n}{2}} \big)^{\frac{2}{n}}}\bigg)\\
    &\sim \frac{\omega_{n-1}}{n-\gamma}~r^{n-\gamma}~\operatorname{\exp}\Bigg[\ell\Big(1-\frac{\gamma}{n}\Big)\Bigg|\frac{\delta_r(\log\frac{1}{r})^{1-\frac{2}{n}}}{\big(\omega_{n-1}^{1-\frac{n}{2}}\alpha(n,2)^{\frac{n}{2}}\Theta_{\frac{n}{2}} \big)^{\frac{2}{n}}}\Bigg|^{\frac{n}{n-2}}\Bigg]\\
    &=\frac{\omega_{n-1}}{n-\gamma}~\operatorname{\exp}\Bigg[n\Big(1-\frac{\gamma}{n}\Big)\log\frac{1}{r}\Bigg(\frac{\ell}{\beta(n,2)}\frac{\delta_r^{\frac{n}{n-2}}}{\Big(1+2\varepsilon\Big(\|f^\prime\|_\infty+O\Big(\Big(\log\frac{1}{r}\Big)^{-1}\Big)\Big)^{\frac{n}{2}}\Big)^{\frac{2}{n-2}}}-1\Bigg)\Bigg]\\
    &\to+\infty\quad\text{as}\quad r\to 0^+,~\text{if we take }~\varepsilon>0~\text{sufficiently small enough}.
\end{align*}
This finishes the proof of the sharpness of $\beta(n,2)$. Finally, we deduce from Lemma \ref{lem7.2} that
\begin{align}\label{eq7.4}
    \sup_{\ell\in(0,\beta(n,2))}\Bigg(\cfrac{1-\big(\frac{\ell}{\beta(n,2)}\big)^{(\frac{n-2}{n})a}}{(\frac{\ell}{\beta(n,2)}\big)^{(\frac{n-2}{n})b}}\Bigg)^{\frac{p^\ast}{b}\big(1-\frac{\gamma}{n}\big)}\text{ATSC}(\ell,\gamma)\leq \text{ATC}_{a,b}(\gamma).
\end{align}
Let $\{v_k\}_k\subset E\setminus\{0\}$ be a maximizing sequence of $\text{ATC}_{a,b}(\gamma)$. It follows that $\|\Delta v_k\|^a_{\frac{n}{2}}+\|\Delta v_k\|^b_p\leq 1$ and there holds
$$\lim_{k\to\infty}\int_{\mathbb{R}^n}\frac{\Phi_{\beta_{\gamma,n},j_0}(v_k)}{|x|^\gamma}\,\mathrm{d}x=\text{ATC}_{a,b}(\gamma). $$
For each $k\in\mathbb{N}$, we define the  sequence $\{w_k\}_k$ as follows
$$ w_k(x)=\frac{v_k(\lambda_k x)}{\|\Delta v_k\|_{\frac{n}{2}}},\quad\text{where}\quad\lambda_k=\bigg(\frac{1-\|\Delta v_k\|^a_{\frac{n}{2}}}{\|\Delta v_k\|^b_{\frac{n}{2}}}\bigg)^{\frac{p^\ast}{nb}} .$$
Thus, we have
$$ \|\Delta w_k\|_{\frac{n}{2}}=1\quad\text{and}\quad \|\Delta w_k\|_p=\frac{\|\Delta v_k\|_p}{\lambda_k^{\frac{n-2p}{p}}\|\Delta v_k\|_{\frac{n}{2}}}\leq \frac{(1-\|\Delta v_k\|^a_{\frac{n}{2}})^{\frac{1}{b}}}{\lambda_k^{\frac{n-2p}{p}}\|\Delta v_k\|_{\frac{n}{2}}}=1.$$
Hence, we obtain from the definition of $\text{ATSC}(\ell,\gamma)$ that
\begin{align*}
  \int_{\mathbb{R}^n}\frac{\Phi_{\beta_{\gamma,n},j_0}(v_k)}{|x|^\gamma}\,\mathrm{d}x&=\int_{\mathbb{R}^n}\frac{\Phi_{\beta_{\gamma,n},j_0}(v_k(\lambda_k x))}{|\lambda_k x|^\gamma}\,\mathrm{d}(\lambda_k x)\leq \lambda_k^{n-\gamma}\int_{\mathbb{R}^n}\frac{\Phi_{ \|\Delta v_k\|_{\frac{n}{2}}^{\frac{n}{n-2}}\beta_{\gamma,n},j_0}(w_k)}{|x|^\gamma}\,\mathrm{d}x \\
   &\leq \lambda_k^{n-\gamma}\text{ATSC}\Big(\|\Delta v_k\|_{\frac{n}{2}}^{\frac{n}{n-2}}\beta(n,2),\gamma\Big)\\
   &\leq \sup_{\ell\in(0,\beta(n,2))}\Bigg(\cfrac{1-\big(\frac{\ell}{\beta(n,2)}\big)^{(\frac{n-2}{n})a}}{\big(\frac{\ell}{\beta(n,2)}\big)^{(\frac{n-2}{n})b}}\Bigg)^{\frac{p^\ast}{b}\big(1-\frac{\gamma}{n}\big)}\text{ATSC}(\ell,\gamma).
\end{align*}
From the above inequality, we have
\begin{align}\label{eq7.5}
  \text{ATC}_{a,b}(\gamma)\leq \sup_{\ell\in(0,\beta(n,2))}\Bigg(\cfrac{1-\big(\frac{\ell}{\beta(n,2)}\big)^{(\frac{n-2}{n})a}}{\big(\frac{\ell}{\beta(n,2)}\big)^{(\frac{n-2}{n})b}}\Bigg)^{\frac{p^\ast}{b}\big(1-\frac{\gamma}{n}\big)}\text{ATSC}(\ell,\gamma).  
\end{align}
In view of \eqref{eq7.4} and \eqref{eq7.5}, we deduce that \eqref{eq1.15} holds. To complete the proof, we claim that $\text{ATC}_{a,b}(\gamma)<+\infty$ implies $b\leq \frac{n}{2}$. Indeed, if not, let there exists some $b>\frac{n}{2}$ such that $\text{ATC}_{a,b}(\gamma)<+\infty$. Thus, one can obtain from \eqref{eq1.15} that
\begin{align}\label{eq7.6}
    \limsup_{\ell\to\beta(n,2)^-}\Bigg(\cfrac{1-\big(\frac{\ell}{\beta(n,2)}\big)^{(\frac{n-2}{n})a}}{\big(\frac{\ell}{\beta(n,2)}\big)^{(\frac{n-2}{n})b}}\Bigg)^{\frac{p^\ast}{b}\big(1-\frac{\gamma}{n}\big)}\text{ATSC}(\ell,\gamma)<+\infty.
\end{align}
On the other hand, by Theorem \ref{thm5.1}, we have
\begin{align}\label{eq7.7}
    \limsup_{\ell\to\beta(n,2)^-}\Bigg(\cfrac{1-\big(\frac{\ell}{\beta(n,2)}\big)^{(\frac{n-2}{n})a}}{\big(\frac{\ell}{\beta(n,2)}\big)^{(\frac{n-2}{n})b}}\Bigg)^{\frac{2p^\ast}{n}\big(1-\frac{\gamma}{n}\big)}\text{ATSC}(\ell,\gamma)>0.
\end{align}
In virtue of \eqref{eq7.6} and \eqref{eq7.7}, one has
$$  \liminf_{\ell\to\beta(n,2)^-}\Bigg(\cfrac{1-\big(\frac{\ell}{\beta(n,2)}\big)^{(\frac{n-2}{n})a}}{\big(\frac{\ell}{\beta(n,2)}\big)^{(\frac{n-2}{n})b}}\Bigg)^{\frac{p^\ast}{b}\big(1-\frac{\gamma}{n}\big)}\Bigg(\cfrac{\big(\frac{\ell}{\beta(n,2)}\big)^{(\frac{n-2}{n})b}}{1-\big(\frac{\ell}{\beta(n,2)}\big)^{(\frac{n-2}{n})a}}\Bigg)^{\frac{2p^\ast}{n}\big(1-\frac{\gamma}{n}\big)}<+\infty, $$
which is a contradiction, thanks to $b>\frac{n}{2}$ and thus, the claim is verified. This finishes the proof.
\end{proof}
\section{The compact embedding result: Proof of Theorem \ref{thm1.4}}\label{sec7}
This section is concerned with the attainability of the best constant in the embedding $E\hookrightarrow L^\varrho(\mathbb{R}^n,|x|^{-\gamma}\mathrm{d}x)$ for all $\varrho\geq p^\ast$ and $\gamma\in(0,n)$. For this purpose, we first need to prove Theorem \ref{thm1.4}.
\begin{proof}[\bf{Proof of Theorem \ref{thm1.4}}.]
First, we claim that the continuous embedding $E\hookrightarrow L^\varrho(\mathbb{R}^n,|x|^{-\gamma}\mathrm{d}x)$ holds for all $\varrho\geq p^\ast$ and $\gamma\in(0,n)$. Indeed, let $v\in E$ and $\eta,~\eta^\prime>1$ be such that $\frac{1}{\eta}+\frac{1}{\eta^\prime}=1$ satisfying $\gamma \eta^\prime<n$. By using the Hölder's inequality and Lemma \ref{lem2.5}, one has
\begin{align*}
   \int_{\mathbb{R}^n}\frac{|v|^\varrho}{|x|^\gamma}\,\mathrm{d}x &= \bigg(\int_{B_1}+\int_{B^c_1}\bigg)\frac{|v|^\varrho}{|x|^\gamma}\,\mathrm{d}x\leq \bigg(\int_{B_1}|v|^{\varrho\eta}\,\mathrm{d}x\bigg)^{\frac{1}{\eta}}\bigg(\int_{B_1}\frac{1}{|x|^{\gamma \eta^\prime}}\,\mathrm{d}x\bigg)^{\frac{1}{\eta^\prime}}+\int_{\mathbb{R}^n}|v|^\varrho\,\mathrm{d}x\lesssim\|v\|^\varrho,
\end{align*}
which concludes the claim. It remains to check the compactness of the above embedding. To prove this, take a bounded sequence $\{v_k\}_k\subset E$, then one can find a subsequence $\{v_{k_j}\}_j$ and $v$ in $E$ satisfying
\begin{align*}
  v_{k_j}\rightharpoonup v~\text{in} ~ E,~ v_{k_j}\to v~\text{in} ~ L^\theta(B_R)~~\text{for all}~ R>0~~\text{and}~~\theta\geq 1,~v_{k_j}\to v~\text{a.e. in}~\mathbb{R}^n~\text{as}~j\to\infty,
\end{align*}
we thank to Lemma \ref{lem2.5}. Now, thanks to the Egoroff's theorem, we can notice that for any open ball $B_R$ and $\delta>0$, there exists $A_\delta\subset B_R$ with $|A_\delta|<\delta$ satisfying $v_{k_j}\to v$ uniformly in $B_R\setminus A_\delta$ as $j\to\infty$. Further, let us write
\begin{equation}\label{eq4.1}
\begin{aligned}
\lim_{j\to\infty}\int_{\mathbb{R}^n}\frac{|v_{k_j}- v|^\varrho}{|x|^\gamma}\,\mathrm{d}x  &= \lim_{R\to\infty}\lim_{\delta\to 0}\lim_{j\to\infty}\int_{\mathbb{R}^n}\frac{|v_{k_j}- v|^\varrho}{|x|^\gamma}\,\mathrm{d}x\\ 
    &=\lim_{R\to\infty}\lim_{\delta\to 0}\lim_{j\to\infty}\int_{ A_\delta}\frac{|v_{k_j}- v|^\varrho}{|x|^\gamma}\,\mathrm{d}x+\lim_{R\to\infty}\lim_{\delta\to 0}\lim_{j\to\infty}\int_{B_R\setminus A_\delta}\frac{|v_{k_j}- v|^\varrho}{|x|^\gamma}\,\mathrm{d}x\\ &\quad+\lim_{R\to\infty}\lim_{\delta\to 0}\lim_{j\to\infty}\int_{\mathbb{R}^n\setminus B_R}\frac{|v_{k_j}- v|^\varrho}{|x|^\gamma}\,\mathrm{d}x:=J_1+J_2+J_3.
\end{aligned}
\end{equation}
Suppose that $1<t<\frac{n}{\gamma}$ and $1<r<\frac{n}{t\gamma}$ such that $\frac{1}{\tau}+\frac{1}{\tau^\prime}=1$ for all $\tau\in\{t,r\}$. Hence, by using the Hölder's inequality and Lemma \ref{lem2.5}, we have
\begin{align*}
 \int_{ A_\delta}\frac{|v_{k_j}- v|^\varrho}{|x|^\gamma}\,\mathrm{d}x &\leq |A_\delta|^{\frac{1}{t^\prime}}\bigg(\int_{ A_\delta}\frac{|v_{k_j}- v|^{\varrho t}}{|x|^{t\gamma }}\,\mathrm{d}x \bigg)^{\frac{1}{t}}\leq|A_\delta|^{\frac{1}{t^\prime}}\bigg(\int_{ A_\delta}\frac{1}{|x|^{rt \gamma}}\,\mathrm{d}x\bigg)^{\frac{1}{rt}} \bigg(\int_{ A_\delta}|v_{k_j}- v|^{\varrho r^\prime t}\,\mathrm{d}x\bigg)^{\frac{1}{r^\prime t}} \\ 
 & \leq |A_\delta|^{\frac{1}{t^\prime}}\bigg(\int_{B_R}\frac{1}{|x|^{rt \gamma}}\,\mathrm{d}x\bigg)^{\frac{1}{rt}} \bigg(\int_{ \mathbb{R}^n}|v_{k_j}- v|^{\varrho r^\prime t}\,\mathrm{d}x\bigg)^{\frac{1}{r^\prime t}} \\ &
 \lesssim  |A_\delta|^{\frac{1}{t^\prime}}\bigg(\int_{B_R}\frac{1}{|x|^{rt \gamma}}\,\mathrm{d}x\bigg)^{\frac{1}{rt}} \sup_{j\in\mathbb{N}}\|v_{k_j}- v\|^\varrho\lesssim |A_\delta|^{\frac{1}{t^\prime}},
\end{align*}
which directly implies that
\begin{equation}\label{eq4.2}
    J_1=\lim_{R\to\infty}\lim_{\delta\to 0}\lim_{j\to\infty}\int_{ A_\delta}\frac{|v_{k_j}- v|^\varrho}{|x|^\gamma}\,\mathrm{d}x=0.
\end{equation}
Moreover, we also have
\begin{align*}
 \int_{B_R\setminus A_\delta}\frac{|v_{k_j}- v|^\varrho}{|x|^\gamma}\,\mathrm{d}x\leq \sup_{B_R\setminus A_\delta}|v_{k_j}- v|^\varrho\int_{B_R} \frac{1}{|x|^{ \gamma}}\,\mathrm{d}x\lesssim \sup_{B_R\setminus A_\delta}|v_{k_j}- v|^\varrho, 
\end{align*}
which together with the fact that $v_{k_j}\to v$ uniformly in $B_R\setminus A_\delta$ as $j\to\infty$ ensures that
\begin{align}\label{eq4.3}
  J_2=\lim_{R\to\infty}\lim_{\delta\to 0}\lim_{j\to\infty}\int_{B_R\setminus A_\delta}\frac{|v_{k_j}- v|^\varrho}{|x|^\gamma}\,\mathrm{d}x=0.  
\end{align}
Finally, one can observe that
\begin{align*}
  \int_{\mathbb{R}^n\setminus B_R}\frac{|v_{k_j}- v|^\varrho}{|x|^\gamma}\,\mathrm{d}x\leq \frac{1}{R^\gamma}  \int_{\mathbb{R}^n}|v_{k_j}- v|^\varrho\,\mathrm{d}x\lesssim \frac{1}{R^\gamma}\sup_{j\in\mathbb{N}}\|v_{k_j}- v\|^\varrho\lesssim \frac{1}{R^\gamma},
\end{align*}
which implies at once that
\begin{align}\label{eq4.4}
 J_3=\lim_{R\to\infty}\lim_{\delta\to 0}\lim_{j\to\infty}\int_{\mathbb{R}^n\setminus B_R}\frac{|v_{k_j}- v|^\varrho}{|x|^\gamma}\,\mathrm{d}x=0.   
\end{align}
In virtue of \eqref{eq4.2}, \eqref{eq4.3} and \eqref{eq4.4}, we obtain from \eqref{eq4.1} that $v_{k_j}\to v$ in $L^\varrho(\mathbb{R}^n,|x|^{-\gamma}\mathrm{d}x)$ as $j\to\infty$. This completes the proof.
\end{proof}
The following corollary is an immediate consequence of the above theorem.
\begin{corollary}
 The best constant $S_\varrho$ in the embedding $E\hookrightarrow L^\varrho(\mathbb{R}^n,|x|^{-\gamma}\mathrm{d}x)$ for  all $\varrho\geq p^\ast$ and $\gamma\in(0,n)$ is achieved, where $S_\varrho$ is given by
$$ S_\varrho:=\inf_{u\in E\setminus\{0\}}\frac{\|u\|}{~\|u\|_{\varrho,\gamma}}. $$   
\end{corollary}
\begin{proof}
 Indeed, let $\{v_k\}_k\subset E$ be a sequence such that $\|v_k\|_{\varrho,\gamma}=1~\text{and}~\|v_k\|\to S_\varrho~\text{ as}~ k\to\infty$. It follows that $\{v_k\}_k\subset E$ is bounded. Thus, up to a subsequence, still denoted by itself and $v\in E$ such that 
$$ v_{k}\rightharpoonup v~\text{in} ~ E,~ v_{k}\to v~\text{in} ~ L^\varrho(\mathbb{R}^n,|x|^{-\gamma}\mathrm{d}x)~~\text{for all}~~\varrho\geq p^\ast~\text{and}~v_{k}\to v~\text{a.e. in}~\mathbb{R}^n~\text{as}~k\to\infty,$$ we thank to Theorem \ref{thm1.4}. This infers that $\|v\|_{\varrho,\gamma}=\lim_{k\to\infty}\|v_k\|_{\varrho,\gamma}=1$. Moreover, due to the weak lower semicontinuity of the norm, we have $\|v\|\leq \liminf_{k\to\infty} \|v_k\|=S_\varrho\leq \|v\|$. It follows that the best constant $S_\varrho$ is achieved. This finishes the proof. 
\end{proof}
\section{Existence of nontrivial weak solutions: Proof of Theorem \ref{thm1.5}}\label{sec8}
In this section, we shall use the mountain pass theorem, standard topological tools, singular Adams' inequality, a new concentration-compactness principle due to Lions (see, for example, Theorem \ref{thm1.1} and Theorem \ref{thm1.2}) and Theorem \ref{thm1.4} to prove Theorem \ref{thm1.5}. Evidently, the problem \eqref{main problem} has a variational structure. Indeed, since we are interested in studying the problem \eqref{main problem}, we define the Euler–Lagrange variational energy functional $J \colon E\to \mathbb{R}$ by
\begin{align}
   J(u):=\frac{1}{p}\|\Delta u\|^p_p+\frac{2}{n}\|\Delta u\|^{\frac{n}{2}}_{\frac{n}{2}}-\int_{\mathbb{R}^n}\frac{G(x,u)}{|x|^\gamma}\,\mathrm{d}x,~\forall~ u\in E.  
\end{align}
It is easy to see that $J$ is well-defined, of class $C^1(E,\mathbb{R})$ and its Gâteaux derivative is given by 
\begin{align*}
    \langle{J^\prime(u),v}\rangle:=\int_{\mathbb{R}^n} |\Delta u|^{p-2}\Delta u\Delta v\,\mathrm{d}x+\int_{\mathbb{R}^n} |\Delta u|^{\frac{n}{2}-2}\Delta u\Delta v\,\mathrm{d}x-\int_{\mathbb{R}^n}\frac{g(x,u)v}{|x|^\gamma}\,\mathrm{d}x,~\forall~u,v\in E.
\end{align*}
Moreover, the critical points of $J$ are exactly weak solutions of the problem \eqref{main problem}. To prove our existence results, we need a version of the mountain pass theorem of Rabinowitz \cite{Rabinowitz-1986}, as stated below.
\begin{definition}(\text{Palais-Smale compactness condition})
 Let $X$ be a Banach space and $\mathcal{I}:X\to\mathbb{R}$ be a functional of class $C^1(X,\mathbb{R})$. We say that  $\mathcal{I}$ satisfies the Palais-Smale compactness condition at a suitable level $c\in \mathbb{R}$, if for any sequence $\{u_k\}_k\subset X$ such that
 \begin{equation}\label{eq5.1}
  \mathcal{I} (u_k)\to c\quad\text{and}\quad\sup_{\|\varphi\|_{X}=1}|\langle{\mathcal{I}^\prime(u_k),\varphi}\rangle|\to 0\quad\text{as}\quad k\to\infty   
 \end{equation}
has a strongly convergent subsequence in  $X$. Note that the sequence $\{u_k\}_k\subset X$ satisfying \eqref{eq5.1} is known as a Palais-Smale sequence at level $c\in \mathbb{R}$, which is denoted by the symbol $(PS)_c$.
\end{definition}
\begin{theorem}(\text{The mountain pass theorem})\label{thm5.2}
Let $X$ be a real Banach space and $\mathcal{I}\in C^1(X,\mathbb{R})$. Suppose $\mathcal{I}(0)=0$ and there hold:
\begin{itemize}
    \item[\textnormal{(a)}] There exist two constants $\alpha,\rho>0$ such that $\mathcal{I}(u)\geq \alpha$ for all $u\in X$ with $\|u\|=\rho$;
     \item[\textnormal{(b)}] There exists $e\in X$ satisfying $\|e\|>\rho$ such that $\mathcal{I}(e)<0$.  
\end{itemize}
Define $$\Gamma:=\big\{\gamma\in C^1([0,1],X)\colon \gamma(0)=0,~\gamma(1)=e\big\}. $$
Then $$c:=\inf_{\gamma\in \Gamma}\max_{t\in[0,1]} \mathcal{I}(\gamma(t))\geq \alpha $$
and there exists a $(PS)_c$ sequence $\{u_k\}_k$ for $\mathcal{I}$ in $X$. Consequently, if $\mathcal{I}$ satisfies the Palais-Smale condition, then $\mathcal{I}$ possesses a critical value $c\geq\alpha$.   
\end{theorem}
To use Theorem \ref{thm5.2}, we first prove the mountain pass geometry of the energy functional $J$ as follows.
\begin{lemma}\label{lem5.3}
  Let the hypotheses \ref{g1}--\ref{g4} be satisfied. Then the following results are true:
    \begin{itemize}
        \item[\textnormal{(i)}] There exist $\delta,\rho>0$ such that $J(u)\geq \delta$ for all $u\in E$ and $\|u\|=\rho$;
        \item[\textnormal{(ii)}] There exists $e\in E$ with $\|e\|>\rho$ such that $J(e)<0$.
    \end{itemize}
\end{lemma}
\begin{proof}
    Let $t,t^\prime>1$ be such that $\frac{1}{t}+\frac{1}{t^\prime}=1$. Moreover, we assume $\kappa>0$ such that $\alpha t^\prime\kappa^{\frac{n}{n-2}}\leq\beta_{\gamma,n}$ whenever $\|u\|\leq\kappa$. From \eqref{eq1.2}, we obtain by using the Hölder's inequality, Theorem \ref{thm1.4} and Corollary \ref{cor2.6} that
    \begin{align*}
      \int_{\mathbb{R}^n}\frac{G(x,u)}{|x|^\gamma}\,\mathrm{d}x &\leq \zeta \|u\|^\tau_{\tau,\gamma}+D_\zeta \|u\|^q_{qt,\gamma}\bigg( \int_{\mathbb{R}^n}\frac{\Phi_{\alpha t^\prime,j_0}(u)}{|x|^\gamma}\,\mathrm{d}x \bigg)^{\frac{1}{t^\prime}}\\
      &\leq \zeta S^{-\tau}_\tau\|u\|^\tau+D_\zeta S^{-q}_{qt}\|u\|^q\bigg( \int_{\mathbb{R}^n}\frac{\Phi_{\beta_{\gamma,n},j_0}(u/\|u\|)}{|x|^\gamma}\,\mathrm{d}x \bigg)^{\frac{1}{t^\prime}}\\
      &\leq C(\zeta \|u\|^\tau+\|u\|^q),~\forall~\zeta>0,
    \end{align*}
where we have used Theorem \ref{thm1.1} to obtain the constant $C>0$, which is given by
\begin{align*}
    C:=\max\Bigg\{S^{-\tau}_\tau,D_\zeta S^{-q}_{qt}\bigg( \int_{\mathbb{R}^n}\frac{\Phi_{\beta_{\gamma,n},j_0}(u/\|u\|)}{|x|^\gamma}\,\mathrm{d}x \bigg)^{\frac{1}{t^\prime}}\Bigg\}<+\infty.
\end{align*}
Define $\rho=\min\{1,\kappa\}$. Hence, for all $u\in E$ with $\|u\|=\rho$, we have
\begin{align*}
    J(u)\geq \|u\|^{\frac{n}{2}}\bigg[\frac{2}{n}- C\big(\zeta \|u\|^{\tau-\frac{n}{2}}+\|u\|^{q-\frac{n}{2}}\big)\bigg].
\end{align*}
Let us assume $0<\zeta<\frac{2}{nC}$. Notice that $\tau,q>\frac{n}{2}$, therefore we may choose $\rho>0$ small enough such that $$\frac{2}{n}- C\big(\zeta \rho^{\tau-\frac{n}{2}}+\rho^{q-\frac{n}{2}}\big)>0.$$
It follows that
\begin{align*}
  J(u)\geq \rho^{\frac{n}{2}}\bigg[\frac{2}{n}- C\big(\zeta \rho^{\tau-\frac{n}{2}}+\rho^{q-\frac{n}{2}}\big)\bigg] :=\delta>0\quad\text{for}\quad \|u\|=\rho.   
\end{align*}
This finishes the proof of the first part of Lemma \ref{lem5.3}. Next, for the proof of the second part, we assume that $u\in E\setminus\{0\}$ and $u\geq 0$ with $\|u\|=1$. Define the map $\Psi\colon [1,+\infty)\to\mathbb{R}$ by $\Psi(s):=s^{-\mu}G(x,su)-G(x,u)$ for all $s\geq 1$ and $u\in\mathbb{R}^+:=(0,+\infty)$. By using \ref{g4}, one can easily check that the map $\cdot\mapsto \Psi(\cdot)$ is an increasing function on $[1,+\infty)$ and thus, we have $G(x,su)\geq s^{\mu}G(x,u)$ for all $s\geq 1$ and $u\in\mathbb{R}^+$. Now, choosing $s\geq 1$ sufficiently large enough, one has
\begin{align*}
    J(su)\leq \frac{2}{p}s^{\frac{n}{2}}-s^\mu \int_{\mathbb{R}^n}\frac{G(x,u)}{|x|^\gamma}\,\mathrm{d}x\to-\infty\quad\text{as}\quad s\to\infty,
\end{align*}
we thank to $\mu>\frac{n}{2}$. The proof is now completed by choosing $e=su$ with $s$ sufficiently large. Hence, the lemma is well-established.
\end{proof}
Due to Lemma \ref{lem5.3} and Theorem \ref{thm5.2}, we can  define the mountain pass minimax level as follows
 $$c:=\inf_{\gamma\in \Gamma}\max_{t\in[0,1]} J(\gamma(t))\geq \rho>0,$$
 where $$\Gamma:=\{\gamma\in C^1([0,1],E)\colon \gamma(0)=0,~ J(\gamma(1))<0\}. $$
More precisely, we have the following crucial lemma, which describes the range of the mountain pass minimax level $c$.
\begin{lemma}\label{lem5.4}
There exists $\lambda_0>0$ sufficiently large enough such that if the hypothesis \ref{g6} holds for all $\lambda\geq \lambda_0$, then the mountain pass minimax level $c$ satisfies
\begin{align}\label{eq5.3}
0<c<c_0:=\min\Bigg\{\frac{1}{2^{\frac{n}{2}}}\bigg(\frac{2\mu-n}{n\mu}\bigg)^{\frac{n}{2p}}\bigg(\frac{\beta_{\gamma,n}}{\alpha_0}\bigg)^{\frac{n-2}{2}},\frac{1}{2^p}\bigg(\frac{2\mu-n}{n\mu}\bigg)\bigg(\frac{\beta_{\gamma,n}}{\alpha_0}\bigg)^{\frac{(n-2)p}{n}}\Bigg\}.  
\end{align}    
\end{lemma}
\begin{proof}
Let $\psi\in C^\infty_0(\mathbb{R}^n,[0,1])$ be a cut-off function such that $\psi\equiv 1$ in $ B_{\frac{1}{2}}$, $\psi\equiv 0$ in $ B^c_1$, $|\nabla \psi(x)|\leq 2$ and $|\Delta \psi(x)|\leq 4$ for all $x\in\mathbb{R}^n$. By using \ref{g6}, for any $s\in[0,1]$, we have
\begin{align*}
J(s\psi)&\leq \frac{s^p}{p}\int_{B_1}|\Delta \psi|^p\,\mathrm{d}x+\frac{2s^{\frac{n}{2}}}{n}\int_{B_1}|\Delta \psi|^{\frac{n}{2}}\,\mathrm{d}x-\lambda s^\vartheta\int_{B_{\frac{1}{2}}}\frac{1}{|x|^\gamma}\,\mathrm{d}x\\
 &\leq \bigg(\frac{2^{2p}}{p}+\frac{2^{n+1}}{n}\bigg)\sigma_n s^p-\frac{\lambda n\sigma_n}{2^{n-\gamma}(n-\gamma)}s^\vartheta.
\end{align*}
Define the map $\pi\colon[0,1]\to E$ by $\pi(s)=s\psi$. Let $\lambda_1>0$ be such that for all $\lambda>\lambda_1$, we have
 \begin{align*}
        J(\psi)\leq \bigg(\frac{2^{2p}}{p}+\frac{2^{n+1}}{n}\bigg)\sigma_n-\frac{\lambda_1 n\sigma_n}{2^{n-\gamma}(n-\gamma)}<0.
    \end{align*}
It follows that $\pi\in \Gamma$. Moreover, by direct calculations, one has  
\begin{align*}
    c\leq \max_{s\in[0,1]} J(s\psi)&\leq \sigma_n\max_{s\geq0}\bigg\{\bigg(\frac{2^{2p}}{p}+\frac{2^{n+1}}{n}\bigg) s^p-\frac{\lambda n}{2^{n-\gamma}(n-\gamma)}s^\vartheta\bigg\}\\
    &=\sigma_n\bigg(\frac{\vartheta-p}{\vartheta}\bigg)\bigg(\frac{2^{2p}}{p}+\frac{2^{n+1}}{n}\bigg)^{\frac{\vartheta}{\vartheta-p}}\bigg(\frac{2^{n-\gamma}(n-\gamma)p}{\lambda n\vartheta}\bigg)^{\frac{p}{\vartheta-p}}\to 0\quad\text{as}\quad \lambda\to\infty.
\end{align*}
This indicates that there exists $\lambda_0>\lambda_1$ sufficiently large enough such that $c<c_0$ for all $\lambda\geq\lambda_0$, where $c_0$ is defined as in \eqref{eq5.3}. This completes the proof.
\end{proof}
\begin{lemma}\label{lem5.5}
    Suppose that $0<c<c_0$, where $c_0$ is defined as in Lemma \ref{lem5.4}. Then any $(PS)_c$ sequence $\{u_k\}_k\subset E$ for $J$ is bounded in $E$. Moreover, there holds
    \begin{align}\label{eq5.4}
        \limsup_{k\to\infty}\|u_k\|^{\frac{n}{n-2}}<\frac{\beta_{\gamma,n}}{\alpha_0}.
    \end{align}
\end{lemma}
\begin{proof}
 Let $0<c<c_0$ and $\{u_k\}_k\subset E$ be a $(PS)_c$ sequence for $J$. It follows that
\begin{align*}
        J(u_k)= c+o_k(1)\quad\text{and}\quad\bigg\langle{J^\prime(u_k),\frac{u_k}{\|u_k\|}}\bigg\rangle=o_k(1)\quad\text{as}\quad k\to\infty.
\end{align*}
By using \ref{g4} and the above fact, we have
\begin{align}\label{eq5.5}
        c+o_k(1)+o_k(1)\|u_k\|\geq \bigg(\frac{2}{n}-\frac{1}{\mu}\bigg)\Big(\|\Delta u_k\|_p^p+\|\Delta u_k\|_{\frac{n}{2}}^{\frac{n}{2}}\Big)\quad\text{as}\quad k\to\infty.
\end{align}
Indeed, if possible, let $\{u_k\}_k\subset E$ be unbounded in $E$. Thus, we have the following situations.

\textbf{Case-1:}
Suppose that $\|\Delta u_k\|_p\to \infty$ and $\|\Delta u_k\|_{\frac{n}{2}}\to\infty$ as $k\to\infty$. It follows from $1<p<\frac{n}{2}$ that 
\begin{align*}
 \|\Delta u_k\|_{\frac{n}{2}}^{\frac{n}{2}}\geq \|\Delta u_k\|_{\frac{n}{2}}^p>1\quad\text{for large}\quad k.   
\end{align*}
Moreover, we obtain from \eqref{eq1.99} and \eqref{eq5.5} that
\begin{align*}
  c+o_k(1)+o_k(1)\|u_k\|\geq 2^{1-p}\bigg(\frac{2}{n}-\frac{1}{\mu}\bigg)\|u_k\|^p \quad\text{as}\quad k\to\infty. 
\end{align*}
Dividing $\|u_k\|^p$ on the both the sides and passing $k\to\infty$, we get $0\geq 2^{1-p}\big(\frac{2}{n}-\frac{1}{\mu}\big)>0$, which is a contradiction.

\textbf{Case-2:}
Suppose that $\|\Delta u_k\|_p\to \infty$ as $k\to\infty$ and $\|\Delta u_k\|_{\frac{n}{2}}$ is bounded. It follows from \eqref{eq1.99} and \eqref{eq5.5} that
\begin{align*}
 c+o_k(1)+o_k(1)\|u_k\|_E\geq \bigg(\frac{2}{n}-\frac{1}{\mu}\bigg)\|\Delta u_k\|_p^p\quad\text{as}\quad k\to\infty.  
\end{align*}
Dividing $\|\Delta u_k\|_p^p$ on the both the sides and passing $k\to\infty$, we get $0\geq \big(\frac{2}{n}-\frac{1}{\mu}\big)>0$, which is again a contradiction.

\textbf{Case-3:}
Suppose that $\|\Delta u_k\|_{\frac{n}{2}}\to\infty$ as $k\to\infty$ and $\|\Delta u_k\|_p$ is bounded, then arguing as before in Case-$2$, we get a contradiction.

From the above three cases, we conclude that $\{u_k\}_k\subset E$ is a bounded sequence in $E$. Thus, by using \eqref{eq5.5}, one has the following estimates 
\begin{align}\label{eq5.6}
\limsup_{k\to\infty} \|\Delta u_k\|_{\frac{n}{2}}^{\frac{n}{2}}\leq \bigg(\frac{n\mu}{2\mu-n}\bigg)c\quad\text{and}\quad \limsup_{k\to\infty} \|\Delta u_k\|_p^{\frac{n}{2}}\leq \bigg(\bigg(\frac{n\mu}{2\mu-n}\bigg)c\bigg)^{\frac{n}{2p}}.   
\end{align}
Recalling the inequality $$(a+b)^\wp\leq a^\wp+b^\wp,~\forall~a,b\geq 0\quad\text{and}\quad\wp\in(0, 1],$$ we obtain from \eqref{eq5.6} that
\begin{align*}
   \limsup_{k\to\infty}\|u_k\|^{\frac{n}{n-2}}\leq \bigg(\bigg(\frac{n\mu}{2\mu-n}\bigg)c\bigg)^{\frac{2}{n-2}}+\bigg(\bigg(\frac{n\mu}{2\mu-n}\bigg)c\bigg)^{\frac{n}{(n-2)p}}.  
\end{align*}
Notice that $\frac{n\mu}{2\mu-n}>1$ and $\frac{2}{n-2}<\frac{n}{(n-2)p}$. Hence, it follows from the above inequality that
\begin{equation}\label{eq5.7}
\begin{aligned}
  \limsup_{k\to\infty}\|u_k\|^{\frac{n}{n-2}}&\leq \bigg(2\bigg(\frac{n\mu}{2\mu-n}\bigg)^{\frac{n}{2p}}\bigg)^{\frac{2}{n-2}} \Big(c^{\frac{2}{n-2}}+c^{\frac{n}{(n-2)p}}\Big)\\ & \leq
  \begin{cases}
      \bigg(2^{\frac{n}{2}}\Big(\frac{n\mu}{2\mu-n}\Big)^{\frac{n}{2p}} c\bigg)^{\frac{2}{n-2}}\quad\qquad\text{if}\quad c\in(0,1],\\
      \bigg(2^{\frac{n}{2}}\Big(\frac{n\mu}{2\mu-n}\Big)^{\frac{n}{2p}} c^{\frac{n}{2p}}\bigg)^{\frac{2}{n-2}} \quad\quad\text{if}\quad c\in[1,+\infty).
  \end{cases}
\end{aligned}  
\end{equation}
Based on \eqref{eq5.3} and \eqref{eq5.7}, we immediately conclude the proof.
\end{proof}
\begin{lemma}\label{lem5.6}
 Let the hypotheses \ref{g1}--\ref{g4} and \ref{g6} be satisfied. Let $\{u_k\}_k\subset E$ be an arbitrary $(PS)_c$ sequence for $J$ with $0<c<c_0$, where $c_0$ is defined as in Lemma \ref{lem5.4}. Then there exists a subsequence of $\{u_k\}_k$ still not relabelled and $u\in E$ such that $\Delta u_k\to \Delta u$ a.e. in $\mathbb{R}^n$ as $k\to\infty$. In addition, there holds
 \begin{align*}
     |\Delta u_k|^{t-2}\Delta u_k \rightharpoonup |\Delta u|^{t-2}\Delta u\quad\text{in}\quad L^{\frac{t}{t-1}}(\mathbb{R}^n)\quad\text{as}\quad k\to\infty\quad\text{for}\quad t\in\Big\{p,\frac{n}{2}\Big\}.
 \end{align*}
\end{lemma}
\begin{proof}
Let $0<c<c_0$ and $\{u_k\}_k\subset E$ be a $(PS)_c$ sequence for $J$. By using Lemma \ref{lem5.5}, $\{u_k\}_k$ is a bounded sequence in $E$ and \eqref{eq5.4} holds. Hence, we can assume without loss of generality, up to a subsequence, still not relabelled and $u\in E$ such that $u_k\rightharpoonup u$ in $E$ as $k\to\infty$. Now, by using Lemma \ref{lem2.5}, Theorem \ref{thm1.4} and Lemma \ref{lem2.7}, we have 
\begin{align}\label{eq5.10}
    \begin{cases}
        u_k\to u\quad\text{in}\quad L^\theta _{\textit{loc}}(\mathbb{R}^n)\quad\text{for all}\quad\theta\in [1,+\infty);\\
         u_k\to u\quad\text{in}\quad L^\varrho(\mathbb{R}^n,|x|^{-\gamma}\mathrm{d}x)\quad\text{for all}\quad\varrho\in [p^\ast,+\infty);\\
         \nabla u_k\to \nabla u \quad\text{in}\quad L^t _{\textit{loc}}(\mathbb{R}^n)\quad\text{for}\quad t\in\{p,\frac{n}{2}\};\\
         u_k\to u\quad\text{a.e.~~in}\quad \mathbb{R}^n\quad\text{as}\quad k\to\infty.
    \end{cases}
\end{align}
Pick $R>0$ and $\eta\in C^\infty_0(\mathbb{R}^n)$ such that $0\leq\eta\leq 1$ in $\mathbb{R}^n$, $\eta\equiv 1$ in $B_R$ and $\eta\equiv 0$ in $B^c_{2R}$. Moreover, using the fact that $u_k\rightharpoonup u$ in $E$ as $k\to\infty$ and $J^\prime(u_k)\to 0$ in $E^\ast$ as $k\to\infty$, one has
\begin{align}\label{eq5.11}
    \langle{J^\prime(u_k)-J^\prime(u),(u_k-u)\eta}\rangle=o_k(1)\quad\text{as}\quad k\to\infty.
\end{align}
Due to the convexity, we have
\begin{align*}
    \textit{P}^t_k(u_k,u):=(|\Delta u_k|^{t-2}\Delta u_k-|\Delta u|^{t-2}\Delta u)(\Delta u_k-\Delta u)\geq 0\quad\text{a.e. in}\quad \mathbb{R}^n
\end{align*}
for any $k\in\mathbb{N}$ and $t\in\{p,\frac{n}{2}\}$. In view of the well-known Simon's inequality with $n\geq 4$ (for instance, see \eqref{eq5.27}) and \eqref{eq5.11}, there exists $c_{\frac{n}{2}}>0$ such that as $k\to\infty$, we have
\begin{equation}
\begin{aligned} \label{eq5.12}		
c^{-1}_{\frac{n}{2}} \int_{B_R}& |\Delta u_k-\Delta u|^{\frac{n}{2}}\,\mathrm{d}x \\
&\leq    \int_{B_R} \textit{P}^{\frac{n}{2}}_k(u_k,u) \,\mathrm{d}x\leq \displaystyle\sum_{t\in\{p,\frac{n}{2}\}} \bigg(\int_{B_R}\textit{P}^t_k(u_k,u) \,\mathrm{d}x\bigg) \\
&\leq \displaystyle\sum_{t\in\{p,\frac{n}{2}\}} \bigg(\int_{\mathbb{R}^n}\textit{P}^t_k(u_k,u) \eta\,\mathrm{d}x\bigg)\\ 
&=o_k(1)+\int_{\mathbb{R}^n}\frac{(g(x,u_k)-g(x,u))(u_k-u)\eta}{|x|^\gamma}\,\mathrm{d}x  \\ 
& \quad -\sum_{t\in\{p,\frac{n}{2}\}} \Bigg(\int_{\mathbb{R}^n}\big(|\Delta u_k|^{t-2}\Delta u_k-|\Delta u|^{t-2}\Delta u)\Delta\eta ( u_k-u)\,\mathrm{d}x \Bigg)    \\ 
&\quad-2\sum_{t\in\{p,\frac{n}{2}\}} \Bigg(\int_{\mathbb{R}^n}\big(|\Delta u_k|^{t-2}\Delta u_k-|\Delta u|^{t-2}\Delta u)\nabla\eta \cdot\nabla( u_k-u)\,\mathrm{d}x\Bigg).
\end{aligned}
\end{equation}
By employing the Hölder's inequality and \eqref{eq5.10}, we obtain for $t\in\{p,\frac{n}{2}\}$ that
\begin{align*}
    \bigg|\int_{\mathbb{R}^n}\big(|\Delta u_k|^{t-2}\Delta u_k-&|\Delta u|^{t-2}\Delta u)\Delta\eta ( u_k-u)\,\mathrm{d}x \bigg|
		\\ & \leq \|\Delta\eta\|_\infty\big(\|\Delta u_k\|^{t-1}_t+\|\Delta u\|^{t-1}_t\big) \bigg(\int_{B_{2R}}|u_k-u|^t\,\mathrm{d}x\bigg)^{\frac{1}{t}}\\
        &\lesssim \bigg(\int_{B_{2R}}|u_k-u|^t\,\mathrm{d}x\bigg)^{\frac{1}{t}}\to 0\quad\text{as }k\to\infty.
\end{align*}
This yields for $t\in\{p,\frac{n}{2}\}$ that
\begin{align}\label{eq5.13}
  \lim_{k\to\infty} \int_{\mathbb{R}^n}\big(|\Delta u_k|^{t-2}\Delta u_k-|\Delta u|^{t-2}\Delta u)\Delta\eta ( u_k-u)\,\mathrm{d}x=0. 
\end{align}
Similarly, again by using the Hölder's inequality and \eqref{eq5.10}, we deduce for $t\in\{p,\frac{n}{2}\}$ that
\begin{align*}
    \bigg|\int_{\mathbb{R}^n}\big(|\Delta u_k|^{t-2}\Delta u_k-&|\Delta u|^{t-2}\Delta u)\nabla\eta \cdot\nabla( u_k-u)\,\mathrm{d}x\bigg|\\ &\leq  \|\nabla\eta\|_\infty\big(\|\Delta u_k\|^{t-1}_t+\|\Delta u\|^{t-1}_t\big)\bigg(\int_{B_{2R}}|\nabla( u_k-u)|^t\,\mathrm{d}x\bigg)^{\frac{1}{t}}\\
        & \lesssim \bigg(\int_{B_{2R}}|\nabla( u_k-u)|^t\,\mathrm{d}x\bigg)^{\frac{1}{t}}\to 0\quad\text{as }k\to\infty.
\end{align*}
It follows that for $t\in\{p,\frac{n}{2}\}$, we have
\begin{align}\label{eq5.14}
    \lim_{k\to\infty} \int_{\mathbb{R}^n}\big(|\Delta u_k|^{t-2}\Delta u_k-|\Delta u|^{t-2}\Delta u)\nabla\eta \cdot\nabla( u_k-u)\,\mathrm{d}x=0.
\end{align}
Define
\begin{align*}
    J_1:=\zeta\int_{B_{2R}}\frac{(|u_k|^{\tau-1}+|u|^{\tau-1})|u
    _k-u|}{|x|^\gamma}\,\mathrm{d}x
\end{align*}
and \begin{align*}
    J_2:=D_\zeta\int_{B_{2R}}\frac{\big(|u_k|^{q-1} \Phi_{\alpha,j_0}(u_k)+|u|^{q-1} \Phi_{\alpha,j_0}(u)\big)|u
    _k-u|}{|x|^\gamma}\,\mathrm{d}x.
\end{align*}
Let $1<s<\frac{n}{\gamma}$ and $r\geq p^\ast$ be such that $\frac{1}{\upsilon}+\frac{1}{\upsilon^\prime}=1$ for $\upsilon\in\{s,r\}$. Then by using the Hölder's inequality and \eqref{eq5.10}, we get
\begin{align*}
J_1&\leq \zeta\Bigg[\bigg(\int_{B_{2R}}\frac{|u_k|^{(\tau-1)r^\prime}}{|x|^\gamma}\,\mathrm{d}x\bigg)^{\frac{1}{r^\prime}}+\bigg(\int_{B_{2R}}\frac{|u|^{(\tau-1)r^\prime}}{|x|^\gamma}\,\mathrm{d}x\bigg)^{\frac{1}{r^\prime}}\Bigg]\bigg(\int_{B_{2R}}\frac{|u
_k-u|^r}{|x|^\gamma}\,\mathrm{d}x\bigg)^{\frac{1}{r}}\\
&\leq \zeta\Big(\|u_k\|^{\tau-1}_{L^{(\tau-1)r^\prime s^\prime}(B_{2R})}+\|u\|^{\tau-1}_{L^{(\tau-1)r^\prime s^\prime}(B_{2R})}\Big)\bigg(\int_{B_{2R}}\frac{1}{|x|^{\gamma s}}\,\mathrm{d}x\bigg)^{\frac{1}{r^\prime s}}\bigg(\int_{B_{2R}}\frac{|u
_k-u|^r}{|x|^\gamma}\,\mathrm{d}x\bigg)^{\frac{1}{r}}\\
& \lesssim \bigg(\int_{\mathbb{R}^n}\frac{|u
_k-u|^r}{|x|^\gamma}\,\mathrm{d}x\bigg)^{\frac{1}{r}}\to 0\quad\text{as}\quad k\to\infty.
\end{align*}
Due to \eqref{eq5.4}, passing if necessary to a subsequence, still denoted by itself, we can assume
\begin{align*}
    \sup_{k\in\mathbb{N}}\|u_k\|^{\frac{n}{n-2}}<\frac{\beta_{\gamma,n}}{\alpha_0}.
\end{align*}
Let $\eta,~\eta^\prime>1$ be such that $\frac{1}{\eta}+\frac{1}{\eta^\prime}=1$. From the above inequality, we first fix $m\in \big(\|u_k\|^{\frac{n}{n-2}},\frac{\beta_{\gamma,n}}{\alpha_0}\big)$, $\alpha>\alpha_0$ close to $\alpha_0$ and $\eta>1$ close to $1$ such that  $\alpha \eta m<\beta_{\gamma,n}$. It follows from the Hölder's inequality, Theorem \ref{thm1.1}, Corollary \ref{cor2.6},  Corollary \ref{cor3.1} and \eqref{eq5.10} that
\begin{align*}
J_2&\leq D_\zeta \Bigg[\bigg(\int_{\mathbb{R}^n}\frac{|u_k|^{(q-1)\eta^\prime}|u_k-u|^{\eta^\prime}}{|x|^{\gamma}}\,\mathrm{d}x\bigg)^{\frac{1}{\eta^\prime}}\bigg(\int_{\mathbb{R}^n}\frac{\Phi_{\alpha \eta,j_0}(u_k)}{|x|^\gamma}\,\mathrm{d}x\bigg)^{\frac{1}{\eta}}\\
&\quad\quad+\bigg(\int_{\mathbb{R}^n}\frac{|u|^{(q-1)\eta^\prime}|u_k-u|^{\eta^\prime}}{|x|^{\gamma}}\,\mathrm{d}x\bigg)^{\frac{1}{\eta^\prime}}\bigg(\int_{\mathbb{R}^n}\frac{\Phi_{\alpha \eta,j_0}(u)}{|x|^\gamma}\,\mathrm{d}x\bigg)^{\frac{1}{\eta}}\Bigg]\\
&\lesssim \Bigg[\bigg(\int_{\mathbb{R}^n}\frac{|u_k|^{q \eta^\prime}}{|x|^{\gamma}}\,\mathrm{d}x\bigg)^{\frac{q-1}{q \eta^\prime}} \bigg(\int_{\mathbb{R}^n}\frac{\Phi_{\beta_{\gamma,n},j_0}(u_k/\|u_k\|)}{|x|^\gamma}\,\mathrm{d}x\bigg)^{\frac{1}{\eta}}\\
&\quad\quad+\bigg(\int_{\mathbb{R}^n}\frac{|u|^{q \eta^\prime}}{|x|^{\gamma}}\,\mathrm{d}x\bigg)^{\frac{q-1}{q \eta^\prime}} \bigg(\int_{\mathbb{R}^n}\frac{\Phi_{\alpha \eta,j_0}(u)}{|x|^\gamma}\,\mathrm{d}x\bigg)^{\frac{1}{\eta}}\Bigg]\bigg(\int_{\mathbb{R}^n}\frac{|u_k-u|^{q \eta^\prime}}{|x|^{\gamma}}\,\mathrm{d}x\bigg)^{\frac{1}{q \eta^\prime}}\\
&\lesssim \bigg(\int_{\mathbb{R}^n}\frac{|u_k-u|^{q \eta^\prime}}{|x|^{\gamma}}\,\mathrm{d}x\bigg)^{\frac{1}{q \eta^\prime}}\to 0\quad\text{as}\quad k\to\infty.
\end{align*}
In addition, by using \eqref{eq1.1}, one sees that
\begin{align*}
 \bigg|\int_{\mathbb{R}^n}\cfrac{(g(x,u_k)-g(x,u))(u_k-u)\eta}{|x|^\gamma}\, \mathrm{d}x\bigg|&\leq \int_{B_{2R}}\cfrac{(|g(x,u_k)|+|g(x,u)|)|u_n-u|}{|x|^\gamma}\, \mathrm{d}x \\
 &\leq J_1+J_2\to 0\quad\text{as}\quad k\to\infty.   
\end{align*}
This shows that
\begin{equation}\label{eq5.15}
  \lim_{k\to\infty} \int_{\mathbb{R}^n}\cfrac{(g(x,u_k)-g(x,u))(u_k-u)\eta}{|x|^\gamma}\, \mathrm{d}x=0.  
\end{equation}
Letting $k\to\infty$ in \eqref{eq5.12} and using \eqref{eq5.13}, \eqref{eq5.14} and \eqref{eq5.15}, we deduce that $\Delta u_k\to \Delta u$ in $L^{\frac{n}{2}}(B_{R})$ as $k\to\infty$ for all $R>0$. Hence, we can assume that up to a subsequence still denoted by itself such that  $\Delta u_k\to \Delta u$ a.e. in $\mathbb{R}^n$ as $k\to\infty$. It follows that $|\Delta u_k|^{t-2}\Delta u_k\to |\Delta u|^{t-2}\Delta u$ a.e. in $\mathbb{R}^n$ as $k\to\infty$ for $t\in\{p,\frac{n}{2}\}$. Consequently, by using the fact that $\{|\Delta u_k|^{t-2}\Delta u_k\}_k$ is bounded in $L^{\frac{t}{t-1}}(\mathbb{R}^n)$, we conclude the proof.
\end{proof}
\begin{proposition}\label{prop5.7}
    Let the hypotheses \ref{g1}--\ref{g6} be satisfied. Then the functional $J$ satisfies $(PS)_c$ condition for any $0<c<c_0$, where $c_0$ is defined as in Lemma \ref{lem5.4}.
\end{proposition}
\begin{proof}
Let $\{u_k\}_k\subset E$ be a $(PS)_c$ sequence for $J$ with $0<c<c_0$. It follows that
\begin{align}\label{eq5.16}
    \frac{1}{p}\|\Delta u_k\|^p_p+\frac{2}{n}\|\Delta u_k\|^{\frac{n}{2}}_{\frac{n}{2}}-\int_{\mathbb{R}^n}\frac{G(x,u_k)}{|x|^\gamma}\,\mathrm{d}x\to c\quad\text{as}\quad k\to\infty
\end{align}
and 
\begin{align}\label{eq5.17}
  | \langle{J^\prime(u_k),\psi}\rangle |\leq \zeta_k\|\psi\|,~\forall~ \psi\in E,
\end{align}
where $\zeta_k\to 0$ as $k\to\infty$. Pick $\psi=u_k$ in \eqref{eq5.17}, then one has
\begin{align}\label{eq5.18}
 \int_{\mathbb{R}^n}\frac{g(x,u_k)u_k}{|x|^\gamma}\,\mathrm{d}x-\|\Delta u_k\|^p_p-\|\Delta u_k\|^{\frac{n}{2}}_{\frac{n}{2}}\leq \zeta_k\|\ u_k\|.  
\end{align}
By using Lemma \ref{lem5.5}, we infer that $\{u_k\}_k$ is a bounded sequence in $E$ and \eqref{eq5.4} holds. Therefore, one can deduce from \eqref{eq5.16} and \eqref{eq5.18} that
\begin{align}\label{eq5.19}
\int_{\mathbb{R}^n}\frac{|G(x,u_k)|}{|x|^\gamma}\,\mathrm{d}x\leq C\quad\text{and}\quad \int_{\mathbb{R}^n}\frac{|g(x,u_k)u_k|}{|x|^\gamma}\,\mathrm{d}x \leq C  
\end{align}
for some suitable constant $C>0$. Notice that \eqref{eq5.10} also holds. Due to \eqref{eq1.1}, Corollary \ref{cor3.1}, Lemma \ref{lem2.5} and the Hölder's inequality, one has $\frac{g(x,u)}{|x|^\gamma}\in L^1_{\textit{loc}}(\mathbb{R}^n)$. Now, by employing a similar strategy developed in de Souza-do Ó \cite{de Souza-do O-2011}, we have
\begin{align}\label{eq5.20}
  \frac{g(x,u_k)}{|x|^\gamma}\to \frac{g(x,u)}{|x|^\gamma}\quad\text{in}\quad L^1_{\textit{loc}}(\mathbb{R}^n)\quad\text{as}\quad k\to\infty. 
\end{align}
Next, we claim that
\begin{align}\label{eq5.21}
  \frac{G(x,u_k)}{|x|^\gamma}\to \frac{G(x,u)}{|x|^\gamma}\quad\text{in}\quad L^1(\mathbb{R}^n)\quad\text{as}\quad k\to\infty. 
\end{align}
To prove this, take $R>0$ and let us write
\begin{align*}
\int_{\mathbb{R}^n}\frac{|G(x,u_k)-G(x,u)|}{|x|^\gamma}\,\mathrm{d}x=\int_{B_R}\frac{|G(x,u_k)-G(x,u)|}{|x|^\gamma}\,\mathrm{d}x+\int_{B_R^c}\frac{|G(x,u_k)-G(x,u)|}{|x|^\gamma}\,\mathrm{d}x.    
\end{align*}
By using \ref{g4} and \ref{g5}, we can find a constant $K_0>0$ such that
\begin{align}\label{eq5.22}
    \frac{G(x,u_k)}{|x|^\gamma}\leq K_0 \frac{|g(x,u_k)|}{|x|^\gamma},~\forall~ x\in \mathbb{R}^n. 
\end{align}
Hence, we obtain from \eqref{eq5.10}, \eqref{eq5.20}, \eqref{eq5.22} and the generalized Lebesgue dominated convergence theorem that
\begin{align}\label{eq5.23}
    \lim_{R\to\infty}\lim_{k\to\infty} \int_{B_R}\frac{|G(x,u_k)-G(x,u)|}{|x|^\gamma}\,\mathrm{d}x=0.
\end{align}
Join us in writing
\begin{align*}
 \int_{B_R^c}\frac{|G(x,u_k)-G(x,u)|}{|x|^\gamma}\,\mathrm{d}x&=\int_{B_R^c\cap\{x\in\mathbb{R}^n\colon~|u_k|> K\}}\frac{|G(x,u_k)-G(x,u)|}{|x|^\gamma}\,\mathrm{d}x\\
 &\quad\quad+\int_{B_R^c\cap\{x\in\mathbb{R}^n\colon~|u_k|\leq K\}}\frac{|G(x,u_k)-G(x,u)|}{|x|^\gamma}\,\mathrm{d}x.   
\end{align*}
In view of \eqref{eq5.19} and \eqref{eq5.22}, we have the following estimates
\begin{align*}
 \int_{B_R^c\cap\{x\in\mathbb{R}^n\colon~|u_k|> K\}}\frac{|G(x,u_k)-G(x,u)|}{|x|^\gamma}\,\mathrm{d}x &\leq \frac{K_0}{K}\int_{B_R^c\cap\{x\in\mathbb{R}^n\colon~|u_k|> K\}}\frac{|g(x,u_k)u_k|}{|x|^\gamma}\,\mathrm{d}x+\int_{B_R^c}\frac{G(x,u)}{|x|^\gamma}\,\mathrm{d}x\\
 &\leq \frac{CK_0}{K}+\int_{B_R^c}\frac{G(x,u)}{|x|^\gamma}\,\mathrm{d}x,
\end{align*}
 which together with \eqref{eq1.2}, Corollary \ref{cor3.1}, Theorem \ref{thm1.4} and the Lebesgue dominated convergence theorem yields
\begin{align*}
 \lim_{K\to\infty} \lim_{R\to\infty}\lim_{k\to\infty}  \int_{B_R^c\cap\{x\in\mathbb{R}^n\colon~|u_k|> K\}}\frac{|G(x,u_k)-G(x,u)|}{|x|^\gamma}\,\mathrm{d}x=0. 
\end{align*}
Due to \ref{g2}, we can find a constant $C_1>0$ such that $|G(x,s)|\leq C_1|s|^\tau$ for all $|s|\leq K$. It follows from Theorem \ref{thm1.4} that
\begin{align*}
    \int_{B_R^c\cap\{x\in\mathbb{R}^n\colon~|u_k|\leq K\}}\frac{|G(x,u_k)-G(x,u)|}{|x|^\gamma}\,\mathrm{d}x&\leq \frac{C_1}{R^{\frac{\gamma}{2}}}\int_{B_R^c\cap\{x\in\mathbb{R}^n\colon~|u_k|\leq K\}}\frac{|u_k|^\tau}{|x|^{\frac{\gamma}{2}}}\,\mathrm{d}x+\int_{B_R^c}\frac{G(x,u)}{|x|^\gamma}\,\mathrm{d}x\\
    & \leq \frac{C_2}{R^{\frac{\gamma}{2}}}\sup_{k\in\mathbb{N}}\|u_k\|^\tau+\int_{B_R^c}\frac{G(x,u)}{|x|^\gamma}\,\mathrm{d}x,
\end{align*}
where $C_2>0$ is a suitable constant independent on $k$. The above inequality with \eqref{eq1.2}, Corollary \ref{cor3.1}, Theorem \ref{thm1.4} and the Lebesgue dominated convergence theorem implies
\begin{align*}
 \lim_{K\to\infty} \lim_{R\to\infty}\lim_{k\to\infty}  \int_{B_R^c\cap\{x\in\mathbb{R}^n\colon~|u_k|\leq  K\}}\frac{|G(x,u_k)-G(x,u)|}{|x|^\gamma}\,\mathrm{d}x=0. 
\end{align*}
In conclusion, we have
\begin{align}\label{eq5.24}
    \lim_{R\to\infty}\lim_{k\to\infty} \int_{B^c_R}\frac{|G(x,u_k)-G(x,u)|}{|x|^\gamma}\,\mathrm{d}x=0.
\end{align}
In virtue of \eqref{eq5.23} and \eqref{eq5.24}, we deduce that \eqref{eq5.21} holds. Hence, we obtain from \eqref{eq5.17}, \eqref{eq5.20}, Lemma \ref{lem5.6} and the Lebesgue dominated convergence theorem that
\begin{align*}
    \int_{\mathbb{R}^n} |\Delta u|^{p-2}\Delta u\Delta \psi\,\mathrm{d}x+\int_{\mathbb{R}^n} |\Delta u|^{\frac{n}{2}-2}\Delta u\Delta \psi~\mathrm{d}x=\int_{\mathbb{R}^n}\frac{g(x,u)\psi}{|x|^\gamma}\,\mathrm{d}x,~\forall~ \psi\in C^\infty_0(\mathbb{R}^n).
\end{align*}
By exploiting the density of $C^\infty_0(\mathbb{R}^n)$ in $E$, it follows that $u$ is a weak solution of the problem \eqref{main problem}. Now, to complete the proof, it is sufficient to show the compactness of the $(PS)_c$ sequence $\{u_k\}_k$ in $E$. Invoking the fact that $c>0$, we have two possibilities as discussed below.

\textbf{Case-1:}$(c\neq 0,~u=0)$ In this situation, from \eqref{eq5.21}, one has
\begin{align*}
     \lim_{k\to\infty}\int_{\mathbb{R}^n}\frac{G(x,u_k)}{|x|^\gamma}\,\mathrm{d}x=0.
\end{align*}
This together with $J(u_k)\to c$ as $k\to\infty$ yields
\begin{align*}
    c+o_k(1)&\notag=\frac{1}{p}\|\Delta u_k\|^p_p+\frac{2}{n}\|\Delta u_k\|^{\frac{n}{2}}_{\frac{n}{2}}\geq\frac{2}{n}\big(\|\Delta u_k\|^p_p+\|\Delta u_k\|^{\frac{n}{2}}_{\frac{n}{2}}\big) \\
    &\geq \bigg(\frac{2}{n}-\frac{1}{\mu}\bigg)\Big(\|\Delta u_k\|_p^p+\|\Delta u_k\|_{\frac{n}{2}}^{\frac{n}{2}}\Big)\quad\text{as}\quad k\to\infty.
\end{align*}
The above inequality together with $c\in(0,c_0)$ implies the validity of \eqref{eq5.4}, we thank to Lemma \ref{lem5.5}. Hence, without loss of generality up to a subsequence, still not relabelled, we can assume that
\begin{align*}
    \sup_{k\in\mathbb{N}}\|u_k\|^{\frac{n}{n-2}}<\frac{\beta_{\gamma,n}}{\alpha_0}.
\end{align*}
Let us fix $m\in(\|u_k\|^{\frac{n}{n-2}},\frac{\beta_{\gamma,n}}{\alpha_0})$ and choose $r$, $r^\prime>1$ such that $\frac{1}{r}+\frac{1}{r^\prime}=1$. Pick $\alpha>\alpha_0$ close to $\alpha_0$, $r>1$ close to $1$ such that $\alpha r m< \beta_{\gamma,n}$. Therefore, by using \eqref{eq1.1}, Theorem \ref{thm1.1}, Corollary \ref{cor2.6}, Theorem \ref{thm1.4} and the Hölder's inequality, we have
\begin{align*}
    \bigg|\int_{\mathbb{R}^n}\frac{g(x,u_k)u_k}{|x|^\gamma}\,\mathrm{d}x\bigg|&\leq\zeta \int_{\mathbb{R}^n}\frac{|u_k|^\tau}{|x|^{\gamma}}\,\mathrm{d}x+D_\zeta \bigg(\int_{\mathbb{R}^n}\frac{|u_k|^{qr^\prime}}{|x|^{\gamma}}\,\mathrm{d}x\bigg)^{\frac{1}{r^\prime}}\bigg(\int_{\mathbb{R}^n}\frac{\Phi_{\alpha r,j_0}(u_k)}{|x|^\gamma}\,\mathrm{d}x\bigg)^{\frac{1}{r}}\\
    & \lesssim \int_{\mathbb{R}^n}\frac{|u_k|^\tau}{|x|^{\gamma}}\,\mathrm{d}x+ \bigg(\int_{\mathbb{R}^n}\frac{|u_k|^{qr^\prime}}{|x|^{\gamma}}\,\mathrm{d}x\bigg)^{\frac{1}{r^\prime}}\bigg(\int_{\mathbb{R}^n}\frac{\Phi_{\beta_{\gamma,n},j_0}(u_k/\|u_k\|)}{|x|^\gamma}\,\mathrm{d}x\bigg)^{\frac{1}{r}}\\
    & \lesssim \int_{\mathbb{R}^n}\frac{|u_k-u|^\tau}{|x|^{\gamma}}\,\mathrm{d}x+ \bigg(\int_{\mathbb{R}^n}\frac{|u_k-u|^{qr^\prime}}{|x|^{\gamma}}\,\mathrm{d}x\bigg)^{\frac{1}{r^\prime}}\to 0\quad\text{as}\quad k\to\infty.
\end{align*}
It follows that
\begin{align}\label{eq5.25}
    \lim_{k\to\infty} \int_{\mathbb{R}^n}\frac{g(x,u_k)u_k}{|x|^\gamma}\,\mathrm{d}x=0.
\end{align}
Hence, we deduce from $\langle{J^\prime(u_k),u_k}\rangle\to 0$ as $k\to\infty$ and \eqref{eq5.25} that $u_k\to 0$ in $E$ as $k\to\infty$. Consequently, since $J\in C^1(E,\mathbb{R})$, one can easily obtain  $c=0$, which is a contradiction. Thus, this situation is impossible.

\textbf{Case-2:}$(c\neq 0,~u\neq 0)$
In view of $u_k\rightharpoonup u$ in $E$ as $k\to\infty$, we obtain from the lower semicontinuity of the norm that, up to a subsequence, still denoted by itself, such that $\|u\|\leq\lim_{k\to\infty}\|u_k\|$. Indeed, if $\|u\|=\lim_{k\to\infty}\|u_k\|$, then by using the result of Autuori-Pucci \cite[Corollary A.2]{Autuori-Pucci-2013}, one has $u_k\to u$ in $E$ as $k\to\infty$. It follows that the sequence $\{u_k\}_k$ satisfies the $(PS)_c$ compactness condition and thus $u$ is a nontrivial mountain pass solution of the problem  \eqref{main problem}, we thank to $J\in C^1(E,\mathbb{R})$.

However, to complete the proof, we simply need to show $\|u\|<\lim_{k\to\infty}\|u_k\|$ is not possible. Assume by contradiction that $\|u\|<\lim_{k\to\infty}\|u_k\|$ holds true. Define the sequence $\{\psi_k\}_k$ and $\psi_0$ as follows
\begin{align*}
    \psi_k:=\frac{u_k}{\|u_k\|}\quad\text{and}\quad\psi_0:=\frac{u}{\lim_{k\to\infty}\|u_k\|}.
\end{align*}
In view of \eqref{eq5.10}, it is standard to check that $\|\psi_k\|=1$, $\psi_0\neq 0$ and $\psi_k\rightharpoonup \psi_0$ in $E$ as $k\to\infty$. Moreover, if $\|\psi_0\|=1$, then again by using \cite[Corollary A.2]{Autuori-Pucci-2013}, we have $u_k\to u$ in $E$ as $k\to\infty$. This contradicts $\|u\|<\lim_{k\to\infty}\|u_k\|$ and thus, the proof is finished. 

Next, we consider the case $\|\psi_0\|<1$. By direct calculations and arguing similarly as in the proof of Lemma \ref{lem5.5}, up to a subsequence still denoted by itself, one sees that
\begin{align*}
\lim_{k\to\infty}\|u_k\|^{\frac{n}{n-2}}<\frac{\beta_{\gamma,n}}{\alpha_0}\big(1-\|\psi_0\|^{\frac{n}{2}}\big)^{-\frac{2}{n-2}} .   
\end{align*}
Let $r$, $r^\prime>1$ be such that $\frac{1}{r}+\frac{1}{r^\prime}=1$. From the above inequality, we infer that there exists $m>0$, $\alpha>\alpha_0$ close to $\alpha_0$, $r>1$ close to $1$ such that
\begin{align*}
    \alpha r \|u_k\|^{\frac{n}{n-2}}<m<\beta_{\gamma,n}\big(1-\|\psi_0\|^{\frac{n}{2}}\big)^{-\frac{2}{n-2}} \quad\text{for large}\quad k.
\end{align*}
By using \eqref{eq1.1}, the H\"older's inequality, Theorem \ref{thm1.2}, Theorem \ref{thm1.4}, Corollary \ref{cor2.6}, we have for large $k$ that
\begin{align*}
    \bigg|\int_{\mathbb{R}^n}\frac{g(x,u_k)(u_k-u)}{|x|^\gamma}\,\mathrm{d}x\bigg|&\leq \zeta \bigg(\int_{\mathbb{R}^n}\frac{|u_k|^{\tau}}{|x|^{\gamma}}\,\mathrm{d}x\bigg)^{\frac{\tau-1}{\tau}}\bigg(\int_{\mathbb{R}^n}\frac{|u_k-u|^{\tau}}{|x|^{\gamma}}\,\mathrm{d}x\bigg)^{\frac{1}{\tau}}\\
    &\quad+D_\zeta \bigg(\int_{\mathbb{R}^n}\frac{|u_k|^{(q-1)r^\prime}|u_k-u|^{r^\prime}}{|x|^{\gamma}}\,\mathrm{d}x\bigg)^{\frac{1}{r^\prime}}\bigg(\int_{\mathbb{R}^n}\frac{\Phi_{\alpha r,j_0}(u_k)}{|x|^\gamma}\,\mathrm{d}x\bigg)^{\frac{1}{r}}\\
    & \lesssim \bigg(\int_{\mathbb{R}^n}\frac{|u_k|^{\tau}}{|x|^{\gamma}}\,\mathrm{d}x\bigg)^{\frac{\tau-1}{\tau}}\bigg(\int_{\mathbb{R}^n}\frac{|u_k-u|^{\tau}}{|x|^{\gamma}}\,\mathrm{d}x\bigg)^{\frac{1}{\tau}}+\bigg(\int_{\mathbb{R}^n}\frac{|u_k|^{q r^\prime}}{|x|^{\gamma}}\,\mathrm{d}x\bigg)^{\frac{q-1}{q r^\prime}}\\
    &\quad\times\bigg(\int_{\mathbb{R}^n}\frac{|u_k-u|^{q r^\prime}}{|x|^{\gamma}}\,\mathrm{d}x\bigg)^{\frac{1}{q r^\prime}} \bigg(\int_{\mathbb{R}^n}\frac{\Phi_{m,j_0}(\psi_k)}{|x|^\gamma}\,\mathrm{d}x\bigg)^{\frac{1}{r}} \\
    &\lesssim \bigg(\int_{\mathbb{R}^n}\frac{|u_k-u|^{\tau}}{|x|^{\gamma}}\,\mathrm{d}x\bigg)^{\frac{1}{\tau}}+\bigg(\int_{\mathbb{R}^n}\frac{|u_k-u|^{q r^\prime}}{|x|^{\gamma}}\,\mathrm{d}x\bigg)^{\frac{1}{q r^\prime}}=o_k(1)\quad\text{as}\quad k\to\infty.
\end{align*}
It follows that
\begin{align*}
\lim_{k\to\infty}\int_{\mathbb{R}^n}\frac{g(x,u_k)(u_k-u)}{|x|^\gamma}\,\mathrm{d}x=0.   
\end{align*}
The above convergence together with $\langle{J^\prime(u_k),u_k-u}\rangle=o_k(1)$ as $k\to\infty$ yields
\begin{align*}
    \int_{\mathbb{R}^n} |\Delta u_k|^{p-2}\Delta u_k\Delta (u_k-u)\,\mathrm{d}x+\int_{\mathbb{R}^n} |\Delta u_k|^{\frac{n}{2}-2}\Delta u_k\Delta (u_k-u)\,\mathrm{d}x=o_k(1)\quad\text{as}\quad k\to\infty.
\end{align*}
In addition, we obtain from $u_k\rightharpoonup u$ in $E$ as $k\to\infty$ that
\begin{align*}
    \int_{\mathbb{R}^n} |\Delta u|^{p-2}\Delta u\Delta (u_k-u)\,\mathrm{d}x+\int_{\mathbb{R}^n} |\Delta u|^{\frac{n}{2}-2}\Delta u\Delta (u_k-u)\,\mathrm{d}x=o_k(1)\quad\text{as}\quad k\to\infty.
\end{align*}
Due to convexity of the map $s\mapsto \frac{1}{\eta}|s|^\eta$ for all $\eta\in[1,+\infty)$, the above convergences imply that
\begin{align}\label{eq5.26}
   \int_{\mathbb{R}^n} \big(|\Delta u_k|^{t-2}\Delta u_k-|\Delta u|^{t-2}\Delta u\big)\Delta (u_k-u)\,\mathrm{d}x=o_k(1) \quad\text{as}\quad k\to\infty 
\end{align}
for $t\in\{p,\frac{n}{2}\}$. Recall the well-known Simon's inequality as follows
\begin{equation}\label{eq5.27}
|x-y|^\eta\leq\begin{cases}
C_\eta\big[\big(|x|^{\eta-2}x-|y|^{\eta-2}y\big)(x-y)\big]^\frac{\eta}{2}\big[|x|^\eta+|y|^\eta\big]^\frac{2-\eta}{2}\quad\text{if}~1<\eta<2,\\
c_\eta\big(|x|^{\eta-2}x-|y|^{\eta-2}y\big)(x-y)\quad\quad\quad\quad\quad\quad\quad\quad\quad\quad\text{if}~\eta\geq 2,
\end{cases}  
\end{equation}
for all $x,y\in \mathbb{R}^d$, where $d\in[1,+\infty)$, $C_\eta$ and $c_\eta$ are positive constants depending only on $\eta$. From  \eqref{eq5.26} and \eqref{eq5.27}, we obtain for all $t\in\{p,\frac{n}{2}\}\geq 2$ that 
\begin{align*} \|\Delta u_k-\Delta u\|^t_t &\lesssim \int_{\mathbb{R}^n} \big(|\Delta u_k|^{t-2}\Delta u_k-|\Delta u|^{t-2}\Delta u\big)\Delta (u_k-u)\,\mathrm{d}x=o_k(1) \quad\text{as}\quad k\to\infty. \end{align*}
Similarly, from \eqref{eq5.26}, \eqref{eq5.27} and the  H\"older's inequality, one can deduce for $t=p\in(1,2)$ that
\begin{align*} \|\Delta u_k-\Delta u\|^t_t &\lesssim\int_{\mathbb{R}^n}\Big(\big(|\Delta u_k|^{t-2}\Delta u_k-|\Delta u|^{t-2}\Delta u\big)\Delta (u_k-u)\Big)^{\frac{t}{2}}\Big(|\Delta u_k|^{t}+|\Delta u|^{t}\Big)^{\frac{2-t}{2}}\,\mathrm{d}x\\
&\lesssim\int_{\mathbb{R}^n}\Big(\big(|\Delta u_k|^{t-2}\Delta u_k-|\Delta u|^{t-2}\Delta u\big)\Delta (u_k-u)\Big)^{\frac{t}{2}}\Big(\big(|\Delta u_k|^t\big)^{^{\frac{2-t}{2}}}+\big(|\Delta u|^t\big)^{^{\frac{2-t}{2}}}\Big)\,\mathrm{d}x\\
&\lesssim \Bigg(  \int_{\mathbb{R}^n} \big(|\Delta u_k|^{t-2}\Delta u_k-|\Delta u|^{t-2}\Delta u\big)\Delta (u_k-u)\,\mathrm{d}x\Bigg)^{\frac{t}{2}} \Big(\|\Delta u_k\|_t^{\frac{(2-t)t}{2}}+\|\Delta u\|_t^{\frac{(2-t)t}{2}}\Big)\\
&\lesssim \Bigg(  \int_{\mathbb{R}^n} \big(|\Delta u_k|^{t-2}\Delta u_k-|\Delta u|^{t-2}\Delta u\big)\Delta (u_k-u)\,\mathrm{d}x\Bigg)^{\frac{t}{2}} =o_k(1)\quad\text{as}\quad k\to\infty.
\end{align*}
Hence, we deduce that  $\Delta u_k\to \Delta u$ in $L^t(\mathbb{R}^n)$ as $k\to\infty$ for any $ t\in\{p,\frac{n}{2}\}$ and thus, we obtain $u_k\to u$ in $E$ as $k\to\infty$. Moreover, it contradicts $\|u\|<\lim_{k\to\infty}\|u_k\|$ and thus, the result follows.
\end{proof}
\begin{proof}[\bf{Proof of Theorem \ref{thm1.5}}.]
Let $\lambda_0$ be defined as in Lemma \ref{lem5.4} and the hypothesis \ref{g6} holds for all $\lambda\geq\lambda_0$. Now, by employing Proposition \ref{prop5.7} and Theorem \ref{thm5.2}, one can see that the functional $J$ has a nontrivial critical point $u\in E$, which is a weak solution to the problem \eqref{main problem}. This completes the proof.
\end{proof}
\section*{Acknowledgements}
DKM expresses his deep appreciation for the DST INSPIRE PhD Fellowship, with reference number DST/INSPIRE/03/2019/000265, supported by the Government of India. TM expresses gratitude for the assistance received from the CSIR-HRDG grant, sanction number 25/0324/23/EMR-II. AS received funding from the DST-INSPIRE Grant DST/INSPIRE/04/2018/002208, funded by the Government of India.

\end{document}